\documentclass[11pt]{article}

\usepackage[margin=1in]{geometry}
\usepackage[utf8]{inputenc}
\usepackage[T1]{fontenc}
\usepackage{lmodern}
\usepackage{amsmath,amssymb,amsthm,mathtools}
\usepackage{mathrsfs}
\usepackage{enumitem}
\usepackage{booktabs}
\usepackage{hyperref}
\usepackage{xcolor}
\usepackage{microtype}

\hypersetup{
    colorlinks=true,
    linkcolor=blue!50!black,
    citecolor=blue!50!black,
    urlcolor=blue!50!black
}

\numberwithin{equation}{section}

\newtheorem{theorem}{Theorem}[section]
\newtheorem{proposition}[theorem]{Proposition}
\newtheorem{lemma}[theorem]{Lemma}
\newtheorem{corollary}[theorem]{Corollary}
\newtheorem{definition}[theorem]{Definition}
\newtheorem{assumption}[theorem]{Assumption}
\newtheorem{hypothesis}[theorem]{Hypothesis}

\newtheorem{problem}[theorem]{Problem}

\newtheorem{remark}[theorem]{Remark}
\newtheorem{warning}[theorem]{Warning}

\newcommand{\R}{\mathbb R}
\newcommand{\Z}{\mathbb Z}
\newcommand{\eps}{\varepsilon}
\newcommand{\Supp}{\mathsf{Sup}}
\newcommand{\Tax}{\mathsf{Tax}}
\newcommand{\Leak}{\mathsf{Leak}}
\newcommand{\Err}{\mathsf{Err}}
\newcommand{\Mres}{\mathfrak M}
\newcommand{\Gdec}{\mathfrak G}
\newcommand{\Bbad}{\mathsf B}
\newcommand{\pos}[1]{\left(#1\right)_+}

\newcommand{\dd}{\,d}
\newcommand{\N}{\mathcal N}
\newcommand{\NS}{\mathrm{NS}}
\newcommand{\crit}{\mathrm{crit}}

\newcommand{\calC}{\mathcal C}
\newcommand{\calD}{\mathcal D}

\newcommand{\calY}{\mathcal Y}
\newcommand{\calX}{\mathcal X}
\newcommand{\calE}{\mathcal E}
\newcommand{\calF}{\mathcal F}
\newcommand{\calL}{\mathcal L}
\newcommand{\calR}{\mathcal R}
\newcommand{\Prof}{\mathsf{Prof}}
\newcommand{\Rep}{\mathsf{Rep}}

\newcommand{\Image}{\operatorname{Im}}
\newcommand{\Dist}{\operatorname{dist}}
\newcommand{\limone}{\varprojlim{}^{1}}

\title{Critical Ledgers and Scale-Defect Cascades for Navier--Stokes}
\author{Runlong Yu\\
	The University of Alabama, Tuscaloosa, AL, USA\\
	\texttt{ryu5@ua.edu}}
\date{}

\begin{document}
\maketitle

\begin{abstract}
We prove a finite-scale supply--tax reduction for suitable weak solutions of
the three-dimensional incompressible Navier--Stokes equations near the local
regularity threshold.  Along an admissible chain of parabolic windows, let
\[
    \Bbad_k=A_k+C_k+D_k
\]
denote the scale-critical reservoir badness.  The main PDE theorem says that
if \(\Bbad_k\) stays outside a decay basin for many consecutive scales, then
either cutoff/window leakage accumulates or the solution generates a positive
cumulative amount of untaxed critical supply.  The supply terms come from
nonlinear flux, pressure transport, interpolation amplification, and pressure
regeneration; the tax terms come from viscous dissipation, expected decay of
old reservoir coordinates, and harmonic pressure decay.

The theorem is deliberately finite-scale.  It is not a proof of global
regularity and it is not a construction of a singular solution.  Its role is
to turn persistent scale-critical badness into an explicit accounting
alternative: badness can survive only by paying through untaxed supply or
through non-negligible leakage.

The second part of the paper interprets the ledger coordinates as a
PDE-realizable finite-scale defect package.  This leads to a dynamic
obstruction language in which a dangerous mechanism must be profitable,
reproducible, Navier--Stokes-realizable, and visible or invisible through
specified pressure, flux, energy, and trace channels.  The final part records
finite-window quotient tests, a clean positive-cone anti-phantom theorem, and
a conditional localized transfer framework.  These results are finite-window
or conditional tools built around the main ledger theorem, not an
unconditional Navier--Stokes regularity theorem.
\end{abstract}


\tableofcontents

\section{Introduction}\label{sec:introduction}

The Caffarelli--Kohn--Nirenberg theory turns smallness of certain
scale-critical quantities into local regularity \cite{Scheffer1976,Scheffer1977,CKN1982,Lin1998,Vasseur2007,SereginLectureNotes}.  A possible singular point
is therefore a point at which these quantities fail to enter the decay basin
along arbitrarily small parabolic windows.  The guiding question of this
paper is not merely whether a scale-invariant quantity becomes large.  The
sharper question is:
\begin{center}
\emph{If a potentially singular trajectory avoids all decay scales, what repeatedly keeps the critical badness alive?}
\end{center}
The purpose of this manuscript is to turn that failure into a finite-scale
accounting statement.

\subsection{Main contribution}

The main contribution of this paper is a finite-scale PDE reduction for the
survival of CKN-critical badness.  Along an admissible chain of parabolic
windows, define
\[
    \Bbad_k=A_k+C_k+D_k,
\]
where \(A_k\), \(C_k\), and \(D_k\) are the standard scale-critical energy,
velocity, and pressure quantities.  The theorem proved in Part~I has the
following form:
\[
\boxed{
\text{long survival of }\Bbad_k
\quad\Longrightarrow\quad
\text{cumulative untaxed supply or accumulated leakage.}
}
\]
More precisely, after defining full supply, tax, and leakage terms from the
local energy inequality, interpolation, and pressure decay, we prove
\[
\Bbad_{k+1}-(1-\lambda)\Bbad_k
\le
\Supp^{\mathrm{full}}_k-\Tax^{\mathrm{full}}_k+\Leak^{\mathrm{full}}_k.
\]
Consequently, if \(\Bbad_k\ge \eps\) for \(0\le k\le N-1\), then
\[
\sum_{k=0}^{N-1}
\pos{\Supp^{\mathrm{full}}_k-\Tax^{\mathrm{full}}_k}
\ge
\lambda\eps N-\Bbad_0-
\sum_{k=0}^{N-1}\Leak^{\mathrm{full}}_k.
\]
This is the PDE spine of the paper.  All later defect-complex and
anti-phantom language is organized around this finite-scale ledger
inequality.

\subsection{Theorem status}

The theorem should not be read as a proof of global regularity for
Navier--Stokes, nor as a construction of a singular solution.  The main
unconditional content is the finite-scale ledger theorem above.  It says
that persistent scale-critical badness must be paid for by one of two
mechanisms: either localization/window leakage is non-negligible, or there
is a positive cumulative amount of supply not depleted by the available
taxes.

Thus the result is a necessary-mechanism theorem.  It does not show that
untaxed supply is impossible.  A regularity proof in this language would
require a further theorem showing that every Navier--Stokes-realizable
critical supply is uniformly taxed, up to negligible leakage and without a
derived gluing obstruction.

\subsection{Relation to the literature}

The analytic foundation is the Leray--Hopf weak-solution framework and the
suitable-weak-solution/partial-regularity theory of Scheffer and
Caffarelli--Kohn--Nirenberg, together with later refinements of local
\(\eps\)-regularity, endpoint regularity, and compactness methods \cite{Leray1934,Hopf1951,Scheffer1976,Scheffer1977,CKN1982,Lin1998,ESS2003,Vasseur2007,SereginLectureNotes}.  The pressure side of the ledger uses the standard local Calderon--Zygmund plus harmonic-pressure decomposition, in the spirit of the local pressure estimates used throughout the partial-regularity literature \cite{SohrWahl1986,SereginSverak2002,Wolf2017,SereginLectureNotes}.

The one-component motivation is adjacent to the regularity criteria based on a single velocity component, a single derivative, or anisotropic information, including the works \cite{KukavicaZiane2006,KukavicaZiane2007,ZhouPokorny2009,ZhouPokorny2010,CaoTiti2011,CheminZhang2016,CheminZhangZhang2017,KukavicaRusinZiane2017,HanLeiLiZhao2019,KangNguyen2023}.  Quantitative finite-scale and concentration viewpoints are also related to recent quantitative regularity and weak--strong uniqueness approaches \cite{BarkerPrange2021,AlbrittonBarkerPrange2023}.

The strict-shadow and anti-phantom language used in the later finite-window sections is closest to the finite-scale one-component and Schur-visibility reductions in \cite{Yu2026HarmonicPressure,Yu2026StrictShadows,Yu2026SchurVisibility}.  The present manuscript is complementary to those works: here the main object is not a one-component shadow selection theorem itself, but the supply--tax ledger and the finite-window observability tests that such reductions suggest.

The coarse-grained covariance and flux terminology is also compatible with commutator and anomalous-dissipation formulations of hydrodynamic energy transfer \cite{ConstantinETiti1994,Eyink1994,DuchonRobert2000,LeslieShvydkoy2018}.  Finally, the local-to-global gluing language uses the standard homological-algebra viewpoint on inverse systems and derived limits \cite{MacLane1963,Weibel1994}.

\subsection{Defect-cascade interpretation}

After the finite-scale PDE theorem is proved, the ledger coordinates are
reinterpreted as finite-scale defect data.  This does not replace the PDE
theorem by an abstract complex.  Rather, the local energy ledger supplies
concrete PDE-realizable coordinates, and only then do we write a schematic
finite-scale defect system
\[
    \calC_k \xrightarrow{G_k} \calD_k \xrightarrow{O_k} \calY_k,
    \qquad O_kG_k=0.
\]
Here \(\calD_k\) contains the ledger defect package, \(O_k\) represents
pressure, flux, energy, dissipation, and trace observations, and \(G_k\)
represents cutoff, harmonic-pressure, localization, and coarse-graining
cleanings.

This interpretation makes the obstruction dynamic.  A serious bad mechanism
cannot be merely a static invisible defect at one scale.  It must reproduce
across scales, remain profitable after ledger taxes are charged, and satisfy
the Navier--Stokes compatibility residuals.  The corresponding finite-window
object is a profitable reproducible verified mechanism.

\subsection{Proof architecture}

The proof has three layers.  First, Part~I derives the finite-scale ledger
directly from the local energy inequality, interpolation, pressure decay, and
dyadic bookkeeping.  Second, Part~II repackages the resulting PDE quantities
as localized finite-window defect data and identifies the dynamic PRV
mechanism that a persistent obstruction would have to realize.  Third,
Part~III studies finite-dimensional quotient tests and conditional
anti-phantom transfers.  The third layer is not another unconditional PDE
theorem; it is the algebraic and perturbative machinery needed for future
localized realizability estimates.

\subsection{Organization}

The manuscript is organized in three parts.

\begin{itemize}[leftmargin=2em]
    \item \textbf{Part I} proves the finite-scale critical ledger theorem.  This is the main PDE content of the paper.
    \item \textbf{Part II} interprets the ledger theorem as a Navier--Stokes-realizable defect-cascade framework and explains the local-to-global gluing obstruction.
    \item \textbf{Part III} records finite-window quotient tests, the positive-cone anti-phantom theorem, and a conditional localized transfer framework.  These sections are finite-window tools and theorem targets, not a solution of the global regularity problem.
\end{itemize}

\begin{warning}[Status of the framework]
The results below do not solve the Navier--Stokes regularity problem.  They also do not construct a singular solution.  The main unconditional content is a finite-scale reduction: persistent scale-critical badness must be paid for by untaxed supply or leakage.  The local-to-global defect language records what remains open after the finite-scale reduction has been made.
\end{warning}

\part{Finite-Scale Critical Ledgers}
\section{Navier--Stokes scale windows}\label{sec:scale-windows}
\noindent This section fixes the parabolic moving-window setting used in the finite-scale ledger theorem.

Let
\[
Q_r(z_0)=B_r(x_0)\times (t_0-r^2,t_0),
\qquad z_0=(x_0,t_0)\in \R^3\times \R.
\]
We consider the three-dimensional incompressible Navier--Stokes equations
\begin{equation}\label{eq:NS}
\partial_t u-\Delta u+u\cdot\nabla u+\nabla p=0,
\qquad \nabla\cdot u=0.
\end{equation}
Throughout, $(u,p)$ is a suitable weak solution in the sense of the local partial-regularity theory \cite{Scheffer1977,CKN1982,SereginLectureNotes}; in particular it satisfies the local energy inequality: for every nonnegative $\phi\in C_c^\infty$,
\begin{align}\label{eq:LEI}
&\int |u(x,t)|^2\phi(x,t)\dd x
+2\int \!\! \int |\nabla u|^2\phi\dd x\dd s  
\notag\\
&\qquad\leq
\int \!\! \int |u|^2(\partial_s\phi+\Delta\phi)\dd x\dd s
+
\int \!\! \int (|u|^2+2p)u\cdot\nabla\phi\dd x\dd s,
\end{align}
for the relevant time intervals.

\begin{definition}[Admissible scale-window chain]\label{def:chain}
Fix $0<\theta<1$ and $r_0>0$.  A sequence
\[
r_k=\theta^k r_0,
\qquad z_k=(x_k,t_k),
\qquad Q_k=Q_{r_k}(z_k),
\]
is called an admissible scale-window chain if
\[
Q_{k+1}\subset Q_k
\]
and the centers satisfy the controlled drift condition
\[
|x_{k+1}-x_k|\leq \eta r_k,
\qquad
|t_{k+1}-t_k|\leq \eta r_k^2
\]
for some fixed $\eta>0$.
\end{definition}

\begin{remark}
The moving-window formulation is included because a possible singular core need not remain centered at a fixed spatial point at every scale.  The framework keeps the drift cost visible through cutoff and window leakage terms.
\end{remark}

\section{CKN critical quantities and decay basins}\label{sec:critical-quantities}
\noindent This section introduces the scale-critical coordinates and the decay-basin language used to formulate sustained bad trajectories.

For each $Q_k=Q_{r_k}(z_k)$, define the standard Caffarelli--Kohn--Nirenberg scale-invariant quantities \cite{CKN1982,Lin1998,SereginLectureNotes}
\begin{align}
A_k
&=r_k^{-1}\operatorname*{ess\,sup}_{t_k-r_k^2<t<t_k}
\int_{B_{r_k}(x_k)} |u(x,t)|^2\dd x,\label{eq:A}\\
E_k
&=r_k^{-1}\int_{Q_k}|\nabla u|^2\dd x\dd t,\label{eq:E}\\
C_k
&=r_k^{-2}\int_{Q_k}|u|^3\dd x\dd t,\label{eq:C}\\
D_k
&=r_k^{-2}\int_{Q_k}|p-(p)_{B_{r_k}(x_k)}(t)|^{3/2}\dd x\dd t.\label{eq:D}
\end{align}
Here
\[
(p)_{B_{r_k}(x_k)}(t)=\frac{1}{|B_{r_k}|}
\int_{B_{r_k}(x_k)}p(x,t)\dd x.
\]

\begin{definition}[Preliminary critical state]\label{def:prelim-state}
A preliminary critical state at scale $k$ is the vector
\[
X_k=(A_k,E_k,C_k,D_k).
\]
A weighted critical size may be taken as
\[
\N_k=A_k+E_k+C_k+D_k,
\]
or, in a refined framework, as a weighted sum including flux, pressure-transfer, and leakage channels.
\end{definition}

\begin{definition}[Decay basin]\label{def:decay-basin}
Let $\eps>0$.  A basic CKN-type decay basin is
\[
\Gdec_\eps=\{X_k:C_k+D_k<\eps\}.
\]
More generally, once a critical size functional $\N$ is fixed, one may set
\[
\Gdec_\eps=\{X:\N(X)<\eps\}.
\]
\end{definition}

A sustained bad trajectory is one that remains outside the decay basin for many scales.  The central issue is not only that $\N_k$ remains large, but how the dynamics pays for this persistence.

\section{Discrete ledger summation lemma}\label{sec:discrete-ledger}
\noindent This section isolates the elementary summation principle that converts one-step supply--tax inequalities into finite-scale survival alternatives.

We first isolate the purely discrete bookkeeping principle.  This section is independent of the detailed Navier--Stokes estimates.

\begin{definition}[Sustenance residue]\label{def:sustenance-residue}
Let $\N_k\geq 0$ be a nonnegative critical-size sequence and let $0<\lambda<1$.  The one-step sustenance residue is
\[
\Mres_k=\N_{k+1}-(1-\lambda)\N_k.
\]
It measures the failure of the expected decay
\[
\N_{k+1}\leq (1-\lambda)\N_k.
\]
\end{definition}

\begin{definition}[Abstract critical ledger inequality]\label{def:abstract-ledger}
We say that the sequence satisfies an abstract critical ledger inequality if there are nonnegative leakage and error terms $\Leak_k,\Err_k$ and supply and tax terms $\Supp_k,\Tax_k$ such that
\begin{equation}\label{eq:abstract-ledger}
\Mres_k\leq \Supp_k-\Tax_k+\Leak_k+\Err_k.
\end{equation}
The term $\Supp_k$ represents critical supply, $\Tax_k$ represents depletion or taxation, and $\Leak_k+\Err_k$ represents localization, gauge, compactness, or bookkeeping loss.
\end{definition}

\begin{definition}[Untaxed critical supply]\label{def:untaxed}
Set
\[
U_k=\Supp_k-\Tax_k.
\]
The positive part $\pos{U_k}$ is the untaxed critical supply at step $k$.  Equivalently, scale $k$ carries untaxed supply of size at least $\beta>0$ if
\[
\Supp_k-\Tax_k\geq \beta.
\]
\end{definition}

\begin{definition}[Sustained bad orbit]\label{def:sustained-bad-orbit}
Let $\eps>0$, $M<\infty$, and $0<\lambda<1$.  A finite trajectory $\N_0,\ldots,\N_N$ is an $(\eps,M,\lambda)$-sustained bad orbit if
\[
\eps\leq \N_k\leq M,
\qquad 0\leq k\leq N-1,
\]
and the ledger inequality \eqref{eq:abstract-ledger} holds for every $0\leq k\leq N-1$.
\end{definition}

\begin{theorem}[Finite-scale critical survival alternative]\label{thm:abstract-survival}
Let $\N_k\geq0$ for $0\leq k\leq N$.  Suppose $0<\lambda<1$ and
\[
\Mres_k=\N_{k+1}-(1-\lambda)\N_k
\]
satisfies
\[
\Mres_k\leq \Supp_k-\Tax_k+\Leak_k+\Err_k.
\]
If
\[
\N_k\geq \eps,
\qquad 0\leq k\leq N-1,
\]
then
\begin{equation}\label{eq:abstract-survival-conclusion}
\sum_{k=0}^{N-1}\pos{\Supp_k-\Tax_k}
\geq
\lambda\eps N-\N_0-
\sum_{k=0}^{N-1}(\Leak_k+\Err_k).
\end{equation}
\end{theorem}

\begin{proof}
By definition,
\[
\Mres_k=\N_{k+1}-(1-\lambda)\N_k.
\]
Summing from $k=0$ to $N-1$ gives
\begin{align*}
\sum_{k=0}^{N-1}\Mres_k
&=\sum_{k=1}^{N}\N_k-(1-\lambda)\sum_{k=0}^{N-1}\N_k\\
&=\N_N-\N_0+\lambda\sum_{k=0}^{N-1}\N_k.
\end{align*}
Since $\N_N\geq0$ and $\N_k\geq \eps$ for $0\leq k\leq N-1$,
\[
\sum_{k=0}^{N-1}\Mres_k
\geq -\N_0+\lambda\eps N.
\]
On the other hand, the ledger inequality implies
\[
\sum_{k=0}^{N-1}\Mres_k
\leq \sum_{k=0}^{N-1}(\Supp_k-\Tax_k)
+\sum_{k=0}^{N-1}(\Leak_k+\Err_k).
\]
Since $\Supp_k-\Tax_k\leq \pos{\Supp_k-\Tax_k}$, the conclusion follows.
\end{proof}

\begin{corollary}[Positive-density untaxed supply]\label{cor:positive-density}
Assume the hypotheses of Theorem \ref{thm:abstract-survival}.  Suppose in addition that
\[
\sum_{k=0}^{N-1}(\Leak_k+\Err_k)\leq \frac14\lambda\eps N,
\qquad
\N_0\leq \frac14\lambda\eps N.
\]
Then
\[
\sum_{k=0}^{N-1}\pos{\Supp_k-\Tax_k}
\geq \frac12\lambda\eps N.
\]
If also $0\leq\pos{\Supp_k-\Tax_k}\leq B$ for all $k$, then for every $0<\beta<\frac12\lambda\eps$,
\[
\#\{0\leq k\leq N-1:\Supp_k-\Tax_k\geq\beta\}
\geq
\frac{\frac12\lambda\eps-\beta}{B-\beta}N.
\]
\end{corollary}

\begin{proof}
The first assertion follows immediately from \eqref{eq:abstract-survival-conclusion}.  For the second, let
\[
\mathcal I_\beta=\{k:\Supp_k-\Tax_k\geq\beta\},
\qquad m=\#\mathcal I_\beta.
\]
Then
\[
\sum_{k=0}^{N-1}\pos{\Supp_k-\Tax_k}\leq mB+(N-m)\beta.
\]
Combining this with the lower bound $\frac12\lambda\eps N$ and solving for $m$ gives the claim.
\end{proof}

\begin{corollary}[Infinite-scale version]\label{cor:infinite}
Let $\N_k\geq \eps$ for all $k\geq0$ and suppose
\[
\lim_{N\to\infty}\frac1N\sum_{k=0}^{N-1}(\Leak_k+\Err_k)=0.
\]
Then
\[
\liminf_{N\to\infty}\frac1N\sum_{k=0}^{N-1}\pos{\Supp_k-\Tax_k}
\geq \lambda\eps.
\]
\end{corollary}

\begin{proof}
Divide \eqref{eq:abstract-survival-conclusion} by $N$ and let $N\to\infty$.
\end{proof}

\begin{corollary}[Uniform taxation excludes sustained bad orbits]\label{cor:taxation-excludes}
Suppose that for all scales in a proposed infinite trajectory,
\[
\Supp_k-\Tax_k\leq \Leak_k+\Err_k,
\]
and
\[
\lim_{N\to\infty}\frac1N\sum_{k=0}^{N-1}(\Leak_k+\Err_k)=0.
\]
Then no infinite sustained bad orbit with $\N_k\geq\eps>0$ can exist.
\end{corollary}

\begin{proof}
The assumed taxation bound gives
\[
\pos{\Supp_k-\Tax_k}\leq \Leak_k+\Err_k.
\]
Thus the average of $\pos{\Supp_k-\Tax_k}$ tends to zero.  This contradicts Corollary \ref{cor:infinite}, which requires the lower bound $\lambda\eps>0$ for an infinite sustained bad orbit.
\end{proof}

\section{Local energy ledger}\label{sec:local-energy-ledger}
\noindent This section derives the one-step energy ledger directly from the local energy inequality.

The abstract ledger becomes meaningful only when its terms arise from the Navier--Stokes equations.  The first PDE input is the local energy inequality, inherited from the Leray--Hopf and suitable-weak-solution frameworks \cite{Leray1934,Hopf1951,CKN1982}.

Let $0<\rho<1$, $r>0$, and $z_0=(x_0,t_0)$.  Let $\phi\geq0$ be a smooth cutoff such that
\[
\phi\equiv1 \text{ on } Q_{\rho r}(z_0),
\qquad
\operatorname{supp}\phi\subset Q_r(z_0).
\]
Define
\begin{align*}
A_{\rho r}
&=(\rho r)^{-1}\operatorname*{ess\,sup}_{t_0-\rho^2r^2<t<t_0}
\int_{B_{\rho r}(x_0)} |u(x,t)|^2\dd x,\\
E_{\rho r}
&=(\rho r)^{-1}\int_{Q_{\rho r}(z_0)}|\nabla u|^2\dd x\dd t.
\end{align*}
Set
\begin{align*}
\Leak_r^{\rm cut}
&=r^{-1}\int_{Q_r(z_0)} |u|^2\left(|\partial_t\phi|+|\Delta\phi|\right)\dd x\dd t,\\
\Supp_r^{\rm flux}
&=r^{-1}\int_{Q_r(z_0)} |u|^2|u\cdot\nabla\phi|\dd x\dd t,\\
\Supp_r^{\rm press}
&=r^{-1}\int_{Q_r(z_0)} |p-(p)_{B_r}(t)|\,|u\cdot\nabla\phi|\dd x\dd t.
\end{align*}

\begin{proposition}[Local energy ledger]\label{prop:local-energy-ledger}
For every suitable weak solution and every cutoff as above,
\begin{equation}\label{eq:local-energy-ledger}
\rho A_{\rho r}+2\rho E_{\rho r}
\leq
\Leak_r^{\rm cut}+\Supp_r^{\rm flux}+2\Supp_r^{\rm press}.
\end{equation}
\end{proposition}

\begin{proof}
Apply the local energy inequality \eqref{eq:LEI} with the cutoff $\phi$.  Since $\phi\equiv1$ on $Q_{\rho r}(z_0)$, the left-hand side controls
\[
\operatorname*{ess\,sup}_{t_0-\rho^2r^2<t<t_0}
\int_{B_{\rho r}(x_0)}|u|^2\dd x
+2\int_{Q_{\rho r}(z_0)}|\nabla u|^2\dd x\dd t.
\]
Multiplying by $r^{-1}$ gives $\rho A_{\rho r}+2\rho E_{\rho r}$.  The first term on the right-hand side is bounded by $\Leak_r^{\rm cut}$, and the nonlinear transport term is bounded by $\Supp_r^{\rm flux}$.

For the pressure term, subtract the spatial mean $(p)_{B_r}(t)$.  Since $\nabla\cdot u=0$ and $\phi$ is compactly supported in $B_r$ at each time,
\[
\int_{B_r} (p)_{B_r}(t)u\cdot\nabla\phi\dd x
=(p)_{B_r}(t)\int_{B_r}u\cdot\nabla\phi\dd x
=0.
\]
Therefore
\[
\int p\,u\cdot\nabla\phi\dd x
=\int (p-(p)_{B_r}(t))u\cdot\nabla\phi\dd x,
\]
and the pressure contribution is bounded by $2\Supp_r^{\rm press}$.  This proves \eqref{eq:local-energy-ledger}.
\end{proof}

\begin{remark}
Proposition \ref{prop:local-energy-ledger} is the PDE origin of the ledger language: energy and dissipation at a smaller window must be paid for by flux supply, pressure supply, or cutoff/window leakage.
\end{remark}

\section{Ledger variables and point-edge critical states}\label{sec:ledger-variables}
\noindent This section separates reservoir coordinates from transition coordinates, so that supply, tax, and leakage are not treated as quantities of the same type.

The preliminary critical state $X_k=(A_k,E_k,C_k,D_k)$ treats all coordinates as if they had the same type.  For the ledger theory, this is not the right structure.  The quantities $A_k,C_k,D_k$ are reservoir coordinates at a scale, while flux, pressure transport, and leakage are transition coordinates from $Q_k$ to $Q_{k+1}$.

\subsection{Cutoffs adapted to the chain}

For an admissible chain $Q_{k+1}\subset Q_k$, choose $\phi_k\in C_c^\infty(Q_k)$ satisfying
\begin{align*}
0\leq\phi_k\leq1,
\qquad
\phi_k\equiv1 \text{ on }Q_{k+1},\\
|\nabla\phi_k|\leq C_\theta r_k^{-1},
\qquad
|\partial_t\phi_k|+|\Delta\phi_k|\leq C_\theta r_k^{-2}.
\end{align*}
Constants may depend on $\theta$ but not on $k,r_k,u,p$.

\begin{definition}[Reservoir badness]\label{def:reservoir}
The reservoir badness at scale $k$ is
\[
\Bbad_k=A_k+C_k+D_k.
\]
The dissipation $E_k$ is not placed in the reservoir; it is primarily a tax coordinate.
\end{definition}

\begin{definition}[Transition ledger coordinates]\label{def:transition-ledger}
For the transition $Q_k\to Q_{k+1}$, define
\begin{align}
\Phi_k
&=r_k^{-1}\int_{Q_k}|u|^2|u\cdot\nabla\phi_k|\dd x\dd t,
\label{eq:Phi}\\
\Pi_k
&=r_k^{-1}\int_{Q_k}|p-(p)_{B_{r_k}(x_k)}(t)|\,|u\cdot\nabla\phi_k|\dd x\dd t,
\label{eq:Pi}\\
\Lambda_k
&=r_k^{-1}\int_{Q_k}|u|^2\left(|\partial_t\phi_k|+|\Delta\phi_k|\right)\dd x\dd t.
\label{eq:Lambda}
\end{align}
Here $\Phi_k$ is nonlinear flux supply, $\Pi_k$ is pressure transport supply, and $\Lambda_k$ is cutoff/window leakage.
\end{definition}

\begin{definition}[Full point-edge critical state]\label{def:full-state}
The full critical state at scale $k$ is the point-edge object
\[
\mathbb X_k=(A_k,C_k,D_k\ ;\ E_{k+1}\ ;\ \Phi_k,\Pi_k,\Lambda_k).
\]
The semicolons distinguish reservoir coordinates, tax coordinates, and transition ledger coordinates.
\end{definition}

\section{Full critical ledger inequality}\label{sec:full-ledger-ineq}
\noindent This section combines the local energy ledger, cubic interpolation, and pressure decay into the one-step full ledger inequality.

The next estimates expand the local energy ledger into a full reservoir ledger for $\Bbad_{k+1}=A_{k+1}+C_{k+1}+D_{k+1}$.

\begin{lemma}[Energy transition ledger]\label{lem:energy-transition}
For every transition $Q_k\to Q_{k+1}$,
\begin{equation}\label{eq:energy-transition}
A_{k+1}+2E_{k+1}
\leq
\theta^{-1}(\Lambda_k+\Phi_k+2\Pi_k).
\end{equation}
In particular,
\begin{equation}\label{eq:A-ledger}
A_{k+1}
\leq
\theta^{-1}(\Lambda_k+\Phi_k+2\Pi_k)-2E_{k+1}.
\end{equation}
\end{lemma}

\begin{proof}
Apply Proposition \ref{prop:local-energy-ledger} with $r=r_k$ and $\rho=\theta$, using the cutoff $\phi_k$ equal to $1$ on $Q_{k+1}$.  Since $r_{k+1}=\theta r_k$,
\[
r_k^{-1}\operatorname*{ess\,sup}_{I_{k+1}}\int_{B_{r_{k+1}}}|u|^2\dd x
=\theta A_{k+1},
\]
and
\[
r_k^{-1}\int_{Q_{k+1}}|\nabla u|^2\dd x\dd t
=\theta E_{k+1}.
\]
Therefore
\[
\theta A_{k+1}+2\theta E_{k+1}\leq \Lambda_k+\Phi_k+2\Pi_k.
\]
Dividing by $\theta$ gives the claim.
\end{proof}

\begin{lemma}[Cubic interpolation]\label{lem:cubic-interpolation}
There is a universal constant $C_I>0$ such that
\begin{equation}\label{eq:cubic-interpolation}
C_{k+1}
\leq C_I\left(A_{k+1}^{3/4}E_{k+1}^{3/4}+A_{k+1}^{3/2}\right)
\leq C_I(A_{k+1}+E_{k+1})^{3/2}.
\end{equation}
Consequently,
\begin{equation}\label{eq:C-next-supply}
C_{k+1}
\leq C_{I,\theta}\left((\Phi_k+2\Pi_k)^{3/2}+\Lambda_k^{3/2}\right).
\end{equation}
\end{lemma}

\begin{proof}
The first estimate is the standard local interpolation inequality obtained from the spatial Gagliardo--Nirenberg and Sobolev inequalities on balls, integrated in time; see, for instance, the treatments in \cite{SereginLectureNotes}.  The second inequality is immediate from $xy\leq x^2+y^2$ in the appropriate powers.

By Lemma \ref{lem:energy-transition},
\[
A_{k+1}+E_{k+1}\leq \theta^{-1}(\Lambda_k+\Phi_k+2\Pi_k).
\]
Substituting into \eqref{eq:cubic-interpolation} and using
\[
(a+b)^{3/2}\leq C(a^{3/2}+b^{3/2})
\]
with $a=\Phi_k+2\Pi_k$ and $b=\Lambda_k$ yields \eqref{eq:C-next-supply}.
\end{proof}

\begin{lemma}[Standard local pressure decay]\label{lem:pressure-decay}
There is a constant $C_P>0$ such that, for $0<\theta<1$,
\begin{equation}\label{eq:pressure-decay}
D_{k+1}\leq C_P\theta D_k+C_P\theta^{-2}C_k.
\end{equation}
\end{lemma}

\begin{proof}
We record the standard Caffarelli--Kohn--Nirenberg pressure decay estimate
in the normalization used here \cite{CKN1982,SohrWahl1986,Wolf2017,SereginLectureNotes}.  Choose a cutoff supported in \(B_{r_k}\)
and equal to one on a slightly smaller ball containing \(B_{r_{k+1}}\).
On this smaller ball decompose
\[
p-(p)_{B_{r_k}}=p^{\rm loc}+h,
\]
where \(p^{\rm loc}\) is generated by the localized source
\((u_i-(u_i)_{B_{r_k}})(u_j-(u_j)_{B_{r_k}})\) through the
Calderon--Zygmund pressure operator, and \(h\) is spatially harmonic.
After scaling to \(B_1\), the Calderon--Zygmund estimate gives the
localized contribution bounded by \(C\theta^{-2}C_k\) in the \(D\)-scale.
The harmonic Campanato decay estimate gives
\[
        D(h;r_{k+1})\le C\theta D(h;r_k).
\]
Combining these two bounds and absorbing harmless changes of spatial
averages into the definition of \(D_k\) gives
\eqref{eq:pressure-decay}.
\end{proof}

\begin{definition}[Full sustenance residue]\label{def:full-residue}
Let $0<\lambda<1$.  The full sustenance residue is
\[
\Mres_k^{\rm full}=\Bbad_{k+1}-(1-\lambda)\Bbad_k.
\]
\end{definition}

\begin{definition}[Full supply, tax, and leakage]\label{def:full-ledger-terms}
Assume $\theta$ is chosen so that
\[
\delta_D=(1-\lambda)-C_P\theta>0.
\]
Define
\begin{align*}
\Supp_k^{\rm full}
&=\theta^{-1}(\Phi_k+2\Pi_k)
+C_{I,\theta}(\Phi_k+2\Pi_k)^{3/2}
+C_P\theta^{-2}C_k,\\
\Tax_k^{\rm full}
&=2E_{k+1}+(1-\lambda)A_k+(1-\lambda)C_k+\delta_DD_k,\\
\Leak_k^{\rm full}
&=\theta^{-1}\Lambda_k+C_{I,\theta}\Lambda_k^{3/2}.
\end{align*}
\end{definition}

\begin{theorem}[Full critical ledger inequality]\label{thm:full-ledger}
Let $(u,p)$ be a suitable weak solution on an admissible scale-window chain.  If $\theta$ is chosen so that $\delta_D=(1-\lambda)-C_P\theta>0$, then
\begin{equation}\label{eq:full-ledger}
\Mres_k^{\rm full}
\leq
\Supp_k^{\rm full}-\Tax_k^{\rm full}+\Leak_k^{\rm full}.
\end{equation}
\end{theorem}

\begin{proof}
From Lemma \ref{lem:energy-transition},
\[
A_{k+1}\leq \theta^{-1}(\Lambda_k+\Phi_k+2\Pi_k)-2E_{k+1}.
\]
From Lemma \ref{lem:cubic-interpolation},
\[
C_{k+1}\leq C_{I,\theta}\left((\Phi_k+2\Pi_k)^{3/2}+\Lambda_k^{3/2}\right).
\]
From Lemma \ref{lem:pressure-decay},
\[
D_{k+1}\leq C_P\theta D_k+C_P\theta^{-2}C_k.
\]
Adding these three inequalities yields
\begin{align*}
\Bbad_{k+1}
&\leq
\theta^{-1}(\Phi_k+2\Pi_k)+\theta^{-1}\Lambda_k-2E_{k+1}\\
&\quad +C_{I,\theta}(\Phi_k+2\Pi_k)^{3/2}
+C_{I,\theta}\Lambda_k^{3/2}\\
&\quad +C_P\theta D_k+C_P\theta^{-2}C_k.
\end{align*}
Subtracting $(1-\lambda)\Bbad_k=(1-\lambda)(A_k+C_k+D_k)$ gives
\begin{align*}
\Mres_k^{\rm full}
&\leq
\theta^{-1}(\Phi_k+2\Pi_k)
+C_{I,\theta}(\Phi_k+2\Pi_k)^{3/2}
+C_P\theta^{-2}C_k\\
&\quad -2E_{k+1}-(1-\lambda)A_k-(1-\lambda)C_k\\
&\quad -\big((1-\lambda)-C_P\theta\big)D_k
+\theta^{-1}\Lambda_k+C_{I,\theta}\Lambda_k^{3/2}.
\end{align*}
Using the definition of $\delta_D$ and grouping the terms gives \eqref{eq:full-ledger}.
\end{proof}

\section{Full critical survival alternative}\label{sec:full-survival}
\noindent This section applies the discrete ledger summation lemma to the Navier--Stokes full ledger inequality.

\begin{definition}[Full untaxed critical supply]\label{def:full-untaxed}
The full untaxed critical supply at scale $k$ is
\[
\mathsf{Untaxed}_k^{\rm full}
=\Supp_k^{\rm full}-\Tax_k^{\rm full}.
\]
The scale $k$ has nontrivial full untaxed supply if
\[
\Supp_k^{\rm full}-\Tax_k^{\rm full}>\Leak_k^{\rm full}.
\]
\end{definition}

\begin{theorem}[Full critical survival alternative]\label{thm:full-survival}
Assume the hypotheses of Theorem \ref{thm:full-ledger}.  If
\[
\Bbad_k\geq\eps,
\qquad 0\leq k\leq N-1,
\]
then
\begin{equation}\label{eq:full-survival}
\sum_{k=0}^{N-1}\pos{\Supp_k^{\rm full}-\Tax_k^{\rm full}}
\geq
\lambda\eps N-
\Bbad_0-
\sum_{k=0}^{N-1}\Leak_k^{\rm full}.
\end{equation}
\end{theorem}

\begin{proof}
Apply Theorem \ref{thm:abstract-survival} with
\[
\N_k=\Bbad_k,
\qquad
\Supp_k=\Supp_k^{\rm full},
\qquad
\Tax_k=\Tax_k^{\rm full},
\qquad
\Leak_k=\Leak_k^{\rm full},
\qquad
\Err_k=0.
\]
\end{proof}

\begin{remark}[Interpretation]
The estimate \eqref{eq:full-survival} says: if the reservoir badness $\Bbad_k=A_k+C_k+D_k$ survives for $N$ consecutive scales, then either accumulated localization leakage is large, or the full untaxed critical supply has positive cumulative size.  The supply consists of nonlinear flux, pressure transport, cubic-to-pressure regeneration, and interpolation amplification.  The tax consists of viscous dissipation, expected decay of old reservoir coordinates, and harmonic pressure decay.
\end{remark}

\section{Closure estimates and immediate consequences}\label{sec:closure-consequences}
\noindent This section records that the transition ledger coordinates are not arbitrary; they are controlled by the reservoir coordinates.

The transition coordinates are not arbitrary.  They are controlled by the reservoir coordinates.

\begin{lemma}[Transition closure]\label{lem:transition-closure}
For the cutoff family above,
\begin{align}
\Lambda_k&\leq C_\theta A_k,\label{eq:Lambda-closure}\\
\Phi_k&\leq C_\theta C_k,\label{eq:Phi-closure}\\
\Pi_k&\leq C_\theta C_k^{1/3}D_k^{2/3}.\label{eq:Pi-closure}
\end{align}
\end{lemma}

\begin{proof}
Since $|\partial_t\phi_k|+|\Delta\phi_k|\leq C_\theta r_k^{-2}$,
\[
\Lambda_k
\leq C_\theta r_k^{-3}\int_{Q_k}|u|^2\dd x\dd t
\leq C_\theta A_k.
\]
Since $|\nabla\phi_k|\leq C_\theta r_k^{-1}$,
\[
\Phi_k
\leq C_\theta r_k^{-2}\int_{Q_k}|u|^3\dd x\dd t
=C_\theta C_k.
\]
For the pressure term, Holder's inequality gives
\begin{align*}
\Pi_k
&\leq C_\theta r_k^{-2}\int_{Q_k}|p-(p)_{B_{r_k}}(t)|\,|u|\dd x\dd t\\
&\leq C_\theta r_k^{-2}
\left(\int_{Q_k}|p-(p)_{B_{r_k}}(t)|^{3/2}\dd x\dd t\right)^{2/3}
\left(\int_{Q_k}|u|^3\dd x\dd t\right)^{1/3}\\
&=C_\theta D_k^{2/3}C_k^{1/3}.
\end{align*}
This proves the closure estimates.
\end{proof}

\begin{corollary}[Discrete critical dynamics]\label{cor:discrete-dynamics}
There is a nondecreasing function $\mathcal T_\theta$ such that
\[
\Bbad_{k+1}\leq \mathcal T_\theta(A_k,C_k,D_k)-2E_{k+1}.
\]
One may take, up to a change of constants,
\begin{align*}
\mathcal T_\theta(A,C,D)
=C_\theta\Big[&A+C+C^{1/3}D^{2/3}
+\big(A+C+C^{1/3}D^{2/3}\big)^{3/2}\\
&+\theta^{-2}C+\theta D\Big].
\end{align*}
\end{corollary}

\begin{proof}
Combine Lemma \ref{lem:transition-closure} with the three component estimates used in the proof of Theorem \ref{thm:full-ledger}.
\end{proof}

\part{Defect-Cascade Interpretation}
\section{Ledger-realizable defect packages}\label{sec:ledger-defects}
\noindent This section translates the already-proved PDE ledger coordinates into a finite-scale defect complex.  No new PDE theorem is claimed here.

The ledger theorem above is the main PDE content of the manuscript.  We now explain how it fits into a broader scale-defect viewpoint.  The purpose of this section is not to introduce a second independent theorem, but to reinterpret the already-defined ledger coordinates as concrete finite-scale defect data.

\subsection{Localized dyadic rescaling}\label{subsec:localized-dyadic-rescaling}

We first fix the local rescaling convention that will be used later to
construct finite-window localized defect packages.  Let
\[
        z_0=(x_0,t_0),\qquad r_k=2^{-k}r_0,
\]
and set
\[
        Q_k^{\rm loc}
        :=
        B_{r_k}(x_0)\times(t_0-r_k^2,t_0).
\]
Assume that \((u,p)\) is a suitable weak solution in a neighborhood of
\(Q_0^{\rm loc}\).  For each \(k\), define the rescaled fields on
\[
        Q_1:=B_1(0)\times(-1,0)
\]
by
\begin{equation}\label{eq:localized-rescaling}
        u^{(k)}(y,s)
        :=
        r_k u(x_0+r_k y,t_0+r_k^2s),
        \qquad
        p^{(k)}(y,s)
        :=
        r_k^2 p(x_0+r_k y,t_0+r_k^2s).
\end{equation}
For \(0<\rho\le1\), define the rescaled critical quantities
\begin{align*}
        A^{(k)}(\rho)
        &:=
        \rho^{-1}
        \operatorname*{ess\,sup}_{-\rho^2<s<0}
        \int_{B_\rho}|u^{(k)}(y,s)|^2\dd y,\\
        E^{(k)}(\rho)
        &:=
        \rho^{-1}
        \int_{B_\rho\times(-\rho^2,0)}
        |\nabla_yu^{(k)}|^2\dd y\dd s,\\
        C^{(k)}(\rho)
        &:=
        \rho^{-2}
        \int_{B_\rho\times(-\rho^2,0)}
        |u^{(k)}|^3\dd y\dd s,\\
        D^{(k)}(\rho)
        &:=
        \rho^{-2}
        \int_{B_\rho\times(-\rho^2,0)}
        \left|
        p^{(k)}-(p^{(k)})_{B_\rho}(s)
        \right|^{3/2}\dd y\dd s .
\end{align*}

\begin{lemma}[Localized scaling invariance]\label{lem:localized-scaling-invariance}
For each \(k\), the pair \((u^{(k)},p^{(k)})\) is a suitable weak solution
of the Navier--Stokes equations on \(Q_1\).  Moreover, for every
\(0<\rho\le1\),
\begin{align}
        A^{(k)}(\rho)
        &=
        (\rho r_k)^{-1}
        \operatorname*{ess\,sup}_{t_0-\rho^2r_k^2<t<t_0}
        \int_{B_{\rho r_k}(x_0)}|u(x,t)|^2\dd x,\label{eq:scaled-A}\\
        E^{(k)}(\rho)
        &=
        (\rho r_k)^{-1}
        \int_{Q_{\rho r_k}(z_0)}
        |\nabla u|^2\dd x\dd t,\label{eq:scaled-E}\\
        C^{(k)}(\rho)
        &=
        (\rho r_k)^{-2}
        \int_{Q_{\rho r_k}(z_0)}
        |u|^3\dd x\dd t,\label{eq:scaled-C}\\
        D^{(k)}(\rho)
        &=
        (\rho r_k)^{-2}
        \int_{Q_{\rho r_k}(z_0)}
        \left|
        p-(p)_{B_{\rho r_k}(x_0)}(t)
        \right|^{3/2}\dd x\dd t.\label{eq:scaled-D}
\end{align}
Thus the Caffarelli--Kohn--Nirenberg quantities are invariant under the
localized normalization \eqref{eq:localized-rescaling}.
\end{lemma}

\begin{proof}
The distributional Navier--Stokes equations are invariant under
\eqref{eq:localized-rescaling}.  Indeed, with
\[
        x=x_0+r_ky,\qquad t=t_0+r_k^2s,
\]
one has
\[
        \partial_s u^{(k)}=r_k^3\partial_tu,\qquad
        \Delta_yu^{(k)}=r_k^3\Delta u,\qquad
        u^{(k)}\cdot\nabla_yu^{(k)}
        =r_k^3u\cdot\nabla u,
\]
and
\[
        \nabla_yp^{(k)}=r_k^3\nabla p.
\]
The divergence condition scales as
\[
        \nabla_y\cdot u^{(k)}=r_k^2\nabla_x\cdot u=0.
\]
Thus \((u^{(k)},p^{(k)})\) satisfies the equations in distributions.

The local energy inequality is preserved by the same change of variables:
given a nonnegative test function \(\psi\in C_c^\infty(Q_1)\), apply the
local energy inequality for \((u,p)\) with
\[
        \phi(x,t)
        =
        \psi\left(\frac{x-x_0}{r_k},\frac{t-t_0}{r_k^2}\right)
\]
and change variables back to \((y,s)\).  Each term acquires the same
overall scaling factor, which cancels, yielding the local energy
inequality for \((u^{(k)},p^{(k)})\).

It remains to check the critical quantities.  Since
\[
        \dd y\dd s=r_k^{-5}\dd x\dd t,\qquad
        |u^{(k)}|=r_k|u|,\qquad
        |\nabla_yu^{(k)}|=r_k^2|\nabla_xu|,
\]
the identities \eqref{eq:scaled-A}--\eqref{eq:scaled-C} follow directly.
For the pressure term,
\[
        (p^{(k)})_{B_\rho}(s)
        =
        r_k^2(p)_{B_{\rho r_k}(x_0)}(t),
        \qquad t=t_0+r_k^2s.
\]
Therefore
\[
        \left|
        p^{(k)}-(p^{(k)})_{B_\rho}(s)
        \right|^{3/2}
        =
        r_k^3
        \left|
        p-(p)_{B_{\rho r_k}(x_0)}(t)
        \right|^{3/2},
\]
and the change of variables gives \eqref{eq:scaled-D}.  This proves the
scale-invariance identities.
\end{proof}

\begin{remark}[Scope of the rescaling step]
Lemma~\ref{lem:localized-scaling-invariance} only fixes the local
normalization and the scale-critical bookkeeping.  It does not yet define
the finite-dimensional localized defect space
\(\mathcal D_\Lambda^{\rm loc}\), the localized cleaning map, the
observability map, or the localized residual and reproduction maps.  Those
objects require the additional finite projection, cutoff, pressure-gauge,
ledger, and slack-variable choices planned in the next construction step.
\end{remark}

\subsection{Localized finite-window defect coordinates}\label{subsec:localized-defect-coordinates}

Let \(\Lambda_{\rm sc}\subset\mathbb Z\) be a finite set of dyadic scale
indices.  A \emph{localized finite-window projection datum} consists, for
each \(k\in\Lambda_{\rm sc}\), of a cutoff
\[
        \chi_k\in C_c^\infty(Q_1),
        \qquad 0\leq \chi_k\leq1,
\]
and finite-dimensional Hilbert spaces
\[
        V_{\Lambda,k}^u,\qquad
        V_{\Lambda,k}^p,\qquad
        V_{\Lambda,k}^R
\]
of localized velocity, pressure, and symmetric stress profiles on \(Q_1\),
together with bounded linear maps
\[
        \mathfrak P_{\Lambda,k}^u:L^3(Q_1;\R^3)\to V_{\Lambda,k}^u,
\]
\[
        \mathfrak P_{\Lambda,k}^p:L^{3/2}(Q_1)\to V_{\Lambda,k}^p,
        \qquad
        \mathfrak P_{\Lambda,k}^R:
        L^{3/2}(Q_1;\R_{\rm sym}^{3\times3})
        \to V_{\Lambda,k}^R .
\]
The cutoff may be built into these maps; equivalently, one may read
\(\mathfrak P_{\Lambda,k}^\bullet\) as ``project and then localize by
\(\chi_k\).''  No approximation quality is assumed at this stage.

For a single scale \(k\), define the localized coordinate space
\[
\begin{aligned}
        \mathcal D_{\Lambda,k}^{\rm loc}
        &:=
        V_{\Lambda,k}^u
        \oplus
        V_{\Lambda,k}^p
        \oplus
        V_{\Lambda,k}^R
        \oplus
        \R_\Phi
        \oplus
        \R_\Pi
        \oplus
        \R_L
        \oplus
        \R^6_s .
\end{aligned}
\]
An element is written
\[
        \mathfrak D_k
        =
        (U_k,P_k,R_k,\Phi_k,\Pi_k,L_k,s_k),
        \qquad
        s_k=(s_k^{\rm en},s_k^{\rm cub},s_k^{\rm prs},
        s_k^L,s_k^\Phi,s_k^\Pi).
\]
Here \(L_k\) is the localized leakage coordinate; it corresponds to the
quantity denoted \(\Lambda_k\) in the earlier ledger sections, but is
renamed here to avoid confusion with the finite window
\(\Lambda_{\rm sc}\).  The finite-window localized defect space is
\[
        \mathcal D_\Lambda^{\rm loc}
        :=
        \prod_{k\in\Lambda_{\rm sc}}
        \mathcal D_{\Lambda,k}^{\rm loc},
\]
with the product Hilbert norm induced by the chosen Hilbert norms on the
finite-dimensional profile spaces and the Euclidean norms on the scalar
coordinates.

Given a suitable weak solution \((u,p)\), the localized rescaled fields
\((u^{(k)},p^{(k)})\), and a projection datum, define the projected
velocity, pressure, and coarse stress coordinates by
\begin{align}
        U_k
        &:=
        \mathfrak P_{\Lambda,k}^u u^{(k)},\label{eq:loc-U-coordinate}\\
        P_k
        &:=
        \mathfrak P_{\Lambda,k}^p
        \left(
        p^{(k)}-(p^{(k)})_{B_1}(s)
        \right),\label{eq:loc-P-coordinate}\\
        R_k
        &:=
        \mathfrak P_{\Lambda,k}^R
        \left(
        u^{(k)}\otimes u^{(k)}
        -
        U_k\otimes U_k
        \right).\label{eq:loc-R-coordinate}
\end{align}
The remaining scalar coordinates \(\Phi_k,\Pi_k,L_k,s_k\) are the
localized ledger and slack coordinates defined in the next step.

\begin{lemma}[Finite-dimensional localized package]\label{lem:localized-package-well-defined}
For every finite localized projection datum, the space
\(\mathcal D_\Lambda^{\rm loc}\) is finite-dimensional.  If \((u,p)\) is
a suitable weak solution near \(Q_0^{\rm loc}\), then the coordinates
\(U_k,P_k,R_k\) in
\eqref{eq:loc-U-coordinate}--\eqref{eq:loc-R-coordinate} are well-defined
for every \(k\in\Lambda_{\rm sc}\).
\end{lemma}

\begin{proof}
Finite-dimensionality follows from the definition, since
\(\Lambda_{\rm sc}\) is finite and each factor
\(V_{\Lambda,k}^u,V_{\Lambda,k}^p,V_{\Lambda,k}^R\) is finite-dimensional.

By Lemma~\ref{lem:localized-scaling-invariance}, each \(u^{(k)}\) is a
local suitable weak velocity field on \(Q_1\).  In particular the standard
local integrability \(u^{(k)}\in L^3_{\rm loc}(Q_1)\) and
\(p^{(k)}\in L^{3/2}_{\rm loc}(Q_1)\) is available on the support of the
localized projection datum.  Hence \(u^{(k)}\otimes u^{(k)}\in
L^{3/2}_{\rm loc}(Q_1)\).  The bounded projection maps then define
\(U_k,P_k,R_k\) as elements of the stated finite-dimensional profile
spaces.
\end{proof}

\begin{remark}[What is not proved here]
Lemma~\ref{lem:localized-package-well-defined} constructs the
finite-dimensional coordinate space and the first projected
Navier--Stokes-derived coordinates.  It does not assert that the projection
error is small, that the stress coordinate satisfies a closed equation, or
that \(\mathcal D_\Lambda^{\rm loc}\) is already compatible with the clean
periodic window.  Those are separate pressure, localization, truncation,
reproduction, and chart estimates.
\end{remark}

\subsection{Localized active and harmonic pressure components}\label{subsec:localized-pressure-splitting}

The local pressure must be split before it can be compared with a clean
periodic pressure solve, following the standard active-plus-harmonic local pressure decomposition \cite{SohrWahl1986,Wolf2017,SereginLectureNotes}.  For each \(k\in\Lambda_{\rm sc}\), choose
\[
        \eta_k\in C_c^\infty(B_1),
        \qquad
        \eta_k\equiv1
        \quad\text{on a neighborhood of }\operatorname{supp}\chi_k .
\]
For the rescaled pressure \(p^{(k)}\), define the active pressure on
\(\R^3\) by
\begin{equation}\label{eq:localized-active-pressure}
        p_k^{\rm act}
        :=
        \mathcal R_i\mathcal R_j
        \bigl(\eta_k u_i^{(k)}u_j^{(k)}\bigr),
\end{equation}
where \(\mathcal R_i\) denotes the \(i\)-th Riesz transform on
\(\R^3\).  The harmonic pressure tail is
\[
        p_k^{\rm harm}:=p^{(k)}-p_k^{\rm act}.
\]

\begin{lemma}[Localized pressure splitting]\label{lem:localized-pressure-splitting}
On every open set compactly contained in the region where \(\eta_k=1\),
the function \(p_k^{\rm harm}\) is spatially harmonic in the sense of
distributions.  Equivalently,
\[
        -\Delta p_k^{\rm act}
        =
        \partial_i\partial_j
        \bigl(\eta_k u_i^{(k)}u_j^{(k)}\bigr)
\]
on \(\R^3\), and
\[
        -\Delta p_k^{\rm harm}=0
\]
inside the localized core where \(\eta_k=1\).
\end{lemma}

\begin{proof}
For suitable weak solutions, the local pressure relation holds modulo an
additive function of time, which is immaterial after applying the spatial
Laplacian and after subtracting spatial means.  Thus, on the rescaled
interior region under consideration,
\[
        -\Delta p^{(k)}
        =
        \partial_i\partial_j
        (u_i^{(k)}u_j^{(k)})
\]
in distributions.  By the definition of the Riesz transforms on
\(\R^3\),
\[
        -\Delta
        \mathcal R_i\mathcal R_j f
        =
        \partial_i\partial_j f
\]
for compactly supported \(f\).  Applying this with
\(f=\eta_k u_i^{(k)}u_j^{(k)}\) gives the displayed identity for
\(p_k^{\rm act}\).  On the region where \(\eta_k=1\), the two Poisson
right-hand sides agree, hence
\[
        -\Delta(p^{(k)}-p_k^{\rm act})=0
\]
there.  This is the asserted harmonicity of \(p_k^{\rm harm}\).
\end{proof}

To keep the later transfer theorem honest, we record the pressure error as
an explicit object rather than treating it as negligible.  Choose
seminorms \(\|\cdot\|_{\mathcal H_{\Lambda,k}}\) for harmonic tails and
\(\|\cdot\|_{\mathcal P_{\Lambda,k}}\) for pressure commutators, and set
\begin{equation}\label{eq:localized-pressure-error}
\begin{aligned}
        \Err_{\rm prs,\Lambda}^{\rm loc}
        &:=
        \left(
        \sum_{k\in\Lambda_{\rm sc}}
        \left[
        \|p_k^{\rm harm}\|_{\mathcal H_{\Lambda,k}}
        +
        \left\|
        [\eta_k,\mathcal R_i\mathcal R_j]
        (u_i^{(k)}u_j^{(k)})
        \right\|_{\mathcal P_{\Lambda,k}}
        \right]^2
        \right)^{1/2}.
\end{aligned}
\end{equation}
The first term measures the pressure not determined by the local active
Poisson solve.  The second term measures the difference between localizing
the nonlinear source before applying the Riesz solve and applying the
clean whole-space pressure operator without a cutoff.

\begin{remark}[Status of the pressure step]
Lemma~\ref{lem:localized-pressure-splitting} proves the local
active/harmonic decomposition.  The estimate that
\(\Err_{\rm prs,\Lambda}^{\rm loc}\) is small, or that it is dominated by a
chosen localization budget \(\Delta_\Lambda\), is not proved here.  That
smallness is a later pressure-transfer input.
\end{remark}

\subsection{Localized ledger and slack coordinates}\label{subsec:localized-ledger-slack}

We now define the scalar ledger coordinates in the localized package.  In
this dyadic subsection set
\[
        \vartheta:=\frac12,
        \qquad
        Q_\vartheta:=B_\vartheta(0)\times(-\vartheta^2,0).
\]
For each \(k\in\Lambda_{\rm sc}\), choose a nonnegative cutoff
\[
        \varphi_k\in C_c^\infty(Q_1),
        \qquad
        \varphi_k\equiv1\quad\text{on }Q_\vartheta .
\]
For the rescaled fields \((u^{(k)},p^{(k)})\), define
\begin{align}
        L_k^{\rm loc}
        &:=
        \int_{Q_1}
        |u^{(k)}|^2
        (|\partial_s\varphi_k|+|\Delta\varphi_k|)
        \dd y\dd s,\label{eq:loc-leakage}\\
        \Phi_k^{\rm loc}
        &:=
        \int_{Q_1}
        |u^{(k)}|^2
        |u^{(k)}\cdot\nabla\varphi_k|
        \dd y\dd s,\label{eq:loc-flux}\\
        \Pi_k^{\rm loc}
        &:=
        \int_{Q_1}
        \left|
        p^{(k)}-(p^{(k)})_{B_1}(s)
        \right|
        |u^{(k)}\cdot\nabla\varphi_k|
        \dd y\dd s.\label{eq:loc-pressure-transport}
\end{align}
These are the rescaled versions of the cutoff leakage, nonlinear flux, and
pressure-transport coordinates used in Part~I.

The associated localized slack variables are defined by
\begin{align}
        s_k^{\rm en}
        &:=
        \vartheta^{-1}
        (L_k^{\rm loc}+\Phi_k^{\rm loc}+2\Pi_k^{\rm loc})
        -
        A^{(k)}(\vartheta)-2E^{(k)}(\vartheta),\label{eq:loc-slack-energy}\\
        s_k^{\rm cub}
        &:=
        C_{I,\vartheta}
        \left(
        (\Phi_k^{\rm loc}+2\Pi_k^{\rm loc})^{3/2}
        +(L_k^{\rm loc})^{3/2}
        \right)
        -
        C^{(k)}(\vartheta),\label{eq:loc-slack-cubic}\\
        s_k^{\rm prs}
        &:=
        C_P\vartheta D^{(k)}(1)
        +
        C_P\vartheta^{-2}C^{(k)}(1)
        -
        D^{(k)}(\vartheta),\label{eq:loc-slack-pressure}\\
        s_k^L
        &:=
        C_\vartheta A^{(k)}(1)-L_k^{\rm loc},\label{eq:loc-slack-L}\\
        s_k^\Phi
        &:=
        C_\vartheta C^{(k)}(1)-\Phi_k^{\rm loc},\label{eq:loc-slack-Phi}\\
        s_k^\Pi
        &:=
        C_\vartheta
        C^{(k)}(1)^{1/3}D^{(k)}(1)^{2/3}
        -
        \Pi_k^{\rm loc}.\label{eq:loc-slack-Pi}
\end{align}

\begin{lemma}[Localized ledger slack realization]\label{lem:localized-ledger-slack}
For every suitable weak solution and every adapted localized cutoff
\(\varphi_k\), the quantities
\[
        L_k^{\rm loc},\quad
        \Phi_k^{\rm loc},\quad
        \Pi_k^{\rm loc}
\]
are finite.  With the constants chosen as in the full ledger estimates,
the slack variables in
\eqref{eq:loc-slack-energy}--\eqref{eq:loc-slack-Pi} are nonnegative.
Equivalently, the localized ledger inequalities may be written as exact
finite-window residual identities by adding the nonnegative slack vector
\[
        s_k
        =
        (s_k^{\rm en},s_k^{\rm cub},s_k^{\rm prs},
        s_k^L,s_k^\Phi,s_k^\Pi).
\]
\end{lemma}

\begin{proof}
The finiteness of the three ledger coordinates follows from the local
integrability of suitable weak solutions, the compact support and
smoothness of \(\varphi_k\), and the pressure integrability
\(p^{(k)}\in L^{3/2}_{\rm loc}\).

The nonnegativity of \(s_k^{\rm en}\) is the rescaled form of the energy
transition ledger, Lemma~\ref{lem:energy-transition}.  The nonnegativity of
\(s_k^{\rm cub}\) follows from the cubic interpolation estimate,
Lemma~\ref{lem:cubic-interpolation}, after the energy transition bound is
inserted.  The pressure slack \(s_k^{\rm prs}\) is nonnegative by the
pressure decay estimate, Lemma~\ref{lem:pressure-decay}.  Finally,
\eqref{eq:loc-slack-L}--\eqref{eq:loc-slack-Pi} are the rescaled closure
estimates of Lemma~\ref{lem:transition-closure}.  The scaling identities
from Lemma~\ref{lem:localized-scaling-invariance} identify these rescaled
inequalities with the original scale-\(k\) ledger inequalities.
\end{proof}

\begin{remark}[Ledger status]
This step converts the already-proved finite-scale ledger inequalities
into localized finite-window coordinates.  It does not prove that these
ledger variables are close to their clean periodic analogues, nor that the
localized leakage is small.  Those claims belong to the later detection-map
and error-budget comparison steps.
\end{remark}

\subsection{Localized Navier--Stokes residual map}\label{subsec:localized-residual-map}

To turn the localized coordinate package into a finite-window compatibility
problem, fix residual target spaces
\[
        Z_{\Lambda,k}^{\rm div},\quad
        Z_{\Lambda,k}^{\rm mom},\quad
        Z_{\Lambda,k}^{\rm prs},\quad
        Z_{\Lambda,k}^{\rm led},\quad
        Z_{\Lambda,k}^{\rm glue}
\]
with chosen norms.  Also fix finite-window critical-size functionals
\[
        A_{\Lambda,k}^{\rm loc},\quad
        E_{\Lambda,k}^{\rm loc},\quad
        C_{\Lambda,k}^{\rm loc},\quad
        D_{\Lambda,k}^{\rm loc}
\]
on the localized coordinate variables, and corresponding smaller-core
functionals
\[
        A_{\Lambda,k,\vartheta}^{\rm loc},\quad
        E_{\Lambda,k,\vartheta}^{\rm loc},\quad
        C_{\Lambda,k,\vartheta}^{\rm loc},\quad
        D_{\Lambda,k,\vartheta}^{\rm loc}.
\]
For an exact NS-derived, unprojected package these are the quantities in
Lemma~\ref{lem:localized-scaling-invariance}; in a projected finite window
they are part of the model data and record the chosen finite-dimensional
critical-size convention.

Finally, for each adjacent pair \(k,k+1\in\Lambda_{\rm sc}\), choose a
finite-window transfer map
\[
        \mathfrak T_{\Lambda,k}^{\rm loc}:
        \mathcal D_{\Lambda,k}^{\rm loc}
        \longrightarrow
        \mathcal D_{\Lambda,k+1}^{\rm loc}.
\]
This map is not assumed to be dynamically exact; its defect is measured by
the glue and reproduction residuals.

For
\[
        \mathfrak D_\Lambda
        =
        (\mathfrak D_k)_{k\in\Lambda_{\rm sc}}
        \in\mathcal D_\Lambda^{\rm loc},
        \qquad
        \mathfrak D_k=(U_k,P_k,R_k,\Phi_k,\Pi_k,L_k,s_k),
\]
define
\[
        \mathcal E_\Lambda^{\rm loc}
        =
        \left(
        \mathcal E_\Lambda^{\rm div},
        \mathcal E_\Lambda^{\rm mom},
        \mathcal E_\Lambda^{\rm prs},
        \mathcal E_\Lambda^{\rm led},
        \mathcal E_\Lambda^{\rm glue}
        \right)
\]
componentwise by
\begin{align}
        \mathcal E_{\Lambda,k}^{\rm div}(\mathfrak D_k)
        &:=
        \nabla\cdot U_k,\label{eq:loc-div-residual}\\
        \mathcal E_{\Lambda,k}^{\rm mom}(\mathfrak D_k)
        &:=
        \partial_sU_k-\Delta U_k
        +
        \nabla\cdot(U_k\otimes U_k+R_k)
        +
        \nabla P_k,\label{eq:loc-mom-residual}\\
        \mathcal E_{\Lambda,k}^{\rm prs}(\mathfrak D_k)
        &:=
        -\Delta P_k
        -
        \partial_i\partial_j
        (U_k^iU_k^j+R_k^{ij}).\label{eq:loc-prs-residual}
\end{align}
The ledger component is the vector
\[
        \mathcal E_{\Lambda,k}^{\rm led}
        =
        \left(
        \mathcal E_{\Lambda,k}^{\rm en},
        \mathcal E_{\Lambda,k}^{\rm cub},
        \mathcal E_{\Lambda,k}^{\rm prsdec},
        \mathcal E_{\Lambda,k}^{L},
        \mathcal E_{\Lambda,k}^{\Phi},
        \mathcal E_{\Lambda,k}^{\Pi}
        \right),
\]
where
\begin{align}
        \mathcal E_{\Lambda,k}^{\rm en}
        &:=
        A_{\Lambda,k,\vartheta}^{\rm loc}(U_k)
        +2E_{\Lambda,k,\vartheta}^{\rm loc}(U_k)
        -
        \vartheta^{-1}(L_k+\Phi_k+2\Pi_k)
        +s_k^{\rm en},\label{eq:loc-led-en-residual}\\
        \mathcal E_{\Lambda,k}^{\rm cub}
        &:=
        C_{\Lambda,k,\vartheta}^{\rm loc}(U_k)
        -
        C_{I,\vartheta}
        \left((\Phi_k+2\Pi_k)^{3/2}+L_k^{3/2}\right)
        +s_k^{\rm cub},\label{eq:loc-led-cub-residual}\\
        \mathcal E_{\Lambda,k}^{\rm prsdec}
        &:=
        D_{\Lambda,k,\vartheta}^{\rm loc}(P_k)
        -
        C_P\vartheta D_{\Lambda,k}^{\rm loc}(P_k)
        -
        C_P\vartheta^{-2}C_{\Lambda,k}^{\rm loc}(U_k)
        +s_k^{\rm prs},\label{eq:loc-led-prs-residual}\\
        \mathcal E_{\Lambda,k}^{L}
        &:=
        L_k-C_\vartheta A_{\Lambda,k}^{\rm loc}(U_k)+s_k^L,
        \label{eq:loc-led-L-residual}\\
        \mathcal E_{\Lambda,k}^{\Phi}
        &:=
        \Phi_k-C_\vartheta C_{\Lambda,k}^{\rm loc}(U_k)+s_k^\Phi,
        \label{eq:loc-led-Phi-residual}\\
        \mathcal E_{\Lambda,k}^{\Pi}
        &:=
        \Pi_k
        -
        C_\vartheta
        C_{\Lambda,k}^{\rm loc}(U_k)^{1/3}
        D_{\Lambda,k}^{\rm loc}(P_k)^{2/3}
        +s_k^\Pi .
        \label{eq:loc-led-Pi-residual}
\end{align}
For adjacent scales set
\begin{equation}\label{eq:loc-glue-residual}
        \mathcal E_{\Lambda,k}^{\rm glue}
        :=
        \mathfrak D_{k+1}
        -
        \mathfrak T_{\Lambda,k}^{\rm loc}(\mathfrak D_k).
\end{equation}

\begin{definition}[Localized NS residual norm]\label{def:localized-residual-norm}
The localized residual size is the product norm
\[
\begin{aligned}
        \|\mathcal E_\Lambda^{\rm loc}(\mathfrak D_\Lambda)\|
        &:=
        \Bigg(
        \sum_{k\in\Lambda_{\rm sc}}
        \Big[
        \|\mathcal E_{\Lambda,k}^{\rm div}\|^2
        +
        \|\mathcal E_{\Lambda,k}^{\rm mom}\|^2
        +
        \|\mathcal E_{\Lambda,k}^{\rm prs}\|^2
        +
        \|\mathcal E_{\Lambda,k}^{\rm led}\|^2
        \Big] \\
        &\quad+
        \sum_{\substack{k,k+1\in\Lambda_{\rm sc}}}
        \|\mathcal E_{\Lambda,k}^{\rm glue}\|^2
        \Bigg)^{1/2},
\end{aligned}
\]
where each term is measured in its chosen residual target norm.
\end{definition}

\begin{remark}[Meaning of zero localized residual]
The equality \(\mathcal E_\Lambda^{\rm loc}(\mathfrak D_\Lambda)=0\)
means exact compatibility with the projected localized finite-window model:
divergence, coarse momentum, pressure Poisson, ledger slack identities, and
the chosen scale-transfer rule all hold in the specified residual spaces.
It does not by itself prove that the package is generated by an exact
Navier--Stokes solution, nor that the projection, cutoff, or truncation
errors are small.
\end{remark}

\subsection{Localized cleaning space}\label{subsec:localized-cleaning-space}

The localized quotient distance must remove only bookkeeping artifacts, not
physical defect channels.  We therefore record the cleaning choices as part
of the finite-window data.  A \emph{localized cleaning datum} consists of
finite-dimensional Hilbert spaces
\[
        \mathcal C_{\Lambda}^{\rm cut},\qquad
        \mathcal C_{\Lambda}^{\rm harm},\qquad
        \mathcal C_{\Lambda}^{\rm cen},\qquad
        \mathcal C_{\Lambda}^{\rm cg},\qquad
        \mathcal C_{\Lambda}^{\rm tr},
\]
and bounded linear maps
\[
        G_{\Lambda}^{\rm cut},\quad
        G_{\Lambda}^{\rm harm},\quad
        G_{\Lambda}^{\rm cen},\quad
        G_{\Lambda}^{\rm cg},\quad
        G_{\Lambda}^{\rm tr}
        :
        \text{the corresponding cleaning space}
        \longrightarrow
        \mathcal D_\Lambda^{\rm loc}.
\]
The five components represent, respectively, cutoff deformation artifacts,
harmonic-pressure gauge artifacts, center or recentring artifacts,
coarse-graining convention artifacts, and finite-truncation artifacts \cite{ConstantinETiti1994,Eyink1994,DuchonRobert2000}.
The convention for harmonic pressure is fixed at this stage: components
placed in \(\mathcal C_{\Lambda}^{\rm harm}\) are quotiented as gauge, while
the remaining harmonic pressure tail is kept in the pressure error budget
\(\Err_{\rm prs,\Lambda}^{\rm loc}\).

Define
\[
        \mathcal C_\Lambda^{\rm loc}
        :=
        \mathcal C_{\Lambda}^{\rm cut}
        \oplus
        \mathcal C_{\Lambda}^{\rm harm}
        \oplus
        \mathcal C_{\Lambda}^{\rm cen}
        \oplus
        \mathcal C_{\Lambda}^{\rm cg}
        \oplus
        \mathcal C_{\Lambda}^{\rm tr}.
\]
For
\[
        a=(a_{\rm cut},a_{\rm harm},a_{\rm cen},a_{\rm cg},a_{\rm tr})
        \in\mathcal C_\Lambda^{\rm loc},
\]
set
\begin{equation}\label{eq:localized-cleaning-map}
\begin{aligned}
        G_\Lambda^{\rm loc}a
        &:=
        G_{\Lambda}^{\rm cut}a_{\rm cut}
        +
        G_{\Lambda}^{\rm harm}a_{\rm harm}
        +
        G_{\Lambda}^{\rm cen}a_{\rm cen}  \\
        &\quad+
        G_{\Lambda}^{\rm cg}a_{\rm cg}
        +
        G_{\Lambda}^{\rm tr}a_{\rm tr}.
\end{aligned}
\end{equation}
The localized quotient distance is
\begin{equation}\label{eq:localized-quotient-distance}
        \operatorname{dist}_{\rm loc}
        (\mathfrak D_\Lambda,\operatorname{Im}G_\Lambda^{\rm loc})
        :=
        \inf_{a\in\mathcal C_\Lambda^{\rm loc}}
        \left\|
        \mathfrak D_\Lambda-G_\Lambda^{\rm loc}a
        \right\|_{\mathcal D_\Lambda^{\rm loc}} .
\end{equation}

\begin{lemma}[Finite-dimensional localized cleaning quotient]\label{lem:localized-cleaning-quotient}
For every localized cleaning datum, the map
\[
        G_\Lambda^{\rm loc}:
        \mathcal C_\Lambda^{\rm loc}
        \longrightarrow
        \mathcal D_\Lambda^{\rm loc}
\]
is a bounded finite-dimensional linear map, and
\eqref{eq:localized-quotient-distance} defines a finite quotient distance
from any localized package to the chosen cleaning image.
\end{lemma}

\begin{proof}
The space \(\mathcal C_\Lambda^{\rm loc}\) is a finite direct sum of
finite-dimensional Hilbert spaces, hence finite-dimensional.  The formula
\eqref{eq:localized-cleaning-map} is a finite sum of bounded linear maps
into \(\mathcal D_\Lambda^{\rm loc}\), so \(G_\Lambda^{\rm loc}\) is a
bounded linear map.  Since \(0\in\mathcal C_\Lambda^{\rm loc}\), the
infimum in \eqref{eq:localized-quotient-distance} is bounded above by
\(\|\mathfrak D_\Lambda\|_{\mathcal D_\Lambda^{\rm loc}}\), and is
therefore finite.
\end{proof}

\begin{remark}[Status of the localized cleaning space]
Lemma~\ref{lem:localized-cleaning-quotient} only defines the quotient used
to measure non-gauge localized defects.  It does not prove that the listed
artifact spaces are complete, that no physical direction is quotiented out,
or that this quotient is comparable to the clean periodic quotient.  Those
are the later chart and quotient-distance comparison inputs.
\end{remark}

\subsection{Localized observability map}\label{subsec:localized-observability-map}

Next choose the localized observation channels.  A \emph{localized
observability datum} consists of finite-dimensional Hilbert spaces
\[
        \mathcal Y_{\Lambda}^{\rm prs,loc},\qquad
        \mathcal Y_{\Lambda}^{\rm flux,loc},\qquad
        \mathcal Y_{\Lambda}^{\rm en,loc},\qquad
        \mathcal Y_{\Lambda}^{\rm tr,loc},
\]
and continuous maps
\[
        O_{\Lambda}^{\rm prs,loc},\quad
        O_{\Lambda}^{\rm flux,loc},\quad
        O_{\Lambda}^{\rm en,loc},\quad
        O_{\Lambda}^{\rm tr,loc}
        :
        \mathcal D_\Lambda^{\rm loc}
        \longrightarrow
        \text{the corresponding observation space}.
\]
These channels record, respectively, pressure observations, nonlinear flux
observations, positive energy--dissipation observations, and selected trace
or adjoint observations.  Set
\[
        \mathcal Y_\Lambda^{\rm loc}
        :=
        \mathcal Y_{\Lambda}^{\rm prs,loc}
        \oplus
        \mathcal Y_{\Lambda}^{\rm flux,loc}
        \oplus
        \mathcal Y_{\Lambda}^{\rm en,loc}
        \oplus
        \mathcal Y_{\Lambda}^{\rm tr,loc}.
\]
The localized observability map is
\begin{equation}\label{eq:localized-observability-map}
        O_\Lambda^{\rm loc}\mathfrak D_\Lambda
        :=
        \left(
        O_{\Lambda}^{\rm prs,loc}\mathfrak D_\Lambda,\,
        O_{\Lambda}^{\rm flux,loc}\mathfrak D_\Lambda,\,
        O_{\Lambda}^{\rm en,loc}\mathfrak D_\Lambda,\,
        O_{\Lambda}^{\rm tr,loc}\mathfrak D_\Lambda
        \right).
\end{equation}
Its size is measured by the product norm
\begin{equation}\label{eq:localized-observability-norm}
\begin{aligned}
        \|O_\Lambda^{\rm loc}\mathfrak D_\Lambda\|_{\mathcal Y_\Lambda^{\rm loc}}
        &:=
        \|O_{\Lambda}^{\rm prs,loc}\mathfrak D_\Lambda\|
        +
        \|O_{\Lambda}^{\rm flux,loc}\mathfrak D_\Lambda\|  \\
        &\quad+
        \|O_{\Lambda}^{\rm en,loc}\mathfrak D_\Lambda\|
        +
        \|O_{\Lambda}^{\rm tr,loc}\mathfrak D_\Lambda\|.
\end{aligned}
\end{equation}

\begin{lemma}[Localized observation package]\label{lem:localized-observation-package}
For every localized observability datum, \(O_\Lambda^{\rm loc}\) is a
well-defined finite-dimensional observation map from
\(\mathcal D_\Lambda^{\rm loc}\) to \(\mathcal Y_\Lambda^{\rm loc}\), and
\eqref{eq:localized-observability-norm} defines a finite observation size
for each localized package.
\end{lemma}

\begin{proof}
Each component map is part of the localized observability datum and has
finite-dimensional target.  The product formula
\eqref{eq:localized-observability-map} therefore defines a map into
\(\mathcal Y_\Lambda^{\rm loc}\).  Since all target norms are finite on
finite-dimensional spaces, the sum in
\eqref{eq:localized-observability-norm} is finite for every
\(\mathfrak D_\Lambda\in\mathcal D_\Lambda^{\rm loc}\).
\end{proof}

\begin{remark}[Observability versus residual]
The map \(O_\Lambda^{\rm loc}\) is a visibility or tax map.  It is kept
separate from the residual map \(\mathcal E_\Lambda^{\rm loc}\), whose role
is finite-window Navier--Stokes compatibility.  A scalar coordinate may be
used in both places only with these two different meanings: as a measured
observable in \(O_\Lambda^{\rm loc}\), and as a compatibility equation in
\(\mathcal E_\Lambda^{\rm loc}\).
\end{remark}

\subsection{Localized reproduction residual}\label{subsec:localized-reproduction-residual}

The dynamic formulation also records whether the localized package
reproduces from scale to scale.  For each adjacent pair
\[
        k,k+1\in\Lambda_{\rm sc},
\]
choose a finite-dimensional drift space
\(\mathcal Z_{\Lambda,k}^{\rm drift,loc}\) and a continuous drift map
\[
        \operatorname{Drift}_{\Lambda,k}^{\rm loc}:
        \mathcal D_{\Lambda,k}^{\rm loc}
        \times
        \mathcal D_{\Lambda,k+1}^{\rm loc}
        \longrightarrow
        \mathcal Z_{\Lambda,k}^{\rm drift,loc}.
\]
This drift records the change of cutoff, center, pressure tail, scale
normalization, or projection convention that prevents the localized
relation from being an exact clean reproduction law.

Using the transfer maps \(\mathfrak T_{\Lambda,k}^{\rm loc}\) already fixed
in the residual datum, define the one-step localized reproduction defect by
\[
        \operatorname{rep}_{\Lambda,k}^{\rm loc}
        (\mathfrak D_k,\mathfrak D_{k+1})
        :=
        \left(
        \left\|
        \mathfrak D_{k+1}
        -
        \mathfrak T_{\Lambda,k}^{\rm loc}(\mathfrak D_k)
        \right\|_{\mathcal D_{\Lambda,k+1}^{\rm loc}}^2
        +
        \left\|
        \operatorname{Drift}_{\Lambda,k}^{\rm loc}
        (\mathfrak D_k,\mathfrak D_{k+1})
        \right\|^2
        \right)^{1/2}.
\]
The finite-window reproduction residual is
\begin{equation}\label{eq:localized-reproduction-residual}
        \operatorname{Rep}_\Lambda^{\rm loc}
        (\mathfrak D_\Lambda)
        :=
        \left(
        \sum_{\substack{k,k+1\in\Lambda_{\rm sc}}}
        \left[
        \operatorname{rep}_{\Lambda,k}^{\rm loc}
        (\mathfrak D_k,\mathfrak D_{k+1})
        \right]^2
        \right)^{1/2}.
\end{equation}

\begin{lemma}[Localized reproduction residual is well-defined]\label{lem:localized-reproduction-residual}
For every localized reproduction datum, the quantity
\(\operatorname{Rep}_\Lambda^{\rm loc}(\mathfrak D_\Lambda)\) in
\eqref{eq:localized-reproduction-residual} is finite for every
\(\mathfrak D_\Lambda\in\mathcal D_\Lambda^{\rm loc}\).
\end{lemma}

\begin{proof}
There are only finitely many adjacent scale pairs in
\(\Lambda_{\rm sc}\).  For each pair, both terms in
\(\operatorname{rep}_{\Lambda,k}^{\rm loc}\) lie in finite-dimensional
normed spaces and are finite by the definitions of
\(\mathfrak T_{\Lambda,k}^{\rm loc}\) and
\(\operatorname{Drift}_{\Lambda,k}^{\rm loc}\).  The finite sum in
\eqref{eq:localized-reproduction-residual} is therefore finite.
\end{proof}

\begin{remark}[Status of reproduction]
Lemma~\ref{lem:localized-reproduction-residual} defines a detector for
failure of scale-to-scale reproduction.  It does not prove that
Navier--Stokes-derived packages have small reproduction residual, nor that
the drift term is perturbative.  Large drift is kept as part of the
localized mechanism or the later error budget, not discarded.
\end{remark}

\subsection{Localized ledger profit functional}\label{subsec:localized-ledger-profit}

The last component of the localized detection package is the net ledger
profit.  A \emph{localized profit datum} consists, for each
\(k\in\Lambda_{\rm sc}\), of continuous scalar functionals
\[
        \mathsf{Supp}_{\Lambda,k}^{\rm loc},\qquad
        \mathsf{Tax}_{\Lambda,k}^{\rm loc},\qquad
        \mathsf{LeakCost}_{\Lambda,k}^{\rm loc},\qquad
        \mathsf{ErrCost}_{\Lambda,k}^{\rm loc}
        :
        \mathcal D_{\Lambda,k}^{\rm loc}\to\R .
\]
The sign convention is the same as in \eqref{eq:net-profit}: supply enters
with positive sign, while tax, leakage cost, and bookkeeping error cost
enter with negative sign.  Define the one-step localized profit by
\[
        \mathsf P_{\Lambda,k}^{\rm loc}(\mathfrak D_k)
        :=
        \mathsf{Supp}_{\Lambda,k}^{\rm loc}(\mathfrak D_k)
        -
        \mathsf{Tax}_{\Lambda,k}^{\rm loc}(\mathfrak D_k)
        -
        \mathsf{LeakCost}_{\Lambda,k}^{\rm loc}(\mathfrak D_k)
        -
        \mathsf{ErrCost}_{\Lambda,k}^{\rm loc}(\mathfrak D_k).
\]
The cumulative localized ledger profit is
\begin{equation}\label{eq:localized-ledger-profit}
        \mathsf P_\Lambda^{\rm loc}(\mathfrak D_\Lambda)
        :=
        \sum_{k\in\Lambda_{\rm sc}}
        \mathsf P_{\Lambda,k}^{\rm loc}(\mathfrak D_k),
        \qquad
        [\mathsf P_\Lambda^{\rm loc}(\mathfrak D_\Lambda)]_+
        :=
        \max\{
        \mathsf P_\Lambda^{\rm loc}(\mathfrak D_\Lambda),0
        \}.
\end{equation}
If a later transfer argument keeps leakage or localization error in
\(\Delta_\Lambda\) rather than charging it in the profit channel, this
convention is implemented by setting the corresponding cost functional to
zero and recording the term in the error budget instead.

\begin{lemma}[Localized profit is well-defined]\label{lem:localized-ledger-profit}
For every localized profit datum, the quantities
\(\mathsf P_\Lambda^{\rm loc}(\mathfrak D_\Lambda)\) and
\([\mathsf P_\Lambda^{\rm loc}(\mathfrak D_\Lambda)]_+\) in
\eqref{eq:localized-ledger-profit} are finite for every
\(\mathfrak D_\Lambda\in\mathcal D_\Lambda^{\rm loc}\).
\end{lemma}

\begin{proof}
The window \(\Lambda_{\rm sc}\) is finite, and each one-step profit
functional is finite-valued on the finite-dimensional coordinate space
\(\mathcal D_{\Lambda,k}^{\rm loc}\).  Therefore the cumulative sum is
finite.  The positive part is the maximum of two finite real numbers.
\end{proof}

\begin{remark}[Profit status]
The functional \(\mathsf P_\Lambda^{\rm loc}\) records a sign convention
and a bookkeeping channel.  This step does not prove that positive profit
exists, nor that profit is comparable with the clean periodic profit under
a local-to-clean chart.  The later profit comparison uses only the
\(1\)-Lipschitz property of \(x\mapsto[x]_+\) after the two profit
functionals have been aligned.
\end{remark}

At an abstract scale $k$, a finite defect system has the schematic form
\begin{equation}\label{eq:defect-complex}
    \calC_k \xrightarrow{G_k} \calD_k \xrightarrow{O_k} \calY_k,
\end{equation}
with $O_kG_k=0$.  Here $\calD_k$ is a defect space, $\calC_k$ is a cleaning or gauge space, and $\calY_k$ is an observable space.  The finite-scale obstruction group is
\begin{equation}\label{eq:H1k}
    H_k^1:=\frac{\ker O_k}{\Image G_k}.
\end{equation}
The condition $H_k^1=0$ says that every defect invisible to the chosen observables is removable by the chosen cleaning mechanism.

For the critical ledger, the finite defect package is not an arbitrary algebraic object.  It is produced by a suitable weak Navier--Stokes solution through the definitions above.  A concrete dictionary is:
\begin{center}
\begin{tabular}{@{}ll@{}}
\toprule
Abstract defect language & Critical ledger realization \\
\midrule
Defect data $\calD_k$ & $(A_k,C_k,D_k;E_{k+1};\Phi_k,\Pi_k,\Lambda_k)$ \\
Observable map $O_k$ & energy, pressure, flux, dissipation, pressure-decay tests \\
Cleaning map $G_k$ & cutoff, harmonic-pressure gauge, localization, coarse-graining artifacts \\
Transition map $\rho_{k+1,k}$ & moving-window transfer $Q_{k+1}\subset Q_k$ and scale projection \\
Finite obstruction $H_k^1$ & invisible non-gauge supply or leakage at scale $k$ \\
\bottomrule
\end{tabular}
\end{center}

The key point is that the ledger coordinates are already NS-realizable: they come from $(u,p)$, an admissible scale-window chain, and cutoffs adapted to that chain.  Thus the ledger gives a finite-scale PDE model for the more abstract question of whether there exist invisible critical defects.

\begin{definition}[Ledger-realizable defect package]
A finite-scale defect package at scale $k$ is ledger-realizable if there exist a suitable weak solution $(u,p)$, an admissible scale-window transition $Q_k\to Q_{k+1}$, and an adapted cutoff $\phi_k$ such that the package is exactly
\[
    (A_k,C_k,D_k;E_{k+1};\Phi_k,\Pi_k,\Lambda_k)
\]
with all terms defined as above.
\end{definition}

\begin{remark}[Why the ledger comes first]
The abstract complex \eqref{eq:defect-complex} is useful only after the PDE coordinates have been fixed.  Without the local energy ledger, $\calD_k$, $O_k$, and $G_k$ risk becoming formal placeholders.  The finite-scale theorem supplies concrete quantities and a concrete inequality that any later defect-complex formulation must respect.
\end{remark}

\section{Profitable reproducible verified mechanisms}\label{sec:dynamic-mechanism}
\noindent This section upgrades the static defect-complex language into a dynamic mechanism language: a dangerous defect must reproduce, remain profitable, and satisfy Navier--Stokes residual tests.

The finite defect complex
\[
    \calC_k\xrightarrow{G_k}\calD_k\xrightarrow{O_k}\calY_k
\]
answers a static question: which defects at scale $k$ are invisible to the chosen observables modulo cleanings?  This is necessary, but not sufficient, for a blow-up mechanism.  A potential singularity is a dynamic object.  It must regenerate badness from one scale to another, it must remain energetically profitable after the available taxes are paid, and it must be realizable by the Navier--Stokes equations rather than by a formal algebraic complex.

The useful upgrade is therefore
\begin{equation}\label{eq:dynamic-five-tuple}
    \boxed{
    (\mathfrak D_k,\;\calR_k,\;\calL_k,\;O_k,\;\calE_k).
    }
\end{equation}
Here $\mathfrak D_k$ is the scale-$k$ defect package, $\calR_k$ is a reproduction law, $\calL_k$ is the critical ledger or profitability functional, $O_k$ is the observability/tax map, and $\calE_k$ is the NS-realizability residual.  The slogan becomes
\[
\boxed{\begin{gathered}
\text{potential singularity}\\
=\text{profitable reproducible NS-realizable defect cascade}\\
\text{invisible to the relevant taxes.}
\end{gathered}}
\]
This is a more restrictive object than a sequence of large critical norms and also more restrictive than a formal element of $\ker O_\infty/\Image G_\infty$.

The three external mechanisms used here as motivation are the dyadic Leray--Hopf non-uniqueness mechanism in \cite{PalasekDyadic2024}, the inverse-cascade norm-growth mechanism in \cite{PalasekNormGrowth2025}, and the ASSF self-similar-profile validation program in \cite{IonescuJiaPalasek2026}; these are also natural stress tests for critical-space well-posedness thresholds such as Koch--Tataru's small-data theory \cite{KochTataru2001}.  The present section does not import their theorems into the finite-energy CKN setting; it translates their mechanism-level lessons into the ledger/defect vocabulary of this manuscript.

\subsection*{The reproduction map}

Let $\calD_k^{\NS}$ denote the class of scale-$k$ defect packages that pass the basic Navier--Stokes structural filters: incompressibility, pressure compatibility, scale-critical normalization, and local energy admissibility.  A reproduction law is a map, or more generally a branch relation,
\begin{equation}\label{eq:reproduction-map}
    \calR_k:\calD_k^{\NS}\rightrightarrows \calD_{k+1}^{\NS}.
\end{equation}
The reproduction residual of a proposed two-scale segment is
\begin{equation}\label{eq:reproduction-residual}
    \Rep_k(\mathfrak D_k,\mathfrak D_{k+1})
    :=
    \Dist\bigl(\mathfrak D_{k+1},\calR_k(\mathfrak D_k)\bigr).
\end{equation}
The dyadic mechanism of Palasek \cite{PalasekDyadic2024} is naturally read in this language.  Its main lesson for the present manuscript is not the specific Obukhov model, but the fact that a scale-critical defect may come with an approximately discretely self-similar reproduction law whose effective dynamics are weakly coupled and finite-dimensional.  Partial breaking of scaling symmetry should be encoded as the possibility that $\calR_k$ is not single-valued: at a scale one may have different heat-taxed and inverse-supply branches,
\[
    \calR_k
    =
    \calR_k^{\rm heat}\cup \calR_k^{\rm inv}\cup \cdots .
\]
Thus the static assertion
\[
    \text{there is an invisible defect at every scale}
\]
should be replaced by the dynamic assertion
\[
    \text{there is an invisible non-gauge defect that reproduces itself across scales.}
\]
A candidate blow-up or non-uniqueness mechanism must therefore provide a compatible branch selection
\begin{equation}\label{eq:branch-selection}
    \mathfrak D_{k+1}\approx \calR_k^{b_k}(\mathfrak D_k),
    \qquad b_k\in\{\rm heat,inv,\ldots\},
\end{equation}
not merely a list of unrelated finite-scale shadows.

\subsection*{The profitability ledger}

The norm-growth mechanism in \cite{PalasekNormGrowth2025} suggests a second correction to the language.  One should not ask only whether a defect survives; one should ask whether the defect is profitable after all taxes are charged.  In the notation of the ledger theorem, define the net critical profit at a transition by
\begin{equation}\label{eq:net-profit}
    \calL_k(\mathfrak D_k)
    :=
    \Supp_k^{\rm full}
    -\Tax_k^{\rm full}
    -\Leak_k^{\rm full}
    -\Err_k.
\end{equation}
When the leakage and bookkeeping errors are kept separate, one may instead use
\[
    \Prof_k:=\Supp_k^{\rm full}-\Tax_k^{\rm full}
\]
and view $\Leak_k^{\rm full}+\Err_k$ as the cost of realizing the finite window.  In either convention, the obstruction is not simply
\[
    \mathfrak D_k\in \ker O_k,
\]
but rather
\begin{equation}\label{eq:profitable-invisible}
    \mathfrak D_k\in \ker O_k\quad\text{modulo cleanings,}
    \qquad
    \calL_k(\mathfrak D_k)>0
\end{equation}
on a positive-density set of scales.  The pressure, flux, energy, and trace observables should therefore be regarded as tax channels:
\begin{equation}\label{eq:tax-observable-map}
    O_k
    =
    \bigl(O_k^P,O_k^F,O_k^E,O_k^T\bigr),
\end{equation}
where the superscripts denote pressure, flux, energy/dissipation, and trace tests.  A genuinely dangerous cascade is one for which the nonlinear supply repeatedly avoids or overwhelms these taxes.  This is the ledger translation of an inverse cascade: the nonlinear term is not an error to be hidden, but a positive supply channel that can refill the critical reservoir faster than the expected decay drains it.

Combining this with Theorem~\ref{thm:abstract-survival}, an infinite bad trajectory with negligible average leakage would force
\begin{equation}\label{eq:average-profitable-scales}
    \liminf_{N\to\infty}
    \frac1N\sum_{k=0}^{N-1}
    \pos{\Supp_k^{\rm full}-\Tax_k^{\rm full}}
    >0.
\end{equation}
Thus any proposed counterexample mechanism must explain the positive-density source of this profit.  Conversely, a regularity route must prove that every NS-realizable supply is taxed, or that any apparent profit is a localization/gauge artifact that vanishes after gluing.

\subsection*{The verification residual}

The self-similar-profile program in \cite{IonescuJiaPalasek2026} suggests the third correction: a defect mechanism should be searchable and certifiable.  The analogue of a self-similar profile equation $\calF(U)=0$ in the present language is a finite-window compatibility equation
\begin{equation}\label{eq:defect-compatibility-equation}
    \calF_\Lambda(\mathfrak D_\Lambda)=0,
\end{equation}
where $\Lambda=\{k_0,\ldots,k_1\}$ and $\mathfrak D_\Lambda=(\mathfrak D_k)_{k\in\Lambda}$.  The residual vector should include at least
\begin{equation}\label{eq:residual-vector}
    \calE_\Lambda(\mathfrak D_\Lambda)
    =
    \begin{pmatrix}
    \calE_\Lambda^{\rm mom}\\
    \calE_\Lambda^{\rm press}\\
    \calE_\Lambda^{\rm energy}\\
    \calE_\Lambda^{\rm glue}\\
    \calE_\Lambda^{\rm gauge}
    \end{pmatrix},
\end{equation}
where these components measure coarse momentum residual, pressure Poisson residual, local energy/flux residual, scale-gluing residual, and cleaning/gauge residual respectively.  Schematic representatives are
\begin{align}
    \calE_k^{\rm mom}
    &:=
    \partial_t U_k-\Delta U_k+
    \mathbb P\nabla\cdot(U_k\otimes U_k+R_k)-f_k^{\rm ext},\label{eq:mom-res}\\
    \calE_k^{\rm press}
    &:=
    -\Delta P_k^{\rm def}-\partial_i\partial_j
    (U_k^iU_k^j+R_k^{ij}),\label{eq:press-res}\\
    \calE_k^{\rm glue}
    &:=
    \mathfrak D_{k+1}-\calR_k(\mathfrak D_k),\label{eq:glue-res}
\end{align}
with the understanding that the actual Banach norms and cutoff conventions depend on the chosen finite-window model.  The role of computation is then not to replace proof, but to find a high-precision candidate $\mathfrak D_\Lambda^0$ and to reduce the remaining proof to validated perturbation estimates.

A typical verification step has the following form.  Let
\[
    \calF_\Lambda:\calX_\Lambda\to\calY_\Lambda
\]
be the nonlinear finite-window residual map and suppose $\mathfrak D_\Lambda^0$ is a numerical or symbolic candidate.  One seeks $h$ such that
\[
    \calF_\Lambda(\mathfrak D_\Lambda^0+h)=0.
\]
Expanding gives
\[
    \calF_\Lambda(\mathfrak D_\Lambda^0+h)
    =
    \calF_\Lambda(\mathfrak D_\Lambda^0)
    +D\calF_\Lambda(\mathfrak D_\Lambda^0)h
    +\N_\Lambda(h).
\]
If one can rigorously bound
\begin{equation}\label{eq:newton-bounds}
    \epsilon:=\|\calF_\Lambda(\mathfrak D_\Lambda^0)\|_{\calY_\Lambda},
    \qquad
    M:=\|D\calF_\Lambda(\mathfrak D_\Lambda^0)^{-1}\|_{\calY_\Lambda\to\calX_\Lambda},
\end{equation}
and prove a quadratic Lipschitz estimate
\begin{equation}\label{eq:newton-nonlinear-bound}
    \|\N_\Lambda(h_1)-\N_\Lambda(h_2)\|_{\calY_\Lambda}
    \leq
    Lr\|h_1-h_2\|_{\calX_\Lambda}
    \qquad \text{for }\|h_i\|_{\calX_\Lambda}\le r,
\end{equation}
with, for instance, $M\epsilon\le r/2$ and $MLr<1/2$, then a Newton--Kantorovich argument produces an exact nearby defect mechanism.  A regularity-oriented version of the same procedure proves the opposite statement: a certified lower bound for the quotient gap of $O_\Lambda$ on the verified residual set excludes an invisible non-gauge candidate in that window.

\begin{definition}[Profitable reproducible verified finite-window mechanism]\label{def:prv-mechanism}
Fix a finite scale window $\Lambda=\{k_0,\ldots,k_1\}$.  A family $\mathfrak D_\Lambda=(\mathfrak D_k)_{k\in\Lambda}$ is an $(\eta,\beta)$-profitable reproducible verified mechanism, or $(\eta,\beta)$-PRV mechanism, if
\begin{align}
    &\Dist(\mathfrak D_k,\Image G_k)\ge 1,
    \qquad k\in\Lambda,\label{eq:prv-nongauge}\\
    &\|O_k\mathfrak D_k\|\le \eta,
    \qquad k\in\Lambda,\label{eq:prv-invisible}\\
    &\|\calE_k(\mathfrak D_k)\|\le \eta,
    \qquad k\in\Lambda,\label{eq:prv-residual}\\
    &\Rep_k(\mathfrak D_k,\mathfrak D_{k+1})\le \eta,
    \qquad k,k+1\in\Lambda,\label{eq:prv-reproduction}\\
    &\sum_{k\in\Lambda}\pos{\calL_k(\mathfrak D_k)}
    \ge \beta |\Lambda|.\label{eq:prv-profit}
\end{align}
If the inequalities hold with $\eta=0$ and the reproduction relation is exact, we call it an exact PRV mechanism.
\end{definition}

With this definition, the finite-energy Clay-adjacent obstruction isolated by this manuscript can be restated as follows:
\[
\boxed{
\text{Can there exist arbitrarily deep NS-realizable PRV mechanisms at CKN scale?}
}
\]
The anti-phantom route tries to prove that the answer is no.  A constructive route tries to produce such mechanisms, first in dyadic or shell models, then in controlled finite windows of the full PDE, and finally with a rigorous residual verifier.

\begin{problem}[Dynamic anti-phantom estimate]\label{prob:dynamic-antiphantom}
Prove, for an appropriate NS-realizable finite-window class, constants $c>0$ and $C<\infty$ such that every normalized candidate satisfies
\begin{equation}\label{eq:dynamic-antiphantom}
\begin{aligned}
    &\|O_\Lambda\mathfrak D_\Lambda\|
    + C\|\calE_\Lambda(\mathfrak D_\Lambda)\|
    + C\Rep_\Lambda(\mathfrak D_\Lambda)\\
    &\qquad\ge
    c\,\Dist(\mathfrak D_\Lambda,\Image G_\Lambda)
    \quad\text{unless}\quad
    \frac1{|\Lambda|}\sum_{k\in\Lambda}\pos{\calL_k(\mathfrak D_k)}
    \text{ is large.}
\end{aligned}
\end{equation}
A precise version should replace the final phrase by an explicit ledger lower bound.  Such an estimate would say: invisibility, small residual, and reproducibility are incompatible with being a non-gauge defect unless the ledger records genuine untaxed profit.
\end{problem}

The conceptual shift is therefore:
\[
    \boxed{
    \text{invisible defect}
    \quad\leadsto\quad
    \text{profitable reproducible verified defect mechanism}.
    }
\]
This is the scale-defect analogue of moving from a formal approximate profile to a rigorously validated unstable profile: the obstruction becomes searchable, falsifiable, and eventually certifiable.

\section{Local-to-global gluing and derived phantoms}\label{sec:local-to-global}
\noindent This section explains why finite-scale exactness alone does not imply a global anti-phantom principle.

A single finite-scale defect is not a blow-up mechanism.  A blow-up mechanism must persist through a compatible sequence of scales.  Suppose therefore that the finite systems \eqref{eq:defect-complex} are connected by transition maps
\[
    \rho_{k+1,k}:\calD_{k+1}\to\calD_k,
    \qquad
    \sigma_{k+1,k}:\calC_{k+1}\to\calC_k,
    \qquad
    \tau_{k+1,k}:\calY_{k+1}\to\calY_k,
\]
which satisfy the functorial identities
\begin{equation}\label{eq:functorial-clean}
    \rho_{k+1,k}G_{k+1}=G_k\sigma_{k+1,k},
\end{equation}
\begin{equation}\label{eq:functorial-observe}
    O_k\rho_{k+1,k}=\tau_{k+1,k}O_{k+1}.
\end{equation}
A compatible defect cascade is then an element
\[
    d=(d_0,d_1,d_2,\ldots)\in \varprojlim_k \calD_k,
    \qquad \rho_{k+1,k}d_{k+1}=d_k.
\]
The global invisible space is $\ker O_\infty$, and the globally removable space is $\Image G_\infty$.  The global critical defect cohomology is
\begin{equation}\label{eq:Hcrit}
    H^1_{\crit,\NS}:=\frac{\ker O_\infty}{\Image G_\infty}.
\end{equation}

Finite exactness alone does not imply global exactness.  Even if
\[
    \ker O_k=\Image G_k \qquad \text{for every finite } k,
\]
each local cleaning may depend on $k$ in a way that fails to glue across scales.  In other words, one may have local cleanability at every finite scale but no compatible family of cleanings in the inverse limit.

\begin{definition}[Asymptotic and derived phantom]
An \emph{asymptotic phantom} is a normalized sequence of finite-scale defect candidates $d_k\in\calD_k$ for which
\[
    \Dist(d_k,\Image G_k)\simeq 1,
    \qquad
    \|O_kd_k\|\to0
\]
along an increasing scale/window sequence.  A \emph{derived phantom} is a compatible invisible cascade whose every finite component is locally cleanable, but whose cleanings do not form a compatible global cleaning.
\end{definition}

A standard algebraic way to express the second obstruction is through derived inverse limits \cite{MacLane1963,Weibel1994}.  Let
\[
    B_k=\Image G_k,
    \qquad
    Z_k=\ker O_k,
    \qquad
    K_k=\ker G_k.
\]
From compatible short exact systems
\[
    0\to K_k\to \calC_k\xrightarrow{G_k}B_k\to0,
\]
one obtains derived inverse-limit terms.  The relevant obstruction is the possible failure of
\begin{equation}\label{eq:gluing-surj}
    \varprojlim_k \calC_k\longrightarrow \varprojlim_k B_k
\end{equation}
to be surjective; this failure is detected, under standard hypotheses, by $\limone K_k$-type terms.

\begin{theorem}[Conditional abstract local-to-global exactness criterion]\label{thm:local-to-global}
Assume that the finite-scale complexes
\[
    \calC_k\xrightarrow{G_k}\calD_k\xrightarrow{O_k}\calY_k
\]
form compatible inverse systems satisfying \(O_kG_k=0\),
\eqref{eq:functorial-clean}, and \eqref{eq:functorial-observe}.  Let
$Z_k=\ker O_k$ and $B_k=\Image G_k$.  Suppose:
\begin{enumerate}[label=(\roman*),leftmargin=2em]
    \item finite anti-phantom exactness holds: $Z_k=B_k$ for every $k$;
    \item global cleaning gluing holds: $\varprojlim_k \calC_k\to \varprojlim_k B_k$ is surjective.
\end{enumerate}
Then the global defect complex is exact at $\varprojlim_k\calD_k$:
\[
    \ker O_\infty=\Image G_\infty.
\]
Equivalently, $H^1_{\crit,\NS}=0$.
\end{theorem}

\begin{proof}
Let $d=(d_k)\in\ker O_\infty$.  Then $O_kd_k=0$ for every $k$, so $d_k\in Z_k$.  By finite exactness, $Z_k=B_k$, hence $d_k\in B_k$ for every $k$.  Compatibility of $d$ gives $(d_k)\in\varprojlim_k B_k$.  By the gluing assumption, there exists $c=(c_k)\in\varprojlim_k\calC_k$ such that $G_kc_k=d_k$ for every $k$.  Therefore $G_\infty c=d$, so $d\in\Image G_\infty$.  The inclusion $\Image G_\infty\subset\ker O_\infty$ follows from $O_kG_k=0$ at every scale.
\end{proof}

\begin{remark}[Interpretation for the ledger]
The critical ledger theorem says that long survival of $\Bbad_k$ forces cumulative untaxed supply or leakage.  The local-to-global theorem says what would be needed to upgrade finite-scale control into a global anti-phantom principle: not only must finite supply defects be observable or cleanable, but the corresponding cleanings must remain compatible under the scale transitions.
\end{remark}

\section{Regularity, counterexamples, and branching obstructions}\label{sec:regularity-counterexample}
\noindent This section places the ledger and defect languages on the two natural routes: a regularity route by uniform taxation and a counterexample route by reproducible untaxed supply.

The ledger and defect languages point to two complementary strategies.

\subsection*{Regularity route}

A regularity proof in this language would require proving that every NS-realizable critical supply is uniformly taxed, up to negligible leakage, and that local cleanings glue across scales.  Schematically,
\[
\boxed{
\begin{gathered}
\text{uniform taxation}+\text{negligible leakage}+\text{no derived gluing obstruction}\\
\Longrightarrow\text{ no sustained bad orbit}
\end{gathered}}
\]
Together with the Caffarelli--Kohn--Nirenberg smallness criterion \cite{CKN1982,Lin1998,SereginLectureNotes}, this would rule out a potential singularity along the chosen scale-window chain.

\subsection*{Counterexample route}

A counterexample mechanism would have to do more than produce a large scale-invariant norm.  It would need to produce an NS-realizable sustained bad orbit with repeated untaxed supply, controlled leakage, a scale-to-scale reproduction law, and a residual verifier.  In the defect language, it would need to generate a nonzero class in $H^1_{\crit,\NS}$, either as an asymptotic phantom whose observability constants degenerate, or as a derived phantom whose finite cleanings fail to glue.

Thus the obstruction isolated by this manuscript is:
\[
\boxed{
\text{a profitable, reproducible, NS-realizable, scale-critical moving-window defect cascade.}
}
\]
This object is more precise than the vague statement that ``a critical quantity fails to decay.''  It must be produced by Navier--Stokes dynamics, compatible with pressure, flux, dissipation, local energy, and scale transitions, and not removable as a cutoff or harmonic-pressure artifact.

\subsection*{A critical branching defect outside the finite-energy ledger}

The non-uniqueness theorem of Coiculescu--Palasek \cite{CoiculescuPalasek2025} gives a useful stress test for the language above.  It does not produce a finite-energy Leray--Hopf counterexample, and it is not an example inside the suitable-weak CKN ledger developed above.  The initial datum in their construction lies in the critical space $BMO^{-1}$ but not in $L^2$, so the correct translation is into a Koch--Tataru critical ledger \cite{KochTataru2001} rather than into the local-energy reservoir $(A_k,C_k,D_k)$.

In this terminology, their result constructs a $BMO^{-1}$-critical NS-realizable branching defect cascade.  There is a divergence-free critical trace
\[
    U_0\in BMO^{-1}(\mathbb T^3),
\]
assembled from lacunary scale packets
\[
    U_0=\sum_{k\ge0}V_k^0,
    \qquad |\xi|\sim N_k,
    \qquad N_0\ll N_1\ll N_2\ll\cdots .
\]
Each packet admits two compatible local transition mechanisms:
\[
    \mathsf H_k=\text{heat-dominated branch},
    \qquad
    \mathsf I_k=\text{inverse-cascade-dominated branch}.
\]
The branch $\mathsf H_k$ is the heat-taxed continuation: the $k$th packet is depleted on the heat time scale $N_k^{-2}$.  The branch $\mathsf I_k$ is an active nonlinear supply continuation: the self-interaction of the $(k+1)$st packet produces, after Leray projection and frequency localization, a non-perturbative forcing at the $k$th frequency shell, schematically
\[
    \Supp^{\rm active}_{k+1\to k}
    \sim
    P_{\sim N_k}\mathbb P\operatorname{div}(v_{k+1}\otimes v_{k+1}),
\]
which annihilates the lower packet on the shorter time scale $N_{k+1}^{-2}$.  Thus $\mathsf I_k$ is not merely a leakage term in an energy inequality; it is an active scale-to-scale supply channel.

The local alternatives cannot be glued arbitrarily.  They admit two global compatible branch selections,
\[
    b^{(1)}=(\mathsf H_0,\mathsf I_1,\mathsf H_2,\mathsf I_3,\ldots),
    \qquad
    b^{(2)}=(\mathsf I_0,\mathsf H_1,\mathsf I_2,\mathsf H_3,\ldots),
\]
which generate two approximate critical cascades.  The residuals of both cascades are cleanable by perturbative corrections $w^{(i)}$, yielding exact Navier--Stokes solutions
\[
    u^{(i)}=v^{(i)}+w^{(i)},
    \qquad i=1,2,
\]
which are smooth for every positive time and lie in the Koch--Tataru critical path class.  They have the same critical trace but different positive-time observables:
\[
    \operatorname{Tr}_{t=0}u^{(1)}
    =
    \operatorname{Tr}_{t=0}u^{(2)}
    =U_0,
    \qquad
    O_{t>0}(u^{(1)})\ne O_{t>0}(u^{(2)}).
\]
Equivalently,
\[
    \#\operatorname{Cont}_{KT}(U_0)\ge2,
\]
where $\operatorname{Cont}_{KT}(U_0)$ denotes the set of Koch--Tataru-critical NS-realizable compatible continuations with trace $U_0$, modulo residual exactification.

This example is not a Clay-scale singularity mechanism, but it is an important model obstruction for the present framework.  It shows that in a genuine critical Navier--Stokes setting, a trace may support more than one scale-compatible global continuation.  In the defect language, the obstruction is not failure of finite residual cleaning; both branches are cleanable.  The obstruction is a nontrivial fiber of the global critical continuation functor over one critical trace.  Thus the phrase ``critical defect cascade'' should be read dynamically: a serious obstruction may be a scale-compatible nonlinear branch selection, not merely a large norm or a static invisible residual.

This suggests a refined target for the finite-energy ledger program.  One should not only seek an anti-phantom estimate
\[
    \ker O_\Lambda\cap (\Image G_\Lambda)^\perp=\{0\},
\]
but also an anti-branching principle: in the suitable-weak local-energy setting, heat-taxed and inverse-supply alternatives should not glue into two distinct compatible continuations with the same finite-energy trace, unless the ledger records persistent untaxed supply or non-negligible leakage.  Put differently, the CKN ledger should aim to rule out Coiculescu--Palasek-type critical branching cascades inside the finite-energy, pressure-compatible, local-energy class.

\part{Finite-Window Tests and Conditional Program}
\section{Verifier-driven finite-window obstruction search}\label{sec:verifier-search}
\noindent This section is programmatic.  It explains how the finite-scale theorem suggests finite-window searches whose outputs must be certified by mathematical verifiers.

The scale-defect interpretation also suggests a computational and AI-assisted research workflow.  The useful task is not to ask a model to solve Navier--Stokes directly.  The useful task is to search for finite-window obstruction candidates and pass them through hard mathematical verifiers.

At a fixed finite window, one constructs matrices or operators representing
\[
    G_\Lambda:\calC_\Lambda\to\calD_\Lambda,
    \qquad
    O_\Lambda:\calD_\Lambda\to\calY_\Lambda.
\]
The basic anti-phantom constant is
\begin{equation}\label{eq:mu-Lambda}
    \mu_\Lambda
    :=
    \inf_{\Dist(d,\Image G_\Lambda)=1}\|O_\Lambda d\|.
\end{equation}
If $\mu_\Lambda>0$, the finite window has no exact invisible non-gauge defect.  If $\mu_\Lambda$ is small, the corresponding singular vector is a near-phantom candidate.

The scale-uniform question is whether the constants remain bounded as windows enlarge:
\[
    \sup_\Lambda \mu_\Lambda^{-1}<\infty.
\]
Degeneration of these constants is a finite-window signature of an asymptotic phantom.  To test gluing failure, one searches for finite chains $(d_1,\ldots,d_N)$ with small loss
\begin{equation}\label{eq:finite-chain-loss}
    \mathcal L(d_1,\ldots,d_N)
    =
    \sum_{k=1}^N\|O_kd_k\|^2
    +\lambda\sum_{k=1}^{N-1}\|\rho_{k+1,k}d_{k+1}-d_k\|^2,
\end{equation}
under normalization $\Dist(d_k,\Image G_k)=1$.  A low-loss chain is not yet a Navier--Stokes obstruction, but it is a candidate for further PDE filtering.

Serious candidates must pass at least the following filters:
\begin{enumerate}[label=(\roman*),leftmargin=2em]
    \item divergence-free and Leray-projection compatibility;
    \item pressure Poisson compatibility, e.g. $-\Delta\pi^{\rm def}=\partial_i\partial_jR_{ij}$;
    \item compatibility with the local energy inequality;
    \item scale-critical normalization under Navier--Stokes scaling;
    \item closure into a coarse-grained residual equation.
\end{enumerate}
Only after passing these filters should a finite-window near-phantom be regarded as a serious NS-realizable obstruction candidate.

The dynamic formulation of Section~\ref{sec:dynamic-mechanism} upgrades this search into a validation problem.  Instead of minimizing only the static loss \eqref{eq:finite-chain-loss}, one should minimize a reproduction--observability--residual functional such as
\begin{equation}\label{eq:dynamic-search-loss}
    \mathfrak J_\Lambda(\mathfrak D_\Lambda)
    =
    \sum_{k\in\Lambda}\|O_k\mathfrak D_k\|^2
    +\alpha\sum_{k\in\Lambda}\|\calE_k(\mathfrak D_k)\|^2
    +\gamma\sum_{k,k+1\in\Lambda}\Rep_k(\mathfrak D_k,\mathfrak D_{k+1})^2
    -\delta\sum_{k\in\Lambda}\pos{\calL_k(\mathfrak D_k)},
\end{equation}
under the normalization \(\Dist(\mathfrak D_k,\Image G_k)\simeq 1\).  A low value of \(\mathfrak J_\Lambda\) is not a proof of an obstruction; it is a request for a theorem.  The next mathematical step is either to certify a true nearby solution of the residual equation by Newton--Kantorovich estimates, or to prove a quotient-gap lower bound showing that all such low-loss candidates are artifacts of the finite truncation, the pressure gauge, the cutoff, or the chosen coarse-graining.

This gives two opposite but equally rigorous uses of computation:
\[
\boxed{
\begin{array}{c}
\text{validated candidate:}\quad
\mathfrak D_\Lambda^0\ \Longrightarrow\ \exists\ \text{exact nearby PRV mechanism},\\[2mm]
\text{validated exclusion:}\quad
\mu_\Lambda\ge \mu_0>0\ \Longrightarrow\ \text{no finite-window phantom in the certified class}.
\end{array}}
\]
The second line is the computational analogue of the anti-phantom quotient theorems below; the first line is the defect-cascade analogue of the ASSF self-similar-profile validation strategy.

\section{Finite-window anti-phantom criteria}\label{sec:finite-window-antiphantom}
\noindent This section records the finite-dimensional quotient and
perturbative estimates used later by the anti-phantom framework.  These
results are rigorous finite-window tests for invisible non-gauge defects;
they are not a full localized Navier--Stokes anti-phantom theorem.  The
section proves the exact quotient principle, identifies the slow gluing
near-phantom in a pure scale-chain model, and gives a perturbative
finite-mode criterion once a linearized observability gap has been
certified.  The positive-cone clean theorem is recorded separately in
Section~\ref{sec:positive-cone-antiphantom}; the local ledger lift and
detailed error budget are placed in Appendix~\ref{app:local-ledger-lift}.

\subsection{Exact quotient criterion}

Let $\calC,\calD,\calY$ be finite-dimensional Hilbert spaces, let
\[
    G:\calC\to\calD,
    \qquad
    O:\calD\to\calY
\]
be linear maps, and assume
\[
    OG=0.
\]
The map $G$ represents cleanings or gauge directions, and $O$ represents the chosen observable tests.  Set
\[
    H:=(\Image G)^\perp\subset\calD.
\]
Define the finite-window anti-phantom constant
\begin{equation}\label{eq:abstract-mu}
    \mu
    :=
    \inf_{\Dist(d,\Image G)=1}\|Od\|_{\calY}.
\end{equation}

\begin{theorem}[Exact finite-window quotient criterion]\label{thm:exact-quotient-criterion}
Under the hypotheses above,
\begin{equation}\label{eq:mu-singular-value}
    \mu
    =
    \inf_{\substack{h\in H\\ \|h\|_{\calD}=1}}
    \|Oh\|_{\calY}
    =
    \sigma_{\min}(O|_H).
\end{equation}
Consequently,
\begin{equation}\label{eq:positive-mu-kernel}
    \mu>0
    \quad\Longleftrightarrow\quad
    \ker O\cap H=\{0\}
    \quad\Longleftrightarrow\quad
    \ker O=\Image G.
\end{equation}
Equivalently, the induced quotient map
\[
    \widetilde O:\calD/\Image G\to\calY,
    \qquad
    \widetilde O([d])=Od,
\]
is injective if and only if it is quantitatively coercive on the finite-dimensional quotient.
\end{theorem}

\begin{proof}
Every $d\in\calD$ has a unique orthogonal decomposition
\[
    d=g+h,
    \qquad
    g\in\Image G,
    \qquad
    h\in H=(\Image G)^\perp.
\]
Therefore
\[
    \Dist(d,\Image G)=\|h\|_{\calD}.
\]
Since $OG=0$, every vector in $\Image G$ is killed by $O$, and hence
\[
    Od=O(g+h)=Oh.
\]
Thus the normalization $\Dist(d,\Image G)=1$ is exactly the quotient normalization
$\|h\|_{\calD}=1$ on the orthogonal representative.  This proves
\[
    \mu
    =
    \inf_{\substack{h\in H\\ \|h\|_{\calD}=1}}
    \|Oh\|_{\calY}.
\]
Because $H$ is finite-dimensional, the infimum is attained on the unit sphere of $H$, and it is precisely the smallest singular value of the restricted map $O|_H$.

It remains to record the kernel criterion.  The equality above gives
\[
    \mu>0
    \quad\Longleftrightarrow\quad
    \ker(O|_H)=\{0\}
    \quad\Longleftrightarrow\quad
    \ker O\cap H=\{0\}.
\]
The inclusion $\Image G\subset\ker O$ follows directly from $OG=0$.  Conversely, assume
$\ker O\cap H=\{0\}$ and take $d\in\ker O$.  Write $d=g+h$ with $g\in\Image G$ and $h\in H$.
Since $g\in\Image G\subset\ker O$, we have
\[
    0=Od=Og+Oh=Oh.
\]
Thus $h\in\ker O\cap H$, so $h=0$.  Hence $d=g\in\Image G$, proving
$\ker O\subset\Image G$.  Therefore $\ker O=
\Image G$.  The reverse implication is immediate: if $\ker O=\Image G$, then
$\ker O\cap H=\Image G\cap(\Image G)^\perp=\{0\}$, and hence $\mu>0$.
\end{proof}

\begin{remark}[Meaning for the ledger complex]
The theorem says that, after the fake directions have been correctly placed in $\Image G$, finite-window anti-phantom is exactly a spectral-gap problem.  Thus a clean active window has no exact invisible non-gauge defect precisely when $O$ is coercive on the quotient space $\calD/\Image G$, represented by $H=(\Image G)^\perp$.
\end{remark}

\subsection{A sharp scale-chain model and the slow gluing near-phantom}

The previous theorem is purely algebraic.  The following model shows why gluing observables alone cannot yield a scale-uniform anti-phantom theorem.

\begin{theorem}[Discrete scale-chain anti-phantom constant]\label{thm:scale-chain-antiphantom}
Let
\[
    \calD=\R^L,
    \qquad
    \calC=\R,
    \qquad
    G(c)=c\mathbf 1,
\]
where $\mathbf 1=(1,\ldots,1)\in\R^L$.  Let
\[
    O:\R^L\to\R^{L-1},
    \qquad
    (Od)_j=d_{j+1}-d_j,
    \quad 1\le j\le L-1.
\]
Then
\begin{equation}\label{eq:scale-chain-mu}
    \mu_L
    :=
    \inf_{\Dist(d,\operatorname{span}\{\mathbf 1\})=1}
    \|Od\|_{\ell^2}
    =
    2\sin\left(\frac{\pi}{2L}\right).
\end{equation}
Moreover,
\begin{equation}\label{eq:scale-chain-asymptotic}
    \mu_L
    =
    \frac{\pi}{L}-\frac{\pi^3}{24L^3}+O(L^{-5}),
    \qquad L\to\infty.
\end{equation}
In particular, $\mu_L\sim\pi/L$.
\end{theorem}

\begin{proof}
By Theorem~\ref{thm:exact-quotient-criterion}, it suffices to restrict $O$ to
\[
    H=(\operatorname{span}\{\mathbf 1\})^\perp
    =
    \left\{d\in\R^L:\sum_{j=1}^L d_j=0\right\}.
\]
The matrix $O^*O$ is the discrete Neumann Laplacian on the path of length $L$:
\[
O^*O=
\begin{pmatrix}
1 & -1 & 0 & \cdots & 0\\
-1 & 2 & -1 & \cdots & 0\\
0 & -1 & 2 & \ddots & 0\\
\vdots & \vdots & \ddots & \ddots & -1\\
0 & 0 & 0 & -1 & 1
\end{pmatrix}.
\]
We compute its spectrum explicitly.  For $m=0,1,\ldots,L-1$ set
\[
    \theta_m:=\frac{m\pi}{L},
    \qquad
    v^{(m)}_n
    :=
    \cos\left(\left(n+\frac12\right)\theta_m\right),
    \qquad n=0,\ldots,L-1.
\]
For an interior index $1\le n\le L-2$, the identity
\[
    \cos(\alpha-\beta)+\cos(\alpha+\beta)
    =2\cos\alpha\cos\beta
\]
gives
\[
    v^{(m)}_{n-1}+v^{(m)}_{n+1}
    =2\cos\theta_m\,v^{(m)}_n.
\]
Hence
\[
    (O^*Ov^{(m)})_n
    =2v^{(m)}_n-v^{(m)}_{n-1}-v^{(m)}_{n+1}
    =2(1-\cos\theta_m)v^{(m)}_n
    =4\sin^2\left(\frac{\theta_m}{2}\right)v^{(m)}_n.
\]
At the left boundary,
\begin{align*}
    (O^*Ov^{(m)})_0
    &=v^{(m)}_0-v^{(m)}_1  \\
    &=\cos\frac{\theta_m}{2}-\cos\frac{3\theta_m}{2}\\
    &=2\sin\theta_m\sin\frac{\theta_m}{2}\\
    &=4\sin^2\left(\frac{\theta_m}{2}\right)\cos\frac{\theta_m}{2}
      =4\sin^2\left(\frac{\theta_m}{2}\right)v^{(m)}_0.
\end{align*}
The right boundary is identical.  Thus
\[
    O^*Ov^{(m)}
    =
    \lambda_m v^{(m)},
    \qquad
    \lambda_m=4\sin^2\left(\frac{m\pi}{2L}\right),
    \quad m=0,\ldots,L-1.
\]
The mode $m=0$ is the constant vector, hence it spans the cleaning direction.  Therefore the smallest positive eigenvalue on $H$ is
\[
    \lambda_1=4\sin^2\left(\frac{\pi}{2L}\right).
\]
Since the singular values of $O|_H$ are the square roots of the positive eigenvalues of $O^*O|_H$,
\[
    \mu_L=\sqrt{\lambda_1}=2\sin\left(\frac{\pi}{2L}\right).
\]
Finally, Taylor expansion gives
\[
    \sin x=x-\frac{x^3}{6}+O(x^5),
    \qquad x=\frac{\pi}{2L},
\]
which yields \eqref{eq:scale-chain-asymptotic}.
\end{proof}

\begin{remark}[Interpretation]
The first nonzero Neumann mode is a slowly varying scale mode.  It is almost constant across $k$ and therefore nearly invisible to pure edge-difference tests.  Thus a long scale chain naturally creates near-phantoms unless the observable map also contains genuinely local tests, such as reservoir, pressure, flux, dissipation, or adjoint trace observables.  This is why the finite-window constants $\mu_\Lambda$ must be studied together with the scale depth of $\Lambda$.
\end{remark}

\subsection{Finite-mode Galerkin observability criterion}

We next state a finite-mode criterion closer to the Navier--Stokes setting.  It is still a finite-dimensional theorem, but it has the correct structure for later computer-assisted or interval-certified verification.

Let $X_N$ be a finite-dimensional Hilbert space of divergence-free Fourier--Galerkin modes on $\mathbb T^3$, and let
\[
    u^\star\in C^1_tC^\infty_x([0,T]\times\mathbb T^3)
\]
be a smooth base trajectory.  Let
\[
    \mathcal L_N(t)v
    :=
    P_N\Delta v
    -P_N\bigl((u^\star\cdot\nabla)v+(v\cdot\nabla)u^\star\bigr)
\]
be the Oseen linearization on $X_N$, where $P_N$ is the finite-dimensional Leray--Galerkin projection.  Let $U_N(t,s)$ be the propagator for
\[
    \partial_t v=\mathcal L_N(t)v.
\]
Let $G_N:\calC_N\to X_N$ be a finite-dimensional cleaning map and set
\[
    H_N=(\Image G_N)^\perp\subset X_N.
\]
Choose observation operators $M_j:X_N\to\R^{m_j}$, time windows $I_j\subset[0,T]$, and nonnegative self-adjoint operators $Q_j:X_N\to X_N$.  Define
\[
    O^{\rm lin}_{\Lambda,N}h
    :=
    \left(
        M_jU_N(t_j,0)h,
        \ Q_j^{1/2}U_N(\cdot,0)h|_{I_j}
    \right)_{j=1}^J
\]
with values in
\[
    \calY_{\Lambda,N}
    =
    \prod_{j=1}^J
    \bigl(\R^{m_j}\times L^2(I_j;X_N)\bigr).
\]

\begin{theorem}[Finite-mode Galerkin anti-phantom criterion]\label{thm:finite-mode-antiphantom}
Define the observability Gramian
\begin{align}\label{eq:finite-gramian}
    W_{\Lambda,N}
    &:=
    \sum_{j=1}^J
    U_N(t_j,0)^*M_j^*M_jU_N(t_j,0)
    \\
    &\quad+
    \sum_{j=1}^J
    \int_{I_j}
    U_N(t,0)^*Q_jU_N(t,0)\dd t .
\end{align}
Then
\begin{equation}\label{eq:gramian-identity}
    \|O^{\rm lin}_{\Lambda,N}h\|_{\calY_{\Lambda,N}}^2
    =
    \langle W_{\Lambda,N}h,h\rangle_{X_N}.
\end{equation}
Consequently, if
\begin{equation}\label{eq:linear-gap-assumption}
    \lambda_{\min}\left(W_{\Lambda,N}|_{H_N}\right)
    \ge
    \gamma_{\Lambda,N}>0,
\end{equation}
then
\begin{equation}\label{eq:linear-mu-lower}
    \mu^{\rm lin}_{\Lambda,N}
    :=
    \inf_{\substack{h\in H_N\\ \|h\|_{X_N}=1}}
    \|O^{\rm lin}_{\Lambda,N}h\|_{\calY_{\Lambda,N}}
    \ge
    \sqrt{\gamma_{\Lambda,N}}.
\end{equation}
\end{theorem}

\begin{proof}
By definition of $O^{\rm lin}_{\Lambda,N}$,
\begin{align*}
    \|O^{\rm lin}_{\Lambda,N}h\|_{\calY_{\Lambda,N}}^2
    &=
    \sum_{j=1}^J
    \|M_jU_N(t_j,0)h\|_{\R^{m_j}}^2
    \\
    &\quad+
    \sum_{j=1}^J
    \int_{I_j}
    \|Q_j^{1/2}U_N(t,0)h\|_{X_N}^2\dd t \\
    &=
    \langle W_{\Lambda,N}h,h\rangle_{X_N}.
\end{align*}
This proves \eqref{eq:gramian-identity}.  Restricting to $H_N$ and taking the infimum over the unit sphere gives \eqref{eq:linear-mu-lower}.
\end{proof}

\begin{remark}[A degenerate terminal observation]
If $u^\star\equiv0$, $G_N=0$, and the only observable is $Oh=e^{T\Delta}h$, then the singular values are $e^{-T|m|^2}$ on Fourier modes.  Hence
\[
    \mu^{\rm Stokes}_{N}=e^{-T\kappa_{\max}^2},
    \qquad
    \kappa_{\max}:=\max\{|m|:m\in K_N\}.
\]
Thus a naive final-time observable becomes exponentially weak as the frequency window grows.  Robust anti-phantom estimates require time-distributed, dissipation-sensitive, or adjointly targeted observables.
\end{remark}

\subsection{Perturbative finite-window anti-phantom criterion}

The preceding finite-mode theorem becomes useful for the nonlinear ledger only after the difference between the true finite-window observable map and its linearized finite-mode model is controlled.

\begin{theorem}[Perturbative finite-window anti-phantom criterion]\label{thm:perturbative-antiphantom}
Let $H_N$ be a finite-dimensional quotient space and suppose
\[
    O_{\Lambda,N}=O^{\rm lin}_{\Lambda,N}+R_{\Lambda,N}
\]
on $H_N$.  Assume the linearized gap
\[
    \inf_{\substack{h\in H_N\\ \|h\|=1}}
    \|O^{\rm lin}_{\Lambda,N}h\|
    \ge
    \gamma_0>0
\]
and the perturbative error bound
\begin{equation}\label{eq:perturbative-error-bound}
    \|R_{\Lambda,N}h\|
    \le
    \eps_*
    \|h\|
    \qquad
    \text{for all }h\in H_N.
\end{equation}
Then
\begin{equation}\label{eq:perturbative-mu-bound}
    \mu_{\Lambda,N}
    :=
    \inf_{\substack{h\in H_N\\ \|h\|=1}}
    \|O_{\Lambda,N}h\|
    \ge
    \gamma_0-\eps_*.
\end{equation}
In particular, if $\eps_*<\gamma_0$, then $\mu_{\Lambda,N}>0$.
\end{theorem}

\begin{proof}
For every $h\in H_N$ with $\|h\|=1$, the triangle inequality gives
\[
    \|O_{\Lambda,N}h\|
    \ge
    \|O^{\rm lin}_{\Lambda,N}h\|-
    \|R_{\Lambda,N}h\|
    \ge
    \gamma_0-\eps_*.
\]
Taking the infimum over the quotient unit sphere proves \eqref{eq:perturbative-mu-bound}.
\end{proof}

In a local Navier--Stokes ledger application, one should decompose
\begin{equation}\label{eq:error-budget}
    \eps_*
    =
    \eps_{\rm tr}
    +
    \eps_{\rm prs}
    +
    \eps_{\rm loc}
    +
    \eps_{\rm nl},
\end{equation}
where $\eps_{\rm tr}$ is the Galerkin or basis truncation error, $\eps_{\rm prs}$ is the harmonic-pressure truncation error, $\eps_{\rm loc}$ is cutoff and localization leakage, and $\eps_{\rm nl}$ is the nonlinear remainder of the defect map.  For a global periodic finite-mode model one often has no harmonic-pressure gauge term.  For local parabolic windows, however, one must include harmonic-pressure directions in the cleaning space or else the anti-phantom constant may measure a pressure gauge artifact rather than a genuine invisible defect.

\begin{remark}[What has and has not been proved]
Theorems~\ref{thm:exact-quotient-criterion}--\ref{thm:perturbative-antiphantom} prove a genuine finite-dimensional anti-phantom skeleton, in the same broad trace-visibility spirit as \cite{Yu2026SchurVisibility}.  They do not prove that a full local Navier--Stokes window has $\mu_\Lambda>0$, nor that $\inf_\Lambda\mu_\Lambda>0$ as the window depth grows.  The next analytic task is to build a concrete $O_\Lambda$ and $G_\Lambda$ from the ledger residuals, certify a linearized quotient gap, and prove that the pressure, localization, truncation, and nonlinear errors in \eqref{eq:error-budget} are smaller than that gap.
\end{remark}

\subsection{Structural sources of gap failure}\label{subsec:structural-gap-failure}

The perturbative and local-ledger error budgets clarify which obstructions are structural and which are mainly technical.  The most dangerous terms are not evenly distributed.

\begin{enumerate}[label=(\roman*),leftmargin=2em]
    \item \textbf{Harmonic-pressure leakage.}  On local windows, pressure is only partly determined by the local Calderon--Zygmund solve.  The remaining spatially harmonic component may look mild in a scale-invariant $L^{3/2}$ oscillation norm while still influencing pressure gradients and transport.  If harmonic-pressure directions are not included in the cleaning space, $\mu_\Lambda$ may measure a pressure gauge artifact rather than a physical invisible defect.

    \item \textbf{Scale-uniform degeneration.}  The explicit chain formula
    \[
        \mu_L=2\sin\frac{\pi}{2L}\sim\frac{\pi}{L}
    \]
    shows that pure gluing observables cannot remain uniformly coercive as the scale depth grows.  Likewise, purely terminal Stokes observations degenerate exponentially in the Fourier cutoff.  Robust observability therefore requires local reservoir, dissipation, pressure, flux, or adjoint trace channels.

    \item \textbf{Genuine unstable eigenmodes.}  A small singular value is not automatically a fake phantom.  Around a nontrivial Oseen or self-similar profile, a near-minimizing direction may represent a genuine unstable mode of the linearized dynamics.  Such a direction should be classified as a candidate obstruction, not discarded as a gauge artifact.
\end{enumerate}

By contrast, $\eps_{\rm nl}$ and $\eps_{\rm tr}$ are often budgetable in a fixed smooth active window: the former is a higher-order Taylor/Duhamel remainder and the latter is a Sobolev or harmonic tail.  The hard case is when these budgetable errors are small but the certified gap has already collapsed because of harmonic pressure, long-scale gluing, or a true unstable mode.

A minimal verification order is therefore:
\[
\begin{gathered}
\text{exact quotient theorem}
\Longrightarrow
\text{explicit scale-chain degeneration}\\
\Longrightarrow
\text{periodic finite-mode gap certification}
\Longrightarrow
\text{local harmonic-pressure ledger window}.
\end{gathered}
\]
If the gap collapses in the third or fourth step, the correct response is not to declare failure, but to extract and classify the minimizing singular vector: slow scale mode, harmonic-pressure mode, localization-forcing mode, or genuine Oseen/self-similar instability.

\section{A conditional finite-window dynamic anti-phantom theorem}\label{sec:conditional-dynamic-antiphantom}

We now isolate the finite-dimensional mechanism behind the dynamic
anti-phantom principle.  The purpose of this section is not yet to prove
a scale-uniform Navier--Stokes regularity criterion.  Rather, we prove
the basic finite-window coercivity statement which turns the absence of
exact non-gauge invisible mechanisms into a quantitative lower bound.

Throughout this section we use the homogeneous finite-window model: the
residuals and ledger channels are either linearized residuals or homogeneous
leading finite-window components.  Equivalently, the combined defect-size
functional defined below descends to the quotient and is positively
one-homogeneous there.  This convention is essential for obtaining a global
quotient coercivity estimate from a compactness argument.  The genuinely
nonlinear finite-window statement is treated locally in
Section~\ref{sec:linearized-to-nonlinear-gap}.

\subsection{Finite-window defect space}

Let \(\Lambda=\{k_0,\ldots,k_0+L\}\) be a finite dyadic window.  We denote
by
\[
        \mathcal D_\Lambda
\]
the finite-dimensional space of ledger defects on this window.  A typical
element will be denoted by
\[
        \mathfrak D_\Lambda
        =
        (U_k,P_k,R_k,\Phi_k,\Pi_k,\Lambda_k,s_k)_{k\in\Lambda},
\]
where \(U_k\) is the resolved velocity defect, \(P_k\) is the pressure
defect, \(R_k\) is the unresolved Reynolds-covariance defect, the terms
\(\Phi_k,\Pi_k,\Lambda_k\) record flux, pressure-transport, and
localization components, and \(s_k\) denotes the slack variables used to
write ledger inequalities as residual equations.

Let
\[
        G_\Lambda:\mathcal C_\Lambda\longrightarrow \mathcal D_\Lambda
\]
be the finite-window cleaning map.  Its image represents defects generated
only by gauge, cutoff, harmonic-pressure, or coarse-graining artifacts.
Thus the physically relevant defect space is the quotient
\[
        \mathcal Q_\Lambda
        :=
        \mathcal D_\Lambda/\operatorname{Im}G_\Lambda .
\]
We write
\[
        \operatorname{dist}_\Lambda
        (\mathfrak D_\Lambda,\operatorname{Im}G_\Lambda)
\]
for the quotient distance.

\subsection{Four channels of defect detection}

We consider four finite-window functionals.

First, the combined observability map
\[
        O_\Lambda:\mathcal D_\Lambda\longrightarrow Y_\Lambda
\]
collects the active pressure, flux, positive energy, and selected-time
trace observations:
\[
        O_\Lambda
        =
        (O_\Lambda^{\rm prs},
         O_\Lambda^{\rm flux},
         O_\Lambda^{\rm en},
         O_\Lambda^{\rm tr}).
\]

Second, the Navier--Stokes realizability residual
\[
        \mathcal E_\Lambda:\mathcal D_\Lambda\longrightarrow Z_\Lambda
\]
measures the failure of the defect package to satisfy the coarse
Navier--Stokes constraints:
\[
        \mathcal E_\Lambda
        =
        (\mathcal E_\Lambda^{\rm div},
         \mathcal E_\Lambda^{\rm mom},
         \mathcal E_\Lambda^{\rm prs},
         \mathcal E_\Lambda^{\rm lei},
         \mathcal E_\Lambda^{\rm glue}).
\]
Here the terms correspond respectively to incompressibility, coarse
momentum balance, pressure compatibility, the local energy inequality, and
inter-scale gluing.

Third, the reproduction residual
\[
        \operatorname{Rep}_\Lambda(\mathfrak D_\Lambda)
\]
measures the failure of the defect to reproduce from scale to scale.  In
one convenient normalization one may take
\[
        \operatorname{Rep}_\Lambda(\mathfrak D_\Lambda)
        :=
        \left(
        \sum_{k=k_0}^{k_0+L-1}
        \operatorname{dist}
        \bigl(
        \mathfrak D_{k+1},\mathcal R_k(\mathfrak D_k)
        \bigr)^2
        \right)^{1/2},
\]
where \(\mathcal R_k\) is the finite-window reproduction relation.

Fourth, the net ledger profit is denoted by
\[
        \mathsf P_\Lambda(\mathfrak D_\Lambda).
\]
It represents the cumulative untaxed supply after subtracting the
available tax and controlled leakage:
\[
        \mathsf P_\Lambda
        =
        \sum_{k\in\Lambda}
        \bigl(
        \mathsf{Supp}^{\rm full}_k
        -
        \mathsf{Tax}^{\rm full}_k
        -
        \mathsf{Leak}^{\rm full}_k
        \bigr).
\]
The positive part
\[
        [\mathsf P_\Lambda(\mathfrak D_\Lambda)]_+
\]
records the amount of genuinely profitable critical supply.

For later reference set
\[
        F_\Lambda(\mathfrak D_\Lambda)
        :=
        \|O_\Lambda\mathfrak D_\Lambda\|
        +
        \|\mathcal E_\Lambda(\mathfrak D_\Lambda)\|
        +
        \operatorname{Rep}_\Lambda(\mathfrak D_\Lambda)
        +
        [\mathsf P_\Lambda(\mathfrak D_\Lambda)]_+ .
\]
In this homogeneous finite-window section we assume that \(F_\Lambda\)
descends to a continuous positively one-homogeneous functional on
\(\mathcal Q_\Lambda\).

\subsection{Exact dynamic phantom exclusion}

The key structural hypothesis is that there is no exact defect which is
simultaneously non-gauge, invisible, Navier--Stokes realizable,
scale-reproducing, and non-profitable.

\begin{hypothesis}[No exact dynamic phantom]
\label{hyp:no-exact-dynamic-phantom}
If \(\mathfrak D_\Lambda\in\mathcal D_\Lambda\) satisfies
\[
        O_\Lambda\mathfrak D_\Lambda=0,
        \qquad
        \mathcal E_\Lambda(\mathfrak D_\Lambda)=0,
        \qquad
        \operatorname{Rep}_\Lambda(\mathfrak D_\Lambda)=0,
        \qquad
        [\mathsf P_\Lambda(\mathfrak D_\Lambda)]_+=0,
\]
then
\[
        \mathfrak D_\Lambda\in\operatorname{Im}G_\Lambda .
\]
\end{hypothesis}

This hypothesis says that the only exact defect which escapes all
observability channels, satisfies the Navier--Stokes residual equations,
reproduces across the window, and produces no positive ledger profit is a
cleaning artifact.

\begin{theorem}[Conditional finite-window dynamic anti-phantom theorem]
\label{thm:conditional-dynamic-anti-phantom}
Assume that \(\mathcal D_\Lambda\) is finite-dimensional, that
\(F_\Lambda\) descends to a continuous positively one-homogeneous
functional on the quotient
\(\mathcal Q_\Lambda=\mathcal D_\Lambda/\operatorname{Im}G_\Lambda\), and
that Hypothesis~\ref{hyp:no-exact-dynamic-phantom} holds.  Then there
exists a constant \(c_\Lambda>0\) such that for every
\(\mathfrak D_\Lambda\in\mathcal D_\Lambda\),
\[
\boxed{
        \|O_\Lambda\mathfrak D_\Lambda\|
        +
        \|\mathcal E_\Lambda(\mathfrak D_\Lambda)\|
        +
        \operatorname{Rep}_\Lambda(\mathfrak D_\Lambda)
        +
        [\mathsf P_\Lambda(\mathfrak D_\Lambda)]_+
        \ge
        c_\Lambda\,
        \operatorname{dist}_\Lambda
        (\mathfrak D_\Lambda,\operatorname{Im}G_\Lambda).
}
\]
Equivalently, a normalized non-gauge defect cannot be simultaneously
low-observable, nearly Navier--Stokes-realizable, nearly reproducible, and
non-profitable.
\end{theorem}

\begin{proof}
It suffices to prove the estimate on the quotient unit sphere
\[
        S_\Lambda
        :=
        \left\{
        [\mathfrak D_\Lambda]\in\mathcal Q_\Lambda:
        \|[\mathfrak D_\Lambda]\|_{\mathcal Q_\Lambda}=1
        \right\}.
\]
Since \(\mathcal Q_\Lambda\) is finite-dimensional, \(S_\Lambda\) is
compact.  By the assumed quotient continuity, \(F_\Lambda\) is continuous
on \(S_\Lambda\).  We claim that \(F_\Lambda\) is strictly positive on
\(S_\Lambda\).

Indeed, if \(F_\Lambda([\mathfrak D_\Lambda])=0\), then each nonnegative
term in the definition of \(F_\Lambda\) vanishes:
\[
        O_\Lambda\mathfrak D_\Lambda=0,
        \qquad
        \mathcal E_\Lambda(\mathfrak D_\Lambda)=0,
        \qquad
        \operatorname{Rep}_\Lambda(\mathfrak D_\Lambda)=0,
        \qquad
        [\mathsf P_\Lambda(\mathfrak D_\Lambda)]_+=0.
\]
By Hypothesis~\ref{hyp:no-exact-dynamic-phantom},
\[
        \mathfrak D_\Lambda\in\operatorname{Im}G_\Lambda.
\]
Hence
\[
        [\mathfrak D_\Lambda]=0
        \quad\text{in }\mathcal Q_\Lambda,
\]
which contradicts
\[
        \|[\mathfrak D_\Lambda]\|_{\mathcal Q_\Lambda}=1.
\]
Therefore \(F_\Lambda>0\) on \(S_\Lambda\).  Compactness gives
\[
        c_\Lambda
        :=
        \inf_{S_\Lambda}F_\Lambda
        >0.
\]
For an arbitrary class \([\mathfrak D_\Lambda]\neq0\), apply this lower
bound to
\[
        [\mathfrak D_\Lambda]/
        \|[\mathfrak D_\Lambda]\|_{\mathcal Q_\Lambda}
\]
and use the positive one-homogeneity of \(F_\Lambda\).  This yields
\[
        F_\Lambda([\mathfrak D_\Lambda])
        \ge
        c_\Lambda
        \|[\mathfrak D_\Lambda]\|_{\mathcal Q_\Lambda}
        =
        c_\Lambda
        \operatorname{dist}_\Lambda
        (\mathfrak D_\Lambda,\operatorname{Im}G_\Lambda),
\]
which proves the theorem.  The zero class is immediate.
\end{proof}

\begin{remark}[Meaning of the theorem]
The theorem does not assert that Navier--Stokes singularities do not
exist.  It gives a finite-window reduction: if a normalized defect is not
a gauge, cutoff, harmonic-pressure, or coarse-graining artifact, then at
least one of the following four alternatives must occur:
\[
\begin{gathered}
        \text{it is observed,}\qquad
        \text{it fails Navier--Stokes realizability,}\\
        \text{it fails scale reproduction,}\qquad
        \text{or it produces positive ledger profit.}
\end{gathered}
\]
Thus a potential singular mechanism cannot be merely an invisible static
defect.  It must be a dynamically sustained, Navier--Stokes-realizable,
reproducible, and ledger-profitable mechanism.
\end{remark}

\section{Quotient-gap formulation of the dynamic anti-phantom condition}\label{sec:dynamic-gap-formulation}

The preceding theorem reduced the dynamic anti-phantom principle to the
absence of exact non-gauge mechanisms which are simultaneously invisible,
Navier--Stokes-realizable, scale-reproducing, and non-profitable.  We now
rewrite this condition as a finite-dimensional quotient-gap problem.

\subsection{The exact dynamic phantom set}

Let \(\mathcal D_\Lambda\) be the finite-window defect space and let
\[
        G_\Lambda:\mathcal C_\Lambda\to\mathcal D_\Lambda
\]
be the cleaning map.  Define the exact dynamic constraint set
\[
        \mathcal K_\Lambda
        :=
        \left\{
        \mathfrak D_\Lambda\in\mathcal D_\Lambda:
        \mathcal E_\Lambda(\mathfrak D_\Lambda)=0,\quad
        \operatorname{Rep}_\Lambda(\mathfrak D_\Lambda)=0,\quad
        [\mathsf P_\Lambda(\mathfrak D_\Lambda)]_+=0
        \right\}.
\]
Thus \(\mathcal K_\Lambda\) consists of defect packages which satisfy the
finite-window Navier--Stokes residual equations, reproduce across the
window, and do not produce positive net ledger profit.  In the homogeneous
finite-window setting, \(\mathcal K_\Lambda\) is a closed cone.

The exact dynamic phantom space is then
\[
        \mathcal Z_\Lambda
        :=
        \mathcal K_\Lambda\cap\ker O_\Lambda.
\]
Elements of \(\mathcal Z_\Lambda\) are invisible,
Navier--Stokes-realizable, reproducible, and non-profitable.  The
no-exact-phantom condition is precisely
\[
        \mathcal Z_\Lambda
        \subset
        \operatorname{Im}G_\Lambda .
\]
Equivalently,
\[
        \mathcal Z_\Lambda/
        \bigl(\mathcal Z_\Lambda\cap\operatorname{Im}G_\Lambda\bigr)
        =
        \{0\}.
\]

\subsection{Dynamic quotient gap}

We now define the dynamic quotient gap
\[
\boxed{
        \mu_\Lambda^{\rm dyn}
        :=
        \inf_{\operatorname{dist}_\Lambda
        (\mathfrak D_\Lambda,\operatorname{Im}G_\Lambda)=1}
        \left[
        \|O_\Lambda\mathfrak D_\Lambda\|
        +
        \|\mathcal E_\Lambda(\mathfrak D_\Lambda)\|
        +
        \operatorname{Rep}_\Lambda(\mathfrak D_\Lambda)
        +
        [\mathsf P_\Lambda(\mathfrak D_\Lambda)]_+
        \right].
}
\]
This number measures how far a normalized non-gauge defect must be from
being simultaneously invisible, Navier--Stokes-realizable, reproducible,
and non-profitable.

\begin{proposition}[No exact dynamic phantom is equivalent to positive dynamic gap]
\label{prop:no-phantom-gap-equivalence}
Assume that \(\mathcal D_\Lambda/\operatorname{Im}G_\Lambda\) is
finite-dimensional, that the combined defect-size functional descends
continuously and positively one-homogeneously to the quotient, and that the
zero set defining \(\mathcal Z_\Lambda\) is conic.  Then
\[
        \mathcal Z_\Lambda
        \subset
        \operatorname{Im}G_\Lambda
\]
if and only if
\[
        \mu_\Lambda^{\rm dyn}>0.
\]
\end{proposition}

\begin{proof}
Suppose first that
\[
        \mu_\Lambda^{\rm dyn}=0.
\]
Then there exists a sequence
\(\mathfrak D_\Lambda^{(n)}\in\mathcal D_\Lambda\) such that
\[
        \operatorname{dist}_\Lambda
        (\mathfrak D_\Lambda^{(n)},\operatorname{Im}G_\Lambda)=1
\]
and
\[
        \|O_\Lambda\mathfrak D_\Lambda^{(n)}\|
        +
        \|\mathcal E_\Lambda(\mathfrak D_\Lambda^{(n)})\|
        +
        \operatorname{Rep}_\Lambda(\mathfrak D_\Lambda^{(n)})
        +
        [\mathsf P_\Lambda(\mathfrak D_\Lambda^{(n)})]_+
        \longrightarrow 0 .
\]
Passing to the quotient unit sphere and using finite-dimensional
compactness, we may extract a subsequence converging to a class
\[
        [\mathfrak D_\Lambda^\ast]
        \in
        \mathcal D_\Lambda/\operatorname{Im}G_\Lambda
\]
with
\[
        \|[\mathfrak D_\Lambda^\ast]\|=1.
\]
By continuity,
\[
        O_\Lambda\mathfrak D_\Lambda^\ast=0,
        \qquad
        \mathcal E_\Lambda(\mathfrak D_\Lambda^\ast)=0,
        \qquad
        \operatorname{Rep}_\Lambda(\mathfrak D_\Lambda^\ast)=0,
        \qquad
        [\mathsf P_\Lambda(\mathfrak D_\Lambda^\ast)]_+=0.
\]
Thus
\[
        \mathfrak D_\Lambda^\ast\in\mathcal Z_\Lambda.
\]
If
\[
        \mathcal Z_\Lambda\subset\operatorname{Im}G_\Lambda,
\]
then
\[
        [\mathfrak D_\Lambda^\ast]=0
\]
in the quotient, contradicting
\[
        \|[\mathfrak D_\Lambda^\ast]\|=1.
\]
Hence \(\mu_\Lambda^{\rm dyn}>0\).

Conversely, suppose that
\[
        \mathcal Z_\Lambda
        \not\subset
        \operatorname{Im}G_\Lambda.
\]
Then there exists
\[
        \mathfrak D_\Lambda^\ast\in\mathcal Z_\Lambda
\]
with
\[
        \operatorname{dist}_\Lambda
        (\mathfrak D_\Lambda^\ast,\operatorname{Im}G_\Lambda)>0.
\]
Normalize it by setting
\[
        \widetilde{\mathfrak D}_\Lambda
        :=
        \frac{\mathfrak D_\Lambda^\ast}
        {\operatorname{dist}_\Lambda
        (\mathfrak D_\Lambda^\ast,\operatorname{Im}G_\Lambda)} .
\]
Since \(\mathcal Z_\Lambda\) is conic,
\[
        \widetilde{\mathfrak D}_\Lambda\in\mathcal Z_\Lambda,
\]
and therefore
\[
        \operatorname{dist}_\Lambda
        (\widetilde{\mathfrak D}_\Lambda,\operatorname{Im}G_\Lambda)=1,
\]
while
\[
        O_\Lambda\widetilde{\mathfrak D}_\Lambda=0,
        \qquad
        \mathcal E_\Lambda(\widetilde{\mathfrak D}_\Lambda)=0,
        \qquad
        \operatorname{Rep}_\Lambda(\widetilde{\mathfrak D}_\Lambda)=0,
        \qquad
        [\mathsf P_\Lambda(\widetilde{\mathfrak D}_\Lambda)]_+=0.
\]
Therefore
\[
        \mu_\Lambda^{\rm dyn}=0.
\]
This proves the equivalence.
\end{proof}

\subsection{Reduced observability gap on the exact constraint set}

The previous proposition is useful but still packages all four channels
together.  A sharper formulation separates the geometric part from the
Navier--Stokes compatibility part.

Define the reduced quotient gap
\[
\boxed{
        \gamma_\Lambda
        :=
        \inf_{\substack{
        \mathfrak D_\Lambda\in\mathcal K_\Lambda\\
        \operatorname{dist}_\Lambda
        (\mathfrak D_\Lambda,\operatorname{Im}G_\Lambda)=1
        }}
        \|O_\Lambda\mathfrak D_\Lambda\| .
}
\]
Thus \(\gamma_\Lambda\) tests whether the observability map separates
non-gauge defects after the defect has already been restricted to the
Navier--Stokes-realizable, reproducible, non-profitable constraint set.

\begin{theorem}[Reduced quotient gap implies dynamic anti-phantom]
\label{thm:reduced-gap-implies-dynamic}
Assume that \(\mathcal D_\Lambda/\operatorname{Im}G_\Lambda\) is
finite-dimensional and that \(\mathcal K_\Lambda\) is a closed cone.  If
\[
        \gamma_\Lambda>0,
\]
then the exact dynamic phantom condition holds:
\[
        \mathcal Z_\Lambda\subset\operatorname{Im}G_\Lambda.
\]
Consequently,
\[
        \mu_\Lambda^{\rm dyn}>0,
\]
and there exists \(c_\Lambda>0\) such that
\[
\boxed{
        \|O_\Lambda\mathfrak D_\Lambda\|
        +
        \|\mathcal E_\Lambda(\mathfrak D_\Lambda)\|
        +
        \operatorname{Rep}_\Lambda(\mathfrak D_\Lambda)
        +
        [\mathsf P_\Lambda(\mathfrak D_\Lambda)]_+
        \ge
        c_\Lambda\,
        \operatorname{dist}_\Lambda
        (\mathfrak D_\Lambda,\operatorname{Im}G_\Lambda)
}
\]
for every \(\mathfrak D_\Lambda\in\mathcal D_\Lambda\), under the
homogeneous quotient convention above.
\end{theorem}

\begin{proof}
Let
\[
        \mathfrak D_\Lambda\in\mathcal Z_\Lambda.
\]
Then, by definition,
\[
        \mathfrak D_\Lambda\in\mathcal K_\Lambda
        \qquad\text{and}\qquad
        O_\Lambda\mathfrak D_\Lambda=0.
\]
Suppose for contradiction that
\[
        \mathfrak D_\Lambda\notin\operatorname{Im}G_\Lambda.
\]
Then
\[
        \operatorname{dist}_\Lambda
        (\mathfrak D_\Lambda,\operatorname{Im}G_\Lambda)>0.
\]
Set
\[
        \widehat{\mathfrak D}_\Lambda
        :=
        \frac{\mathfrak D_\Lambda}
        {\operatorname{dist}_\Lambda
        (\mathfrak D_\Lambda,\operatorname{Im}G_\Lambda)}.
\]
Since \(\mathcal K_\Lambda\) is a cone,
\[
        \widehat{\mathfrak D}_\Lambda\in\mathcal K_\Lambda
\]
and
\[
        \operatorname{dist}_\Lambda
        (\widehat{\mathfrak D}_\Lambda,\operatorname{Im}G_\Lambda)=1.
\]
Moreover,
\[
        O_\Lambda\widehat{\mathfrak D}_\Lambda=0.
\]
This contradicts the assumption
\[
        \gamma_\Lambda>0.
\]
Therefore every element of \(\mathcal Z_\Lambda\) lies in
\(\operatorname{Im}G_\Lambda\).  Proposition
\ref{prop:no-phantom-gap-equivalence} then gives
\[
        \mu_\Lambda^{\rm dyn}>0.
\]
The stated coercive estimate follows directly from the definition of
\(\mu_\Lambda^{\rm dyn}\), with \(c_\Lambda=\mu_\Lambda^{\rm dyn}\).
\end{proof}

\begin{remark}[What remains to be proved]
The dynamic anti-phantom problem has now been reduced to proving the
finite-window gap
\[
        \gamma_\Lambda>0.
\]
Equivalently, one must show that
\[
        \ker O_\Lambda
        \cap
        \mathcal K_\Lambda
        \subset
        \operatorname{Im}G_\Lambda.
\]
This is the precise finite-window form of the next hard mathematical
question: after imposing Navier--Stokes realizability, scale reproduction,
and non-positive ledger profit, do the pressure--flux--energy--trace
observables separate every remaining non-gauge defect?
\end{remark}

\section{Clean periodic finite-window reduction}\label{sec:clean-periodic-reduction}

The previous section reduced the dynamic anti-phantom principle to the
positivity of a finite-window quotient gap.  We now formulate a clean
periodic model in which this gap becomes an explicit finite-dimensional
linear-algebra problem.

This model deliberately removes the perturbative errors produced by
localization, harmonic-pressure tails, moving cutoffs, and truncation.
Those errors will be reintroduced later through an error-budget argument.
The purpose of the present section is to identify the non-perturbative
core of the dynamic anti-phantom mechanism.

\subsection{Clean periodic window}

Let \(\mathbb T^3\) be the periodic box, and let
\[
        \Lambda\subset \mathbb Z^3\setminus\{0\}
\]
be a finite active Fourier window.  We assume that \(\Lambda\) is
symmetric under \(q\mapsto -q\), contains no zero horizontal or vertical
active modes when those modes are tested, and is closed under the finite
triadic interactions retained in the window.

The clean finite-window defect space is denoted by
\[
        \mathcal D_\Lambda^{\rm cl}.
\]
A defect
\[
        \mathfrak D_\Lambda\in\mathcal D_\Lambda^{\rm cl}
\]
consists of the Fourier coefficients of the resolved velocity defect,
pressure defect, Reynolds-stress defect, ledger variables, and slack
variables restricted to \(\Lambda\).  The incompressibility constraint is
included in the definition:
\[
        q\cdot \widehat U(q)=0,
        \qquad q\in\Lambda.
\]

The clean gauge map is
\[
        G_\Lambda^{\rm cl}:
        \mathcal C_\Lambda^{\rm cl}
        \longrightarrow
        \mathcal D_\Lambda^{\rm cl}.
\]
Its image consists of the finite-window artifacts generated by pressure
normalization, harmless coarse-graining choices, and algebraic cleaning
directions.  In the periodic clean model there is no cutoff leakage and no
harmonic-pressure tail.

We write
\[
        \mathcal Q_\Lambda^{\rm cl}
        :=
        \mathcal D_\Lambda^{\rm cl}/\operatorname{Im}G_\Lambda^{\rm cl}.
\]

\subsection{Clean constraint matrix}

In the clean model, the Navier--Stokes realizability residual, the
reproduction residual, and the non-profit ledger constraint are represented
by finite-dimensional homogeneous maps:
\[
        \mathcal E_\Lambda^{\rm cl}:
        \mathcal D_\Lambda^{\rm cl}\to Z_\Lambda^{\rm cl},
\]
\[
        \operatorname{Rep}_\Lambda^{\rm cl}:
        \mathcal D_\Lambda^{\rm cl}\to R_\Lambda^{\rm cl},
\]
and
\[
        \mathsf P_\Lambda^{0}:
        \mathcal D_\Lambda^{\rm cl}\to \mathbb R .
\]
Here \(\mathsf P_\Lambda^{0}\) denotes the homogeneous zero-profit part of
the ledger functional.  It records the condition that the defect does not
produce positive untaxed critical supply at leading order.

Define the clean dynamic constraint operator
\[
        B_\Lambda^{\rm cl}
        :=
        \bigl(
        \mathcal E_\Lambda^{\rm cl},
        \operatorname{Rep}_\Lambda^{\rm cl},
        \mathsf P_\Lambda^{0}
        \bigr).
\]
The exact clean dynamic constraint space is
\[
        \mathcal K_\Lambda^{\rm cl}
        :=
        \ker B_\Lambda^{\rm cl}.
\]
Thus
\[
        \mathfrak D_\Lambda\in\mathcal K_\Lambda^{\rm cl}
\]
means that the defect is exactly Navier--Stokes-compatible, exactly
scale-reproducing, and non-profitable at leading order in the clean
periodic model.

The clean observability map is
\[
        O_\Lambda^{\rm cl}
        =
        \bigl(
        O_\Lambda^{\rm prs},
        O_\Lambda^{\rm flux},
        O_\Lambda^{\rm en},
        O_\Lambda^{\rm tr}
        \bigr).
\]
It records the active pressure channel, the flux channel, the positive
energy channel, and the selected-time or adjoint-trace channel.

\subsection{Clean periodic quotient gap}

We define the clean dynamic quotient gap by
\[
\boxed{
        \gamma_\Lambda^{\rm cl}
        :=
        \inf_{\substack{
        \mathfrak D_\Lambda\in\mathcal K_\Lambda^{\rm cl}\\
        \operatorname{dist}
        (\mathfrak D_\Lambda,\operatorname{Im}G_\Lambda^{\rm cl})=1
        }}
        \|O_\Lambda^{\rm cl}\mathfrak D_\Lambda\| .
}
\]
Equivalently, \(\gamma_\Lambda^{\rm cl}\) is the least amount of
pressure--flux--energy--trace visibility carried by a normalized clean
dynamic defect which is not a gauge artifact.

\begin{theorem}[Clean periodic finite-window anti-phantom criterion]
\label{thm:clean-periodic-gap}
Assume that
\[
        \ker O_\Lambda^{\rm cl}
        \cap
        \ker B_\Lambda^{\rm cl}
        \subset
        \operatorname{Im}G_\Lambda^{\rm cl}.
\]
Then
\[
        \gamma_\Lambda^{\rm cl}>0.
\]
Consequently, there exists \(c_\Lambda^{\rm cl}>0\) such that for every
\(\mathfrak D_\Lambda\in\mathcal D_\Lambda^{\rm cl}\),
\[
\boxed{
        \|O_\Lambda^{\rm cl}\mathfrak D_\Lambda\|
        +
        \|B_\Lambda^{\rm cl}\mathfrak D_\Lambda\|
        \ge
        c_\Lambda^{\rm cl}\,
        \operatorname{dist}
        (\mathfrak D_\Lambda,\operatorname{Im}G_\Lambda^{\rm cl}) .
}
\]
\end{theorem}

\begin{proof}
The proof is finite-dimensional.  Let
\[
        \mathcal Q_\Lambda^{\rm cl}
        =
        \mathcal D_\Lambda^{\rm cl}/\operatorname{Im}G_\Lambda^{\rm cl}.
\]
Since the window is finite, this quotient is finite-dimensional.

Define
\[
        F_\Lambda^{\rm cl}([\mathfrak D_\Lambda])
        :=
        \|O_\Lambda^{\rm cl}\mathfrak D_\Lambda\|
        +
        \|B_\Lambda^{\rm cl}\mathfrak D_\Lambda\|.
\]
The map is continuous and positively one-homogeneous on the quotient.
Suppose, for contradiction, that the desired coercive estimate fails.
Then there exists a sequence \(\mathfrak D_\Lambda^{(n)}\) satisfying
\[
        \operatorname{dist}
        (\mathfrak D_\Lambda^{(n)},\operatorname{Im}G_\Lambda^{\rm cl})
        =1
\]
and
\[
        \|O_\Lambda^{\rm cl}\mathfrak D_\Lambda^{(n)}\|
        +
        \|B_\Lambda^{\rm cl}\mathfrak D_\Lambda^{(n)}\|
        \longrightarrow 0 .
\]
Passing to the quotient unit sphere and using compactness, we may assume
that
\[
        [\mathfrak D_\Lambda^{(n)}]
        \longrightarrow
        [\mathfrak D_\Lambda^\ast]
        \quad\text{in }\mathcal Q_\Lambda^{\rm cl},
\]
with
\[
        \|[\mathfrak D_\Lambda^\ast]\|_{\mathcal Q_\Lambda^{\rm cl}}=1.
\]
By continuity,
\[
        O_\Lambda^{\rm cl}\mathfrak D_\Lambda^\ast=0,
        \qquad
        B_\Lambda^{\rm cl}\mathfrak D_\Lambda^\ast=0.
\]
Hence
\[
        \mathfrak D_\Lambda^\ast
        \in
        \ker O_\Lambda^{\rm cl}
        \cap
        \ker B_\Lambda^{\rm cl}.
\]
By the assumed clean anti-phantom condition,
\[
        \mathfrak D_\Lambda^\ast
        \in
        \operatorname{Im}G_\Lambda^{\rm cl}.
\]
Therefore
\[
        [\mathfrak D_\Lambda^\ast]=0
        \quad\text{in }\mathcal Q_\Lambda^{\rm cl},
\]
contradicting the normalization
\[
        \|[\mathfrak D_\Lambda^\ast]\|_{\mathcal Q_\Lambda^{\rm cl}}=1.
\]
Thus the coercive estimate holds with some
\(c_\Lambda^{\rm cl}>0\).

Restricting the estimate to
\[
        \mathfrak D_\Lambda\in\mathcal K_\Lambda^{\rm cl}
        =
        \ker B_\Lambda^{\rm cl}
\]
gives
\[
        \|O_\Lambda^{\rm cl}\mathfrak D_\Lambda\|
        \ge
        c_\Lambda^{\rm cl}
        \operatorname{dist}
        (\mathfrak D_\Lambda,\operatorname{Im}G_\Lambda^{\rm cl}),
\]
and therefore
\[
        \gamma_\Lambda^{\rm cl}
        \ge
        c_\Lambda^{\rm cl}>0.
\]
\end{proof}

\subsection{Matrix form}

For later verification it is useful to write the criterion as a matrix
rank condition.  After choosing bases of
\[
        \mathcal D_\Lambda^{\rm cl},
        \qquad
        \operatorname{Im}G_\Lambda^{\rm cl},
        \qquad
        Y_\Lambda^{\rm cl},
        \qquad
        Z_\Lambda^{\rm cl},
\]
let
\[
        M_\Lambda^{\rm cl}
        :=
        \begin{pmatrix}
        O_\Lambda^{\rm cl}\\
        B_\Lambda^{\rm cl}
        \end{pmatrix}.
\]
Then the clean anti-phantom condition is equivalent to
\[
        \ker M_\Lambda^{\rm cl}
        \subset
        \operatorname{Im}G_\Lambda^{\rm cl}.
\]
Equivalently,
\[
        \ker M_\Lambda^{\rm cl}
        \cap
        \bigl(\operatorname{Im}G_\Lambda^{\rm cl}\bigr)^\perp
        =
        \{0\}.
\]
Thus the full finite-window coercivity constant can be computed as the
smallest singular value of \(M_\Lambda^{\rm cl}\) restricted to the
orthogonal complement of the gauge directions:
\[
\boxed{
        c_\Lambda^{\rm cl}
        =
        \sigma_{\min}
        \left(
        M_\Lambda^{\rm cl}
        \big|_{(\operatorname{Im}G_\Lambda^{\rm cl})^\perp}
        \right).
}
\]
In particular,
\[
        c_\Lambda^{\rm cl}>0
\]
if and only if this restricted matrix has full column rank.  The reduced
gap on the exact clean constraint set is correspondingly
\[
\boxed{
        \gamma_\Lambda^{\rm cl}
        =
        \sigma_{\min}
        \left(
        O_\Lambda^{\rm cl}
        \big|_{(\operatorname{Im}G_\Lambda^{\rm cl})^\perp
        \cap\ker B_\Lambda^{\rm cl}}
        \right),
}
\]
with the usual convention that the infimum over an empty unit sphere is
\(+\infty\).  Hence the full matrix gap implies the reduced dynamic gap,
and the reduced gap is exactly the observability singular value after the
clean constraints have already been imposed.

\begin{remark}[Why this is the right next reduction]
The clean periodic theorem does not yet treat the full localized
Navier--Stokes problem.  Its role is to isolate the algebraic obstruction:
after removing gauge directions, do the pressure, flux, energy, trace,
Navier--Stokes residual, reproduction, and zero-profit constraints leave
any invisible direction?  If the answer is no, then the localized theorem
can be obtained by perturbing this clean gap and subtracting the error
budget generated by localization, pressure tails, nonlinear remainders,
and truncation.
\end{remark}

\section{Construction of the clean finite-window matrix}\label{sec:clean-finite-window-matrix}

We now construct the clean finite-window matrix whose restricted singular
value gives the algebraic core of the dynamic anti-phantom criterion.  Since
the Navier--Stokes momentum, flux, pressure, and ledger relations contain
quadratic and cubic terms, the matrix below should be understood as the
linearized clean matrix around a reference finite-window configuration.

\subsection{Reference configuration and perturbation variables}

Fix a clean periodic finite window
\[
        \Lambda\subset \mathbb Z^3\setminus\{0\}
\]
and a reference defect configuration
\[
        \bar{\mathfrak D}_\Lambda
        =
        (\bar U_q,\bar P_q,\bar R_q,\bar \Phi_q,\bar \Pi_q,
        \bar L_q,\bar s_q)_{q\in\Lambda}.
\]
A perturbation is written as
\[
        \delta\mathfrak D_\Lambda
        =
        (\delta U_q,\delta P_q,\delta R_q,\delta \Phi_q,\delta \Pi_q,
        \delta L_q,\delta s_q)_{q\in\Lambda}.
\]
Here
\[
        q\cdot \delta U_q=0
        \qquad \text{for every }q\in\Lambda,
\]
and the real-valuedness condition is imposed by
\[
        \delta U_{-q}=\overline{\delta U_q},
        \qquad
        \delta P_{-q}=\overline{\delta P_q},
        \qquad
        \delta R_{-q}=\overline{\delta R_q}.
\]

Let
\[
        \mathbb P_q
        =
        I-\frac{q\otimes q}{|q|^2}
\]
denote the Leray projector in Fourier variables.  All velocity equations
are projected by \(\mathbb P_q\).

\subsection{Linearized incompressibility block}

Although incompressibility is built into the velocity space, it is useful
to display the corresponding constraint block:
\[
        M_\Lambda^{\rm div}\delta\mathfrak D_\Lambda
        :=
        (q\cdot \delta U_q)_{q\in\Lambda}.
\]
In a basis already adapted to the divergence-free subspace, this block is
identically zero and may be omitted.  If an unconstrained Fourier basis is
used, this block must be retained.

\subsection{Linearized momentum block}

The clean periodic coarse momentum residual has the schematic form
\[
        \partial_t U
        -
        \Delta U
        +
        \mathbb P\nabla\cdot (U\otimes U+R)
        =
        0.
\]
On a finite Fourier window, the linearization around
\(\bar{\mathfrak D}_\Lambda\) is
\[
\begin{aligned}
        (M_\Lambda^{\rm mom}\delta\mathfrak D_\Lambda)_q
        :=
        &\ \partial_t \delta U_q
        + |q|^2\delta U_q                                      \\
        &+
        i\mathbb P_q
        \sum_{\substack{p+r=q\\p,r\in\Lambda}}
        \Bigl[
        (\bar U_p\cdot r)\delta U_r
        +
        (\delta U_p\cdot r)\bar U_r
        \Bigr]                                                   \\
        &+
        i\mathbb P_q
        \sum_{\substack{p+r=q\\p,r\in\Lambda}}
        r_j\,\delta R_{p,j\cdot}.
\end{aligned}
\]
Equivalently, this block records the first variation of the resolved
transport and Reynolds-stress forcing.

\subsection{Linearized pressure-compatibility block}

In the periodic clean model, pressure is determined by the Poisson relation
\[
        -\Delta P
        =
        \partial_i\partial_j(U_iU_j+R_{ij}).
\]
For each \(q\in\Lambda\), the linearized pressure compatibility condition is
\[
\begin{aligned}
        (M_\Lambda^{\rm prs}\delta\mathfrak D_\Lambda)_q
        :=
        &\ |q|^2\delta P_q                                      \\
        &+
        q_iq_j
        \sum_{\substack{p+r=q\\p,r\in\Lambda}}
        \left(
        \bar U_{p,i}\delta U_{r,j}
        +
        \delta U_{p,i}\bar U_{r,j}
        +
        \delta R_{p,ij}
        \right).
\end{aligned}
\]
This block removes pressure directions which are not compatible with the
velocity and Reynolds-stress perturbations.

\subsection{Linearized flux block}

Let
\[
        \mathsf F_q(U,R)
\]
denote the clean finite-window flux functional at mode \(q\).  Its precise
form depends on the chosen dyadic shell projector.  In the periodic model
one may take it to be the shell flux generated by
\[
        \int_{\mathbb T^3}
        U_{\le q}\cdot\nabla\cdot(U_{>q}\otimes U_{>q}+R_{>q})
        \,dx .
\]
The linearized flux block is defined by
\[
        M_\Lambda^{\rm flux}\delta\mathfrak D_\Lambda
        :=
        D_{\bar{\mathfrak D}_\Lambda}\mathsf F_\Lambda
        [\delta\mathfrak D_\Lambda].
\]
Equivalently, each component is the sum of the first variations obtained by
replacing exactly one factor in the flux expression by its perturbation.  When
this block is used as an observation channel below, we write
\(O_\Lambda^{\rm flux}:=M_\Lambda^{\rm flux}\).

\subsection{Linearized energy-ledger block}

The clean ledger relation has the schematic form
\[
        \mathsf B_{k+1}
        -
        (1-\lambda)\mathsf B_k
        -
        \mathsf{Supp}^{\rm cl}_k
        +
        \mathsf{Tax}^{\rm cl}_k
        =
        0.
\]
The linearized energy-ledger block is
\[
        M_\Lambda^{\rm led}\delta\mathfrak D_\Lambda
        :=
        D_{\bar{\mathfrak D}_\Lambda}
        \left[
        \mathsf B_{k+1}
        -
        (1-\lambda)\mathsf B_k
        -
        \mathsf{Supp}^{\rm cl}_k
        +
        \mathsf{Tax}^{\rm cl}_k
        \right]
        [\delta\mathfrak D_\Lambda]_{k\in\Lambda}.
\]
This block records whether the perturbation respects the leading-order
supply--tax balance of the critical ledger.

\subsection{Linearized positive-energy observation block}

The positive-energy channel tests the Reynolds-covariance and resolved
energy components against a finite family of nonnegative test tensors
\[
        \{Q_\alpha\}_{\alpha\in A_\Lambda}.
\]
Define
\[
        (O_\Lambda^{\rm en}\delta\mathfrak D_\Lambda)_\alpha
        :=
        \sum_{q\in\Lambda}
        \langle \delta R_q,Q_{\alpha,q}\rangle
        +
        2\operatorname{Re}
        \sum_{q\in\Lambda}
        \langle \bar U_q,\delta U_q\rangle Q_{\alpha,q}.
\]
This block is designed to detect positive Reynolds-covariance defects and
their resolved-energy traces.

\subsection{Linearized active-pressure observation block}

Let
\[
        \{\psi_\beta\}_{\beta\in B_\Lambda}
\]
be a finite family of active pressure test functions.  The pressure
observation block is
\[
        (O_\Lambda^{\rm prs}\delta\mathfrak D_\Lambda)_\beta
        :=
        \sum_{q\in\Lambda}
        \delta P_q\,\overline{\widehat\psi_\beta(q)}.
\]
When the pressure is eliminated through the Poisson block, this becomes an
induced observation on \((\delta U,\delta R)\).

\subsection{Linearized trace block}

Let
\[
        \{\varphi_\eta\}_{\eta\in T_\Lambda}
\]
be selected-time or adjoint trace probes.  The trace block is
\[
        (O_\Lambda^{\rm tr}\delta\mathfrak D_\Lambda)_\eta
        :=
        \sum_{q\in\Lambda}
        \left\langle
        \delta U_q(t_\eta),
        \widehat\varphi_\eta(q)
        \right\rangle
        +
        \sum_{q\in\Lambda}
        \left\langle
        \delta R_q(t_\eta),
        \widehat\Psi_\eta(q)
        \right\rangle .
\]
This block prevents a defect from hiding only between the energy and
pressure observation times.

\subsection{Linearized reproduction block}

Let
\[
        \mathcal R_k
\]
be the clean scale-to-scale reproduction map.  The exact reproduction
condition is
\[
        \mathfrak D_{k+1}=\mathcal R_k(\mathfrak D_k).
\]
Linearizing around the reference configuration gives
\[
        (M_\Lambda^{\rm rep}\delta\mathfrak D_\Lambda)_k
        :=
        \delta\mathfrak D_{k+1}
        -
        D_{\bar{\mathfrak D}_k}\mathcal R_k
        [\delta\mathfrak D_k],
        \qquad k=k_0,\ldots,k_0+L-1.
\]
This block kills one-scale phantoms which cannot reproduce coherently
across the window.

\subsection{Linearized zero-profit block}

Let
\[
        \mathsf P_\Lambda^{0}
        =
        \mathsf{Supp}_\Lambda^{\rm cl}
        -
        \mathsf{Tax}_\Lambda^{\rm cl}
\]
be the leading-order clean net profit.  The zero-profit block is
\[
        M_\Lambda^{\rm prof}\delta\mathfrak D_\Lambda
        :=
        D_{\bar{\mathfrak D}_\Lambda}
        \mathsf P_\Lambda^{0}
        [\delta\mathfrak D_\Lambda].
\]
It removes perturbations which create positive untaxed supply at first
order.

\subsection{The full clean linearized matrix}

The clean linearized dynamic anti-phantom matrix is
\[
\boxed{
        M_{\Lambda,\bar{\mathfrak D}}^{\rm cl}
        :=
        \begin{pmatrix}
        O_\Lambda^{\rm prs}\\
        O_\Lambda^{\rm flux}\\
        O_\Lambda^{\rm en}\\
        O_\Lambda^{\rm tr}\\
        M_\Lambda^{\rm div}\\
        M_\Lambda^{\rm mom}\\
        M_\Lambda^{\rm prs}\\
        M_\Lambda^{\rm led}\\
        M_\Lambda^{\rm rep}\\
        M_\Lambda^{\rm prof}
        \end{pmatrix}.
}
\]
The corresponding clean gauge matrix is
\[
        G_{\Lambda}^{\rm cl}.
\]
The finite-window linear anti-phantom condition is
\[
\boxed{
        \ker M_{\Lambda,\bar{\mathfrak D}}^{\rm cl}
        \cap
        \bigl(\operatorname{Im}G_\Lambda^{\rm cl}\bigr)^\perp
        =
        \{0\}.
}
\]
Equivalently,
\[
\boxed{
        \sigma_{\min}
        \left(
        M_{\Lambda,\bar{\mathfrak D}}^{\rm cl}
        \big|_{(\operatorname{Im}G_\Lambda^{\rm cl})^\perp}
        \right)
        >0.
}
\]

\begin{definition}[Clean linear anti-phantom gap]
We define
\[
        \gamma_{\Lambda,\bar{\mathfrak D}}^{\rm lin}
        :=
        \sigma_{\min}
        \left(
        M_{\Lambda,\bar{\mathfrak D}}^{\rm cl}
        \big|_{(\operatorname{Im}G_\Lambda^{\rm cl})^\perp}
        \right).
\]
If
\[
        \gamma_{\Lambda,\bar{\mathfrak D}}^{\rm lin}>0,
\]
then the reference configuration is called linearly clean anti-phantom on
the window \(\Lambda\).
\end{definition}

\begin{remark}[Interpretation]
The condition
\[
        \gamma_{\Lambda,\bar{\mathfrak D}}^{\rm lin}>0
\]
says that every first-order perturbation which is not a gauge artifact must
be detected by at least one of the following channels:
\[
\begin{gathered}
        \text{active pressure,}\quad
        \text{flux,}\quad
        \text{positive energy,}\quad
        \text{trace,}\\
        \text{Navier--Stokes residual,}\quad
        \text{ledger residual,}\quad
        \text{reproduction residual,}\quad
        \text{or profit variation.}
\end{gathered}
\]
Thus no infinitesimal non-gauge defect can be simultaneously invisible,
Navier--Stokes-compatible, reproducible, and non-profitable.
\end{remark}

\section{From the linearized gap to a nonlinear finite-window gap}\label{sec:linearized-to-nonlinear-gap}

The clean matrix constructed in the previous section gives an infinitesimal
anti-phantom criterion.  We now show that, under a quantitative quadratic
remainder bound, this infinitesimal gap excludes nonlinear near-phantoms in
a sufficiently small neighborhood of the reference configuration.

\subsection{The nonlinear clean residual map}

Let
\[
        \mathscr F_\Lambda^{\rm cl}
        :
        \mathcal D_\Lambda^{\rm cl}
        \longrightarrow
        \mathcal Y_\Lambda^{\rm cl}
\]
denote the full clean nonlinear detection map
\[
        \mathscr F_\Lambda^{\rm cl}
        :=
        \left(
        O_\Lambda^{\rm prs},
        O_\Lambda^{\rm flux},
        O_\Lambda^{\rm en},
        O_\Lambda^{\rm tr},
        \mathcal E_\Lambda^{\rm div},
        \mathcal E_\Lambda^{\rm mom},
        \mathcal E_\Lambda^{\rm prs},
        \mathcal E_\Lambda^{\rm led},
        \operatorname{Rep}_\Lambda,
        \mathsf P_\Lambda^0
        \right).
\]
Thus
\[
        \mathscr F_\Lambda^{\rm cl}(\mathfrak D_\Lambda)
\]
collects all observability, Navier--Stokes residual, ledger, reproduction,
and leading zero-profit channels in the clean periodic finite-window model.

Let \(\bar{\mathfrak D}_\Lambda\) be a reference finite-window
configuration.  We assume that
\[
        \mathscr F_\Lambda^{\rm cl}
\]
is \(C^2\) in a neighborhood of \(\bar{\mathfrak D}_\Lambda\), and that
\[
        D_{\bar{\mathfrak D}_\Lambda}
        \mathscr F_\Lambda^{\rm cl}
        =
        M_{\Lambda,\bar{\mathfrak D}}^{\rm cl}.
\]

We write a perturbation as
\[
        \mathfrak D_\Lambda
        =
        \bar{\mathfrak D}_\Lambda+h_\Lambda .
\]
After quotienting by the clean gauge directions, we choose the gauge-fixed
representative
\[
        h_\Lambda^\perp
        \in
        \bigl(\operatorname{Im}G_\Lambda^{\rm cl}\bigr)^\perp
\]
so that
\[
        \|h_\Lambda^\perp\|
        =
        \operatorname{dist}
        (\mathfrak D_\Lambda-\bar{\mathfrak D}_\Lambda,
        \operatorname{Im}G_\Lambda^{\rm cl}).
\]

\subsection{Quadratic remainder bound}

The Taylor expansion of the nonlinear clean map gives
\[
        \mathscr F_\Lambda^{\rm cl}
        (\bar{\mathfrak D}_\Lambda+h_\Lambda^\perp)
        =
        \mathscr F_\Lambda^{\rm cl}(\bar{\mathfrak D}_\Lambda)
        +
        M_{\Lambda,\bar{\mathfrak D}}^{\rm cl}h_\Lambda^\perp
        +
        \mathscr R_{\Lambda,\bar{\mathfrak D}}(h_\Lambda^\perp),
\]
where the nonlinear remainder satisfies
\[
        \mathscr R_{\Lambda,\bar{\mathfrak D}}(h_\Lambda^\perp)
        =
        O(\|h_\Lambda^\perp\|^2).
\]

We make this quantitative.

\begin{hypothesis}[Finite-window quadratic remainder bound]
\label{hyp:quadratic-remainder}
There exist constants \(K_{\Lambda,\bar{\mathfrak D}}>0\) and
\(r_{\Lambda,\bar{\mathfrak D}}>0\) such that whenever
\[
        \|h_\Lambda^\perp\|
        \le r_{\Lambda,\bar{\mathfrak D}},
\]
one has
\[
        \left\|
        \mathscr R_{\Lambda,\bar{\mathfrak D}}(h_\Lambda^\perp)
        \right\|
        \le
        K_{\Lambda,\bar{\mathfrak D}}
        \|h_\Lambda^\perp\|^2 .
\]
\end{hypothesis}

This hypothesis is automatic in a finite-dimensional window if
\(\mathscr F_\Lambda^{\rm cl}\) is \(C^2\).  The point is to keep track of
the constant \(K_{\Lambda,\bar{\mathfrak D}}\), since this constant will
later compete with the linear anti-phantom gap.

\subsection{Local nonlinear anti-phantom theorem}

\begin{theorem}[Linear gap implies local nonlinear clean anti-phantom]
\label{thm:linear-gap-to-nonlinear-gap}
Assume that the clean linear anti-phantom gap satisfies
\[
        \gamma_{\Lambda,\bar{\mathfrak D}}^{\rm lin}
        :=
        \sigma_{\min}
        \left(
        M_{\Lambda,\bar{\mathfrak D}}^{\rm cl}
        \big|_{(\operatorname{Im}G_\Lambda^{\rm cl})^\perp}
        \right)
        >0 .
\]
Assume also the quadratic remainder bound
\[
        \left\|
        \mathscr R_{\Lambda,\bar{\mathfrak D}}(h_\Lambda^\perp)
        \right\|
        \le
        K_{\Lambda,\bar{\mathfrak D}}
        \|h_\Lambda^\perp\|^2 .
\]
Let
\[
        0<\rho_{\Lambda,\bar{\mathfrak D}}
        \le
        \min\left\{
        r_{\Lambda,\bar{\mathfrak D}},
        \frac{\gamma_{\Lambda,\bar{\mathfrak D}}^{\rm lin}}
        {2K_{\Lambda,\bar{\mathfrak D}}}
        \right\}.
\]
Then, for every gauge-fixed perturbation satisfying
\[
        \|h_\Lambda^\perp\|
        \le
        \rho_{\Lambda,\bar{\mathfrak D}},
\]
one has
\[
\boxed{
        \left\|
        \mathscr F_\Lambda^{\rm cl}
        (\bar{\mathfrak D}_\Lambda+h_\Lambda^\perp)
        -
        \mathscr F_\Lambda^{\rm cl}(\bar{\mathfrak D}_\Lambda)
        \right\|
        \ge
        \frac12
        \gamma_{\Lambda,\bar{\mathfrak D}}^{\rm lin}
        \|h_\Lambda^\perp\| .
}
\]
Equivalently,
\[
\boxed{
        \left\|
        \mathscr F_\Lambda^{\rm cl}
        (\mathfrak D_\Lambda)
        -
        \mathscr F_\Lambda^{\rm cl}(\bar{\mathfrak D}_\Lambda)
        \right\|
        \ge
        \frac12
        \gamma_{\Lambda,\bar{\mathfrak D}}^{\rm lin}
        \operatorname{dist}
        (\mathfrak D_\Lambda-\bar{\mathfrak D}_\Lambda,
        \operatorname{Im}G_\Lambda^{\rm cl})
}
\]
for every \(\mathfrak D_\Lambda\) in the corresponding quotient
neighborhood.
\end{theorem}

\begin{proof}
By Taylor expansion,
\[
        \mathscr F_\Lambda^{\rm cl}
        (\bar{\mathfrak D}_\Lambda+h_\Lambda^\perp)
        -
        \mathscr F_\Lambda^{\rm cl}(\bar{\mathfrak D}_\Lambda)
        =
        M_{\Lambda,\bar{\mathfrak D}}^{\rm cl}h_\Lambda^\perp
        +
        \mathscr R_{\Lambda,\bar{\mathfrak D}}(h_\Lambda^\perp).
\]
Therefore
\[
\begin{aligned}
        \left\|
        \mathscr F_\Lambda^{\rm cl}
        (\bar{\mathfrak D}_\Lambda+h_\Lambda^\perp)
        -
        \mathscr F_\Lambda^{\rm cl}(\bar{\mathfrak D}_\Lambda)
        \right\|
        &\ge
        \left\|
        M_{\Lambda,\bar{\mathfrak D}}^{\rm cl}h_\Lambda^\perp
        \right\|
        -
        \left\|
        \mathscr R_{\Lambda,\bar{\mathfrak D}}(h_\Lambda^\perp)
        \right\|                                      \\
        &\ge
        \gamma_{\Lambda,\bar{\mathfrak D}}^{\rm lin}
        \|h_\Lambda^\perp\|
        -
        K_{\Lambda,\bar{\mathfrak D}}
        \|h_\Lambda^\perp\|^2                         \\
        &=
        \left(
        \gamma_{\Lambda,\bar{\mathfrak D}}^{\rm lin}
        -
        K_{\Lambda,\bar{\mathfrak D}}\|h_\Lambda^\perp\|
        \right)
        \|h_\Lambda^\perp\|.
\end{aligned}
\]
If
\[
        \|h_\Lambda^\perp\|
        \le
        \frac{\gamma_{\Lambda,\bar{\mathfrak D}}^{\rm lin}}
        {2K_{\Lambda,\bar{\mathfrak D}}},
\]
then
\[
        \gamma_{\Lambda,\bar{\mathfrak D}}^{\rm lin}
        -
        K_{\Lambda,\bar{\mathfrak D}}\|h_\Lambda^\perp\|
        \ge
        \frac12
        \gamma_{\Lambda,\bar{\mathfrak D}}^{\rm lin}.
\]
This proves the desired estimate.
\end{proof}

\begin{corollary}[No local nonlinear near-phantoms]
\label{cor:no-local-nonlinear-near-phantoms}
Under the assumptions of
Theorem~\ref{thm:linear-gap-to-nonlinear-gap}, there is no sequence
\[
        \mathfrak D_\Lambda^{(n)}
        =
        \bar{\mathfrak D}_\Lambda+h_\Lambda^{(n)}
\]
such that
\[
        0<
        \operatorname{dist}
        (h_\Lambda^{(n)},\operatorname{Im}G_\Lambda^{\rm cl})
        \le
        \rho_{\Lambda,\bar{\mathfrak D}},
\]
and
\[
        \frac{
        \left\|
        \mathscr F_\Lambda^{\rm cl}
        (\mathfrak D_\Lambda^{(n)})
        -
        \mathscr F_\Lambda^{\rm cl}(\bar{\mathfrak D}_\Lambda)
        \right\|
        }
        {
        \operatorname{dist}
        (h_\Lambda^{(n)},\operatorname{Im}G_\Lambda^{\rm cl})
        }
        \longrightarrow 0 .
\]
Thus a nonlinear perturbation near \(\bar{\mathfrak D}_\Lambda\) cannot be
both non-gauge and invisible to all clean dynamic channels.
\end{corollary}

\begin{proof}
The conclusion follows immediately from
Theorem~\ref{thm:linear-gap-to-nonlinear-gap}, since the ratio is bounded
from below by
\[
        \frac12
        \gamma_{\Lambda,\bar{\mathfrak D}}^{\rm lin}>0.
\]
\end{proof}

\subsection{Exact reference configurations}

If the reference configuration is itself an exact clean dynamic solution,
namely
\[
        \mathscr F_\Lambda^{\rm cl}(\bar{\mathfrak D}_\Lambda)=0,
\]
then Theorem~\ref{thm:linear-gap-to-nonlinear-gap} simplifies to
\[
\boxed{
        \left\|
        \mathscr F_\Lambda^{\rm cl}
        (\mathfrak D_\Lambda)
        \right\|
        \ge
        \frac12
        \gamma_{\Lambda,\bar{\mathfrak D}}^{\rm lin}
        \operatorname{dist}
        (\mathfrak D_\Lambda-\bar{\mathfrak D}_\Lambda,
        \operatorname{Im}G_\Lambda^{\rm cl})
}
\]
for all sufficiently small perturbations of
\(\bar{\mathfrak D}_\Lambda\).

In particular, if
\[
        \mathscr F_\Lambda^{\rm cl}
        (\mathfrak D_\Lambda)=0
\]
and
\[
        \mathfrak D_\Lambda
\]
is sufficiently close to
\(\bar{\mathfrak D}_\Lambda\), then
\[
        \operatorname{dist}
        (\mathfrak D_\Lambda-\bar{\mathfrak D}_\Lambda,
        \operatorname{Im}G_\Lambda^{\rm cl})
        =0.
\]
Hence every nearby exact clean dynamic phantom is gauge-equivalent to the
reference configuration.

\begin{remark}[Meaning of the nonlinear lift]
The linear matrix detects infinitesimal invisible directions.  The
quadratic remainder estimate says that nonlinear interactions are too
small, inside a sufficiently small ball, to cancel this linear detection.
Thus the finite-window anti-phantom mechanism is stable under nonlinear
Navier--Stokes corrections, provided the perturbation size is below the
gap-controlled radius
\[
        \rho_{\Lambda,\bar{\mathfrak D}}
        \sim
        \frac{
        \gamma_{\Lambda,\bar{\mathfrak D}}^{\rm lin}
        }
        {
        K_{\Lambda,\bar{\mathfrak D}}
        }.
\]
This is the first point where the size of the quotient gap matters
quantitatively.
\end{remark}

\section{Localized perturbative dynamic anti-phantom transfer}\label{sec:localized-transfer}

This section records the clean-to-localized transfer step.  It does not
prove the pressure, cutoff, truncation, gauge, or reproduction estimates
inside a local Navier--Stokes window.  Those remain explicit PDE inputs.
The purpose of the section is narrower: it isolates the perturbative lemma
showing that, once the localized transfer hypotheses are available, the
clean nonlinear finite-window gap implies the localized dynamic
anti-phantom inequality, paralleling the strict-shadow and harmonic-pressure quotient viewpoint of \cite{Yu2026HarmonicPressure,Yu2026StrictShadows,Yu2026SchurVisibility}.

Let \(\Lambda\) be a finite dyadic window and let
\[
        \mathcal D_\Lambda^{\rm loc}
\]
denote the localized finite-window defect space generated by the
coarse-grained Navier--Stokes ledger variables on the local parabolic
window.  The localized cleaning map is
\[
        G_\Lambda^{\rm loc}:
        \mathcal C_\Lambda^{\rm loc}
        \longrightarrow
        \mathcal D_\Lambda^{\rm loc}.
\]
We write
\[
        \Dist_{\rm loc}(\mathfrak D,\Image G_\Lambda^{\rm loc})
        :=
        \Dist(\mathfrak D,\Image G_\Lambda^{\rm loc})
\]
for the quotient distance in the chosen localized norm.

The localized detection map is
\[
        \mathscr F_\Lambda^{\rm loc}(\mathfrak D)
        :=
        \left(
        O_\Lambda^{\rm loc}\mathfrak D,\,
        \mathcal E_\Lambda^{\rm loc}(\mathfrak D),\,
        \operatorname{Rep}_\Lambda^{\rm loc}(\mathfrak D),\,
        [\mathsf P_\Lambda^{\rm loc}(\mathfrak D)]_+
        \right).
\]
Its output norm is fixed as
\[
\begin{aligned}
        \left\|\mathscr F_\Lambda^{\rm loc}(\mathfrak D)\right\|_{\rm loc}
        &:=
        \left\|O_\Lambda^{\rm loc}\mathfrak D\right\|
        +
        C_E
        \left\|\mathcal E_\Lambda^{\rm loc}(\mathfrak D)\right\|  \\
        &\quad+
        C_R
        \operatorname{Rep}_\Lambda^{\rm loc}(\mathfrak D)
        +
        [\mathsf P_\Lambda^{\rm loc}(\mathfrak D)]_+ ,
\end{aligned}
\]
where \(C_E,C_R>0\) are fixed bookkeeping weights.  The clean detection
map \(\mathscr F_\Lambda^{\rm cl}\) is measured in the corresponding
clean output norm.

\begin{assumption}[Localized transfer hypotheses]\label{ass:localized-transfer-hypotheses}
Let
\[
        \Theta_\Lambda:
        \mathcal D_\Lambda^{\rm loc}
        \longrightarrow
        \mathcal D_\Lambda^{\rm cl}
\]
be a local-to-clean chart on a class of localized defect packages
\(\mathcal U_\Lambda^{\rm loc}\subset\mathcal D_\Lambda^{\rm loc}\).
Assume the following estimates hold for every
\(\mathfrak D\in\mathcal U_\Lambda^{\rm loc}\).

\begin{enumerate}[label=(\roman*),leftmargin=2em]
    \item \textbf{Clean nonlinear gap.}
    \[
        \left\|
        \mathscr F_\Lambda^{\rm cl}
        (\Theta_\Lambda\mathfrak D)
        \right\|_{\rm cl}
        \ge
        c_\Lambda^{\rm cl}
        \Dist_{\rm cl}
        \left(
        \Theta_\Lambda\mathfrak D,
        \Image G_\Lambda^{\rm cl}
        \right),
        \qquad c_\Lambda^{\rm cl}>0 .
    \]

    \item \textbf{Quotient-distance comparison.}
    \[
        \Dist_{\rm cl}
        \left(
        \Theta_\Lambda\mathfrak D,
        \Image G_\Lambda^{\rm cl}
        \right)
        \ge
        (1-\eps_G)
        \Dist_{\rm loc}
        \left(
        \mathfrak D,
        \Image G_\Lambda^{\rm loc}
        \right)
        -
        \delta_G ,
    \]
    with \(0\leq\eps_G<1\) and \(\delta_G\ge0\).

    \item \textbf{Detection-map comparison.}
    There is an error functional \(\Err_\Lambda(\mathfrak D)\ge0\) such that
    \[
        \left\|
        \mathscr F_\Lambda^{\rm loc}(\mathfrak D)
        \right\|_{\rm loc}
        \ge
        \left\|
        \mathscr F_\Lambda^{\rm cl}
        (\Theta_\Lambda\mathfrak D)
        \right\|_{\rm cl}
        -
        \Err_\Lambda(\mathfrak D).
    \]

    \item \textbf{Normalized error budget.}
    \[
        \Err_\Lambda(\mathfrak D)
        \le
        \eta_\Lambda
        \Dist_{\rm loc}
        \left(
        \mathfrak D,
        \Image G_\Lambda^{\rm loc}
        \right)
        +
        \Delta_\Lambda .
    \]
\end{enumerate}
\end{assumption}

\begin{remark}[What is included in the error]\label{rem:localized-error-components}
In the local Navier--Stokes application the error decomposes schematically as
\[
        \Err_\Lambda
        =
        \Err_{\rm prs}
        +
        \Err_{\rm loc}
        +
        \Err_{\rm tr}
        +
        \Err_{\rm nl}
        +
        \Err_{\rm rep}
        +
        \Err_{\rm gauge}
        +
        \Err_{\rm prof}.
\]
The pressure part contains the harmonic tail and cutoff-Riesz commutator;
the localization part contains cutoff and moving-window leakage; the
truncation part contains the modes outside the finite window; the nonlinear
part is the quadratic Taylor remainder; the reproduction part is the scale
drift; and the gauge part measures the loss in passing between local and
clean quotient representatives.  The profit term is harmless once the
underlying ledger functional has been compared, because the scalar map
\(x\mapsto [x]_+\) is \(1\)-Lipschitz:
\[
        \left|[a]_+-[b]_+\right|
        \le |a-b|.
\]
\end{remark}

\begin{lemma}[Abstract perturbative localized transfer]\label{lem:abstract-localized-transfer}
Assume Hypothesis~\ref{ass:localized-transfer-hypotheses}.  Define
\[
        c_\Lambda^{\rm loc}
        :=
        c_\Lambda^{\rm cl}(1-\eps_G)-\eta_\Lambda
\]
and
\[
        \Delta_\Lambda'
        :=
        \Delta_\Lambda+c_\Lambda^{\rm cl}\delta_G .
\]
Then every
\(\mathfrak D\in\mathcal U_\Lambda^{\rm loc}\) satisfies
\[
\boxed{
        \left\|
        \mathscr F_\Lambda^{\rm loc}(\mathfrak D)
        \right\|_{\rm loc}
        \ge
        c_\Lambda^{\rm loc}
        \Dist_{\rm loc}
        \left(
        \mathfrak D,
        \Image G_\Lambda^{\rm loc}
        \right)
        -
        \Delta_\Lambda' .
}
\]
In particular, if
\[
        \eta_\Lambda
        <
        c_\Lambda^{\rm cl}(1-\eps_G),
\]
then \(c_\Lambda^{\rm loc}>0\).
\end{lemma}

\begin{proof}
Let
\[
        d_{\rm loc}
        :=
        \Dist_{\rm loc}
        \left(
        \mathfrak D,
        \Image G_\Lambda^{\rm loc}
        \right).
\]
The detection-map comparison gives
\[
        \left\|
        \mathscr F_\Lambda^{\rm loc}(\mathfrak D)
        \right\|_{\rm loc}
        \ge
        \left\|
        \mathscr F_\Lambda^{\rm cl}
        (\Theta_\Lambda\mathfrak D)
        \right\|_{\rm cl}
        -
        \Err_\Lambda(\mathfrak D).
\]
Using the clean nonlinear gap,
\[
        \left\|
        \mathscr F_\Lambda^{\rm loc}(\mathfrak D)
        \right\|_{\rm loc}
        \ge
        c_\Lambda^{\rm cl}
        \Dist_{\rm cl}
        \left(
        \Theta_\Lambda\mathfrak D,
        \Image G_\Lambda^{\rm cl}
        \right)
        -
        \Err_\Lambda(\mathfrak D).
\]
The quotient-distance comparison therefore implies
\[
        \left\|
        \mathscr F_\Lambda^{\rm loc}(\mathfrak D)
        \right\|_{\rm loc}
        \ge
        c_\Lambda^{\rm cl}(1-\eps_G)d_{\rm loc}
        -
        c_\Lambda^{\rm cl}\delta_G
        -
        \Err_\Lambda(\mathfrak D).
\]
Finally, the normalized error budget gives
\[
        \left\|
        \mathscr F_\Lambda^{\rm loc}(\mathfrak D)
        \right\|_{\rm loc}
        \ge
        \left[
        c_\Lambda^{\rm cl}(1-\eps_G)-\eta_\Lambda
        \right]d_{\rm loc}
        -
        \left(
        \Delta_\Lambda+c_\Lambda^{\rm cl}\delta_G
        \right).
\]
This is exactly the claimed estimate.
\end{proof}

\begin{theorem}[Conditional perturbative localized dynamic anti-phantom transfer]
\label{thm:localized-dynamic-antiphantom-transfer}
Assume Hypothesis~\ref{ass:localized-transfer-hypotheses} and suppose
\[
        \eta_\Lambda
        <
        c_\Lambda^{\rm cl}(1-\eps_G).
\]
Then every localized finite-window defect
\(\mathfrak D\in\mathcal U_\Lambda^{\rm loc}\) satisfies
\[
\boxed{
\begin{aligned}
        &\left\|O_\Lambda^{\rm loc}\mathfrak D\right\|
        +
        C_E
        \left\|\mathcal E_\Lambda^{\rm loc}(\mathfrak D)\right\|
        +
        C_R
        \operatorname{Rep}_\Lambda^{\rm loc}(\mathfrak D) \\
        &\qquad
        +
        [\mathsf P_\Lambda^{\rm loc}(\mathfrak D)]_+
        \ge
        c_\Lambda^{\rm loc}
        \Dist_{\rm loc}
        \left(
        \mathfrak D,
        \Image G_\Lambda^{\rm loc}
        \right)
        -
        \Delta_\Lambda',
\end{aligned}
}
\]
where
\[
        c_\Lambda^{\rm loc}
        =
        c_\Lambda^{\rm cl}(1-\eps_G)-\eta_\Lambda
        >0,
        \qquad
        \Delta_\Lambda'
        =
        \Delta_\Lambda+c_\Lambda^{\rm cl}\delta_G .
\]
Equivalently,
\[
\boxed{
\begin{aligned}
        &\left\|O_\Lambda^{\rm loc}\mathfrak D\right\|
        +
        C_E
        \left\|\mathcal E_\Lambda^{\rm loc}(\mathfrak D)\right\|
        +
        C_R
        \operatorname{Rep}_\Lambda^{\rm loc}(\mathfrak D) \\
        &\qquad\ge
        c_\Lambda^{\rm loc}
        \Dist_{\rm loc}
        \left(
        \mathfrak D,
        \Image G_\Lambda^{\rm loc}
        \right)
        -
        [\mathsf P_\Lambda^{\rm loc}(\mathfrak D)]_+
        -
        \Delta_\Lambda' .
\end{aligned}
}
\]
\end{theorem}

\begin{proof}
The first displayed inequality is Lemma~\ref{lem:abstract-localized-transfer}
with the definition of the localized output norm expanded.  The strict
smallness condition on \(\eta_\Lambda\) gives \(c_\Lambda^{\rm loc}>0\).
Moving the nonnegative profit term to the right-hand side gives the
controlled ledger-profit form.
\end{proof}

\begin{remark}[Interpretation]\label{rem:localized-transfer-interpretation}
The theorem says that a localized finite-window defect which is far from
the localized cleaning image cannot simultaneously have low
pressure--flux--energy--trace visibility, small localized Navier--Stokes
residual, small reproduction residual, and no positive ledger profit,
except at the size of the pressure, cutoff, truncation, nonlinear,
reproduction, and gauge error budget.  Thus a surviving non-gauge localized
defect must either be detected, fail the equations, fail to reproduce,
carry controlled untaxed profit, or be accounted for by the localized error
terms.
\end{remark}

\section{Positive-cone clean finite-window anti-phantom criterion}\label{sec:positive-cone-antiphantom}
\noindent This section is the most PDE-structured finite-window result: on the NS-realizable positive Reynolds-covariance cone, positive energy observability already rules out true phantoms.\label{subsec:ns-positive-antiphantom}

The quotient criteria above are deliberately formal.  They say that finite-window anti-phantom is a spectral-gap problem after the cleaning directions have been quotiented out.  We now record a more PDE-structured result in the clean periodic setting.  It does not solve the full formal kernel problem, because sign-changing artificial stress directions may remain in the purely algebraic residual space.  Instead, it proves that the most physical residual directions, namely NS-realizable positive Reynolds-covariance directions, cannot be true phantoms.

Let
\[
    \Lambda\subset \Z^3\setminus\{0\}
\]
be a finite symmetric active Fourier window on \(\mathbb T^3\), and fix a finite-dimensional time profile space.  Let \(Y_\Lambda\) denote the cleaned finite-window defect space consisting of tuples
\[
    d=(\dot U,\dot P,\dot R,\dot\Pi)
\]
with \(\dot U\) divergence-free, \(\dot R\) symmetric, and the pressure and flux components tied to the background coarse package \((U^\sharp,P^\sharp,R^\sharp,\Pi^\sharp)\) by the linearized compatibility identities
\begin{equation}\label{eq:pc-positive-cone}
    -\Delta \dot P
    =
    \partial_i\partial_j
    \bigl(U_i^\sharp \dot U_j+
          \dot U_i U_j^\sharp+
          \dot R_{ij}\bigr),
\end{equation}
and
\begin{equation}\label{eq:fi-positive-cone}
    \dot\Pi
    =
    -\dot R:\nabla U^\sharp
    -R^\sharp:\nabla\dot U .
\end{equation}
Since \(0\notin\Lambda\), the pressure representative is the unique zero-mean periodic one, and the harmonic pressure gauge has been removed.

Define the NS-realizable positive cone \(Y_{\Lambda,\NS}^{+}\subset Y_\Lambda\) to be the closed finite-window cone of defect directions obtained as limits of cleaned NS-derived coarse-grained packages satisfying the additional covariance positivity condition
\begin{equation}\label{eq:positive-cone-condition}
    \dot R(x,t)\ge 0
    \quad\text{as a quadratic form on the observation core.}
\end{equation}
This is the finite-window version of the Reynolds covariance positivity inherited from coarse graining \cite{ConstantinETiti1994,DuchonRobert2000,LeslieShvydkoy2018}:
\[
    R_{n,\ell}=S_\ell(u^{(n)}\otimes u^{(n)})-S_\ell u^{(n)}\otimes S_\ell u^{(n)}\ge0.
\]

Let the pressure and flux observations be the full active coefficient maps
\[
    O^P_\Lambda d=\dot P,
    \qquad
    O^F_\Lambda d=\dot\Pi,
\]
written in fixed finite bases.  Let the positive energy observation contain nonnegative bulk and selected-time tests of the form
\begin{equation}\label{eq:positive-energy-seminorm}
    Q^E_\Lambda(d)
    :=
    \left(
        \sum_a \int \theta_a(x)|\dot U(x,s_a)|^2\,dx
    \right)^{1/2}
    +
    \left(
        \sum_b \iint \zeta_b |\nabla\dot U|^2\,dx\,dt
    \right)^{1/2}
    +
    \sum_b \iint \zeta_b\,\frac12\operatorname{tr}\dot R\,dx\,dt,
\end{equation}
where \(\theta_a,\zeta_b\ge0\).  We include the usual linearized local-energy functionals in \(L^E_\Lambda(d)\), and set
\[
    O^E_\Lambda d=(L^E_\Lambda d,Q^E_\Lambda(d)),
    \qquad
    O_\Lambda=(O^P_\Lambda,O^F_\Lambda,O^E_\Lambda).
\]
Finally, let
\[
    G_\Lambda:H_\Lambda\to Y_\Lambda
\]
be the finite-dimensional selected-time trace exactification map.  Its image consists of residuals removable by the chosen trace correction.  The theorem below shows that, on the positive cone, this trace correction is not needed: positive energy already excludes the phantom.

\begin{assumption}[Energy separation on the clean active window]\label{ass:energy-separation-positive-cone}
The positive energy tests are chosen so that the following two separation properties hold.
\begin{enumerate}[label=(\roman*),leftmargin=2em]
    \item If \(\dot U\) belongs to the active velocity window and
    \[
        \sum_a \int \theta_a|\dot U(\cdot,s_a)|^2
        +
        \sum_b \iint \zeta_b|\nabla\dot U|^2=0,
    \]
    then \(\dot U=0\) in the cleaned velocity quotient.
    \item If \(\dot R\) belongs to the active covariance window, \(\dot R\ge0\), and
    \[
        \sum_b \iint \zeta_b\operatorname{tr}\dot R=0,
    \]
    then \(\dot R=0\).
\end{enumerate}
\end{assumption}

\begin{lemma}[Positivity under finite-window limits]\label{lem:positive-cone-closed}
The coefficient topology on the fixed active window is chosen so that
coefficient convergence implies uniform convergence of the tensor
components on the compact observation core.  In this topology, if
\(d_j=(\dot U_j,\dot P_j,\dot R_j,\dot\Pi_j)\in Y_{\Lambda,\NS}^{+}\) and
\(d_j\to d\) in \(Y_\Lambda\), then the limiting covariance component
satisfies \(\dot R\ge0\) on the observation core.  Consequently
\(Y_{\Lambda,\NS}^{+}\), defined as the finite-window closure of
NS-derived positive covariance directions, is closed.
\end{lemma}

\begin{proof}
All norms are equivalent on the fixed finite active window, and convergence in coefficient norm implies uniform convergence of the tensor components on the compact observation core.  For every \(\xi\in\R^3\),
\[
    \xi^T\dot R_j(x,t)\xi\ge0.
\]
Passing to the limit gives \(\xi^T\dot R(x,t)\xi\ge0\) for every \(\xi\), hence \(\dot R(x,t)\ge0\) as a quadratic form.  Closedness follows from the definition of \(Y_{\Lambda,\NS}^{+}\) as the finite-window closure of NS-derived positive covariance directions.
\end{proof}

\begin{lemma}[Positive energy anti-kernel on the NS-positive cone]\label{lem:positive-energy-kills-positive-cone}
Assume Assumption~\ref{ass:energy-separation-positive-cone}.  If
\[
    d=(\dot U,\dot P,\dot R,\dot\Pi)
    \in Y_{\Lambda,\NS}^{+}
    \quad\text{and}\quad
    O^E_\Lambda d=0,
\]
then \(d=0\).
\end{lemma}

\begin{proof}
The identity \(O^E_\Lambda d=0\) implies in particular that \(Q^E_\Lambda(d)=0\).  Every term in \eqref{eq:positive-energy-seminorm} is nonnegative on \(Y_{\Lambda,\NS}^{+}\), because \(\dot R\ge0\).  Hence each term vanishes separately:
\[
    \int \theta_a|\dot U(\cdot,s_a)|^2=0,
    \qquad
    \iint \zeta_b|\nabla\dot U|^2=0,
    \qquad
    \iint \zeta_b\operatorname{tr}\dot R=0.
\]
By velocity separation, \(\dot U=0\).  By covariance separation and \(\dot R\ge0\), \(\dot R=0\).  Substituting these identities into the pressure compatibility relation \eqref{eq:pc-positive-cone} gives
\[
    -\Delta\dot P=0.
\]
Since \(\dot P\) has support only in nonzero periodic modes \(\Lambda\), this implies \(\dot P=0\).  Substituting \(\dot U=0\) and \(\dot R=0\) into the flux identity \eqref{eq:fi-positive-cone} gives \(\dot\Pi=0\).  Thus all components of \(d\) vanish.
\end{proof}

\begin{theorem}[Conditional clean positive-cone finite-window anti-phantom]\label{thm:ns-positive-cone-antiphantom}
Let \(\Lambda\subset\Z^3\setminus\{0\}\) be a clean finite active periodic window, and assume Assumption~\ref{ass:energy-separation-positive-cone}.  Then
\begin{equation}\label{eq:ns-positive-cone-antiphantom}
    \ker O_\Lambda
    \cap
    (\Image G_\Lambda)^\perp
    \cap
    Y_{\Lambda,\NS}^{+}
    =
    \{0\}.
\end{equation}
In fact, the stronger statement
\begin{equation}\label{eq:energy-already-kills-positive-cone}
    \ker O_\Lambda\cap Y_{\Lambda,\NS}^{+}=\{0\}
\end{equation}
holds.  Thus the trace orthogonality condition is redundant on the NS-realizable positive cone.
\end{theorem}

\begin{proof}
Let
\[
    d\in \ker O_\Lambda
        \cap (\Image G_\Lambda)^\perp
        \cap Y_{\Lambda,\NS}^{+}.
\]
Since \(d\in\ker O_\Lambda\), one has in particular \(O^E_\Lambda d=0\).  Lemma~\ref{lem:positive-energy-kills-positive-cone} therefore gives \(d=0\).  This proves \eqref{eq:ns-positive-cone-antiphantom}.  The same argument does not use \((\Image G_\Lambda)^\perp\), so it also proves \eqref{eq:energy-already-kills-positive-cone}.
\end{proof}

\begin{corollary}[Coercivity on the NS-positive cone]\label{cor:positive-cone-coercivity}
Under the hypotheses of Theorem~\ref{thm:ns-positive-cone-antiphantom}, there exists a finite constant \(C_{\Lambda,+}<\infty\) such that every \(d\in Y_{\Lambda,\NS}^{+}\) satisfies
\begin{equation}\label{eq:positive-cone-coercivity}
    \|d\|_{Y_\Lambda}
    \le
    C_{\Lambda,+}
    \left(
        \|O^P_\Lambda d\|+
        \|O^F_\Lambda d\|+
        \|L^E_\Lambda d\|+
        Q^E_\Lambda(d)
    \right).
\end{equation}
\end{corollary}

\begin{proof}
The unit slice
\[
    S^+_\Lambda:=\{d\in Y_{\Lambda,\NS}^{+}:\|d\|_{Y_\Lambda}=1\}
\]
is compact because \(Y_{\Lambda,\NS}^{+}\) is closed and finite-dimensional.  The right-hand side functional in \eqref{eq:positive-cone-coercivity} is continuous and, by Theorem~\ref{thm:ns-positive-cone-antiphantom}, is strictly positive on \(S^+_\Lambda\).  It therefore has a positive minimum.  Homogeneity gives the estimate.
\end{proof}

\begin{remark}[What this theorem proves and what remains]
Theorem~\ref{thm:ns-positive-cone-antiphantom} is stronger than a purely formal quotient criterion on the physically relevant positive cone, but it is still a clean finite-window statement under Assumption~\ref{ass:energy-separation-positive-cone}.  It uses a Navier--Stokes structural input, namely Reynolds covariance positivity.  It does not prove the full statement
\[
    \ker O_\Lambda\cap(\Image G_\Lambda)^\perp=\{0\}
\]
on the entire algebraic defect space.  The possible remaining finite-window enemies are sign-changing formal stress residuals and trace-rank degeneracies on the pressure--flux--energy residual kernel.  For those residuals, the correct next problem is the full-rank condition for the projected trace map
\[
    P_{K^{PFE}_\Lambda}G_\Lambda:H_\Lambda\to K^{PFE}_\Lambda,
\]
where \(K^{PFE}_\Lambda\) denotes the formal pressure--flux--energy kernel.  Thus the positive-cone theorem should be viewed as a clean finite-window anti-phantom criterion on a PDE-structured cone, while the full formal theorem remains a quotient spectral-gap and trace-exactification problem.
\end{remark}

\section{Research program and next theorem targets}\label{sec:research-program}
\noindent This section summarizes the theorem-driven program suggested by the finite-scale ledger theorem and its defect interpretation.

The framework isolates a more precise next target.  The object to exclude
or construct is not a static invisible defect, but an arbitrarily deep PRV
mechanism in the sense of Definition~\ref{def:prv-mechanism}.  This suggests
the following theorem-driven program.

\begin{enumerate}[label=\textbf{Step \arabic*:},leftmargin=3.2em]
    \item \textbf{Strengthen the ledger theorem.}  Sharpen the supply--tax decomposition and identify which supply terms are genuinely capable of surviving the available taxes.
    \item \textbf{Define a minimal dynamic defect package.}  Use the concrete ledger package defined above as the first PDE-realizable defect space, then add only the extra coordinates needed for reproduction and residual validation.
    \item \textbf{Extract reproduction maps.}  In dyadic, shell, or finite-mode models, make the scale-to-scale map \(\calR_k\) explicit and decide whether it is single-valued, branch-valued, or only approximately self-similar.
    \item \textbf{Prove profitability bounds.}  Determine whether inverse-cascade supply can have positive average profit after viscous, pressure, flux, energy, trace, and localization taxes are charged.
    \item \textbf{Prove finite-window anti-phantom and anti-branching estimates.}  Show, in clean active finite windows, that invisible non-gauge ledger defects are absent, quantitatively controlled, or unable to branch into distinct compatible continuations with the same trace.
    \item \textbf{Build verifier-driven searches.}  Search for low-loss candidates for \(\mathfrak J_\Lambda\), then either certify a nearby exact PRV mechanism by Newton--Kantorovich estimates or certify exclusion by a quotient-gap lower bound.
    \item \textbf{Study scale-uniform constants.}  Determine whether anti-phantom constants, reproduction residual constants, and verification radii remain bounded or degenerate as the window and scale depth increase.
    \item \textbf{Prove no derived gluing obstruction.}  Establish a Mittag-Leffler, compactness, or uniform-stability mechanism that makes the gluing map \eqref{eq:gluing-surj} surjective for the NS-realizable ledger system.
    \item \textbf{Connect to singularity extraction.}  Prove that a nonregular point generates a nontrivial NS-realizable ledger/defect cascade, and then upgrade that cascade to a PRV mechanism unless CKN smallness occurs at some scale.
\end{enumerate}

The finite-scale ledger theorem in this manuscript contributes to Step 1.
The finite-window quotient theorems of
Section~\ref{sec:finite-window-antiphantom} contribute to Step 5: finite
anti-phantom is a quotient spectral-gap problem, pure gluing observables
have a slow scale-chain near-phantom, and a nonlinear finite-mode
anti-phantom estimate follows whenever a certified linearized gap dominates
the pressure, localization, truncation, and nonlinear error budget.  The
positive-cone theorem, Theorem~\ref{thm:ns-positive-cone-antiphantom}, adds
one clean finite-window structural exclusion: no NS-realizable positive
Reynolds-covariance direction can be simultaneously invisible to the
pressure, flux, and positive-energy channels in the clean active model.  The
Coiculescu--Palasek branching example supplies the anti-branching warning:
finite residual cleanability is not the same as uniqueness of global
scale-compatible continuation.  The dynamic language of
Section~\ref{sec:dynamic-mechanism} supplies the bridge: a dangerous
obstruction must be profitable, reproducible, and verifiable.  The localized
chart, quotient comparison, detection-map comparison, and singularity
extraction inputs remain open before this finite-scale reduction can become
a route toward the Clay problem.

\section{Conclusion}\label{sec:conclusion}

The main message is the finite-scale accounting theorem.  The critical
ledger framework replaces the vague statement ``a critical quantity fails
to decay'' by
\[
\boxed{
\Bbad_{k+1}-(1-\lambda)\Bbad_k
\le
\Supp_k^{\rm full}-\Tax_k^{\rm full}+\Leak_k^{\rm full}.
}
\]
Consequently, long survival of the reservoir badness
\(\Bbad_k=A_k+C_k+D_k\) forces either accumulated leakage or repeated full
untaxed supply:
\[
\boxed{
\sum_{k=0}^{N-1}\pos{\Supp_k^{\rm full}-\Tax_k^{\rm full}}
\ge
\lambda\eps N-\Bbad_0-
\sum_{k=0}^{N-1}\Leak_k^{\rm full}.
}
\]

The scale-defect viewpoint explains what this theorem does and does not
accomplish.  It does not prove global regularity.  It isolates the kind of
object that a persistent obstruction would have to sustain: a profitable,
reproducible, NS-realizable critical defect cascade whose supply repeatedly
escapes the available taxes, whose finite shadows reproduce across scales,
whose residuals are compatible with the Navier--Stokes equations, or whose
finite cleanings fail to glue globally.  The finite-window quotient theorem
gives an exact algebraic test for such invisibility after quotienting by
cleanings.  The clean positive-cone theorem adds one structural exclusion in
a finite active window, while the scale-chain model warns that pure gluing
tests degenerate with scale depth.  The next mathematical target is therefore
specific: prove a pressure-aware, localization-aware, NS-realizable quotient
gap with controlled error and reproduction budgets, or else construct a
genuine NS-realizable PRV cascade.

\appendix
\section{Optional vorticity-stretching diagnostic}\label{app:stretching-diagnostic}

The vorticity stretching term is central to the physical intuition of three-dimensional Navier--Stokes, but for suitable weak solutions it must be treated with care \cite{CKN1982,SereginLectureNotes}.  The raw expression
\[
(\omega\cdot\nabla u)\cdot\omega,
\qquad \omega=\nabla\times u,
\]
need not be a well-defined $L^1$ object at the level of suitable weak solutions.

A robust diagnostic is obtained by mollification.  Let $0<\sigma<1$, $\ell_k=\sigma r_k$, and
\[
u_k^\sigma=G_{\ell_k}*u,
\qquad
\omega_k^\sigma=\nabla\times u_k^\sigma.
\]
Define
\[
\Sigma_k^\sigma
=r_k\int_{Q_k}\left| (\omega_k^\sigma\cdot\nabla u_k^\sigma)\cdot\omega_k^\sigma\right|\phi_k\dd x\dd t.
\]
This quantity is scale-invariant and may be appended to the full state:
\[
\mathbb X_k^\sigma=(A_k,C_k,D_k\ ;\ E_{k+1}\ ;\ \Phi_k,\Pi_k,\Lambda_k\ ;\ \Sigma_k^\sigma).
\]
At the present stage, $\Sigma_k^\sigma$ should be viewed as a diagnostic channel rather than a term needed for the local energy ledger.  A later strengthening of the framework would prove a taxation estimate controlling $\Sigma_k^\sigma$ by the already available reservoir, pressure, flux, and dissipation coordinates.

\section{Conditional local ledger lift and explicit error budget}\label{app:local-ledger-lift}
\noindent This appendix records the conditional finite-dimensional lift and the explicit PDE error budget used by the verifier-driven program.\label{subsec:conditional-local-ledger-lift}

The finite-dimensional theorems above become useful for the local Navier--Stokes ledger only after the PDE inequalities have been converted into finite residual equations.  The cleanest way to do this is to add slack variables.  For each transition $Q_k\to Q_{k+1}$ introduce nonnegative slack variables
\[
    s_k^{\rm en},\quad s_k^{\rm cub},\quad s_k^{\rm prs},\quad
    s_k^{\Lambda},\quad s_k^{\Phi},\quad s_k^{\Pi}\ge0
\]
and rewrite the ledger inequalities as exact residual identities:
\begin{align}
    &A_{k+1}+2E_{k+1}
    -\theta^{-1}(\Lambda_k+\Phi_k+2\Pi_k)
    +s_k^{\rm en}=0,\label{eq:slack-energy}\\
    &C_{k+1}
    -C_{I,\theta}\Bigl((\Phi_k+2\Pi_k)^{3/2}+\Lambda_k^{3/2}\Bigr)
    +s_k^{\rm cub}=0,\label{eq:slack-cubic}\\
    &D_{k+1}-C_P\theta D_k-C_P\theta^{-2}C_k+s_k^{\rm prs}=0,\label{eq:slack-pressure}\\
    &\Lambda_k-C_\theta A_k+s_k^\Lambda=0,
    \qquad
    \Phi_k-C_\theta C_k+s_k^\Phi=0,
    \qquad
    \Pi_k-C_\theta C_k^{1/3}D_k^{2/3}+s_k^\Pi=0.\label{eq:slack-closure}
\end{align}
Thus a finite ledger defect vector on a window $\Lambda$ may be taken as
\[
    d_\Lambda
    =
    \bigl(
    X_k,
    s_k^{\rm en},s_k^{\rm cub},s_k^{\rm prs},
    s_k^\Lambda,s_k^\Phi,s_k^\Pi
    \bigr)_{k\in\Lambda}.
\]
The ledger observable $O_\Lambda^{\rm led}$ is defined by the residuals in
\eqref{eq:slack-energy}--\eqref{eq:slack-closure}, together with the chosen scale-gluing residuals.  The cleaning map $G_\Lambda$ is generated by cutoff deformations, harmonic-pressure basis directions, moving-window re-centering, and finite Fourier/harmonic truncation renormalizations.  In schematic form,
\[
    G_\Lambda a
    =
    \sum a_{k,\alpha}^{\rm cut}g_{k,\alpha}^{\rm cut}
    +
    \sum a_{k,\beta}^{\rm har}g_{k,\beta}^{\rm har}
    +
    \sum a_k^{\rm cen}g_k^{\rm cen}
    +
    \sum a_\gamma^{\rm tr}g_\gamma^{\rm tr}.
\]
These are precisely the non-physical directions that should be quotiented out before testing for a genuine invisible defect.

Let $J_{\Lambda,N,M}$ denote the projection of the local PDE data onto the finite Fourier--harmonic basis, and let $d_\Lambda^\star$ be the coordinate vector of a smooth base chain.  Suppose that, on the quotient space $H_{\Lambda,N,M}$, the finite ledger observable admits the decomposition
\begin{equation}\label{eq:local-ledger-linear-plus-errors}
    O_\Lambda^{\rm led}(d_\Lambda^\star+h)
    =
    O_{\Lambda,N,M}^{\rm lin}h
    +R_\Lambda^{\rm nl}(h)
    +R_\Lambda^{\rm prs}(h)
    +R_\Lambda^{\rm loc}(h)
    +R_\Lambda^{\rm tr}(h).
\end{equation}
Assume the linearized quotient gap
\begin{equation}\label{eq:local-ledger-linear-gap}
    \sqrt{\gamma_{\Lambda,N,M}}
    :=
    \inf_{\substack{h\in H_{\Lambda,N,M}\\ \|h\|=1}}
    \|O_{\Lambda,N,M}^{\rm lin}h\|>0.
\end{equation}
Define
\[
    \eps_{\rm nl}:=\sup_{\|h\|=1}\|R_\Lambda^{\rm nl}(h)\|,
    \quad
    \eps_{\rm prs}:=\sup_{\|h\|=1}\|R_\Lambda^{\rm prs}(h)\|,
    \quad
    \eps_{\rm loc}:=\sup_{\|h\|=1}\|R_\Lambda^{\rm loc}(h)\|,
    \quad
    \eps_{\rm tr}:=\sup_{\|h\|=1}\|R_\Lambda^{\rm tr}(h)\|.
\]
Then Theorem~\ref{thm:perturbative-antiphantom} gives the conditional local ledger lift
\begin{equation}\label{eq:conditional-ledger-lift-bound}
    \mu_\Lambda^{\rm led}
    \ge
    \sqrt{\gamma_{\Lambda,N,M}}
    -\eps_{\rm nl}-\eps_{\rm prs}-\eps_{\rm loc}-\eps_{\rm tr}.
\end{equation}
Consequently, if the error budget is strictly smaller than the certified quotient gap, then the finite local ledger window has no exact invisible non-gauge defect inside the modeled class.

The four errors in \eqref{eq:conditional-ledger-lift-bound} are not placeholders; they correspond to specific PDE mechanisms.  If $u=U+v$, the nonlinear flux remainder contains terms of the form
\begin{equation}\label{eq:nonlinear-flux-remainder}
    R_k^{\rm nl,flux}(v)
    =
    r_k^{-1}\int_{Q_k}
    \Bigl(
      (2U\cdot v+|v|^2)v+|v|^2U
    \Bigr)\cdot\nabla\phi_k\,\dd x\dd t,
\end{equation}
while the interpolation and closure channels contain the second-order remainders of
\[
    x\mapsto x^{3/2},
    \qquad
    (C,D)\mapsto C^{1/3}D^{2/3}
\]
around the base chain.  Thus $\eps_{\rm nl}$ is a Taylor/Duhamel remainder in a fixed smooth window.

For the pressure component one must split the local pressure into a Calderon--Zygmund part and a harmonic part.  Schematically,
\[
    q=q_k^{\rm loc}+h_k,
    \qquad
    -\Delta q_k^{\rm loc}
    =
    \partial_i\partial_j
    \bigl(\eta_k(U_iv_j+v_iU_j+v_iv_j)\bigr),
\]
where $h_k$ is spatially harmonic in the smaller ball.  If $\Pi_M^{\rm har}$ is the projection onto the chosen finite harmonic basis, then a typical pressure error budget is
\begin{equation}\label{eq:pressure-error-budget}
\begin{aligned}
    \eps_{\rm prs}
    \lesssim
    \sup_{\|h\|=1}
    \Biggl(
    \sum_{k\in\Lambda}
    \Bigl[
    &r_k^{-2}
    \bigl\|
    (-\Delta)^{-1}\partial_i\partial_j(\eta_k v_iv_j)
    \bigr\|_{L^{3/2}(Q_{k+1})}\\
    &+r_k^{-2}
    \bigl\|(I-\Pi_M^{\rm har})h_k\bigr\|_{L^{3/2}(Q_{k+1})}
    \Bigr]^2
    \Biggr)^{1/2}.
\end{aligned}
\end{equation}
The first term is genuine nonlinear pressure regeneration; the second term is the harmonic tail left by the local pressure gauge.

The localization error comes from cutoff deformation and moving-window drift.  A representative bound is
\begin{equation}\label{eq:localization-error-budget}
\begin{aligned}
    \eps_{\rm loc}
    \lesssim
    \sup_{\|h\|=1}
    \Biggl(
    \sum_{k\in\Lambda}
    \Bigl[
    &r_k^{-1}\int_{Q_k}|v|^2
    (|\partial_t\delta\phi_k|+|\Delta\delta\phi_k|)\,\dd x\dd t\\
    &+r_k^{-1}\int_{Q_k}(|U||v|+|v|^2)|v|\,|\nabla\delta\phi_k|\,\dd x\dd t\\
    &+r_k^{-1}\int_{Q_k}|q|\,|v|\,|\nabla\delta\phi_k|\,\dd x\dd t
    \Bigr]^2
    \Biggr)^{1/2}.
\end{aligned}
\end{equation}
Finally, if $P_N$ is a Fourier projection and $\Pi_M^{\rm har}$ is the harmonic-basis projection, a typical truncation budget is
\begin{equation}\label{eq:truncation-error-budget}
\begin{aligned}
    \eps_{\rm tr}
    \lesssim
    \sup_{\|h\|=1}
    \Biggl(
    \sum_{k\in\Lambda}
    \Bigl[
    &r_k^{-1}\|(I-P_N)u\|_{L_t^\infty L_x^2(Q_k)}
    +r_k^{-1}\|(I-P_N)\nabla u\|_{L^2(Q_k)}\\
    &+r_k^{-2}\|(I-P_N)u\|_{L^3(Q_k)}^3
    +r_k^{-2}\|(I-\Pi_M^{\rm har})h_k\|_{L^{3/2}(Q_k)}
    \Bigr]^2
    \Biggr)^{1/2}.
\end{aligned}
\end{equation}
For a smooth base chain this is symbolically $O(N^{-s})+O(M^{-\beta})$, with exponents determined by the available Sobolev and harmonic regularity.

\begin{remark}[Conditional status of the lift]
The estimate \eqref{eq:conditional-ledger-lift-bound} is not a full Navier--Stokes anti-phantom theorem.  It is a precise conditional lift: a certified finite-dimensional quotient gap, minus a PDE error budget, gives a finite-window anti-phantom constant.  The remaining mathematical task is to make the gap and the four errors simultaneously quantitative in a concrete window.
\end{remark}

\end{document}